\documentclass[12pt]{article}

\usepackage{wrapfig}
\usepackage{amssymb}
\usepackage{bm}
\usepackage{amsmath,amsthm}
\usepackage{a4wide}
\usepackage{parskip}
\usepackage{tikz}
\usetikzlibrary{3d}
\usepackage{mathtools}
\usepackage{cjhebrew}
\newcommand{\mem}{\textcjheb{m}}
\newcommand{\resh}{\textcjheb{r}}

\usetikzlibrary{arrows,arrows.meta,decorations.markings,matrix}
\usetikzlibrary{patterns}
\usetikzlibrary{math}
\usetikzlibrary{calc}
\usetikzlibrary{shapes.multipart}
\usetikzlibrary{decorations.pathmorphing,shapes}
\usetikzlibrary{decorations.pathreplacing}

\tikzset{->-/.style={decoration={
              markings,
              mark=at position .5 with {\arrow{>}}},postaction={decorate}}}

\tikzset{-<-/.style={decoration={
              markings,
              mark=at position .5 with {\arrow{<}}},postaction={decorate}}}

\usepackage{comment}

\usepackage{accents}
\newcommand{\dbtilde}[1]{\accentset{\approx}{#1}}

\theoremstyle{plain}
\newtheorem{theorem}{Theorem}[subsection]

\newtheorem{lemma}[theorem]{Lemma}
\newtheorem{proposition}[theorem]{Proposition}
\newtheorem{corollary}[theorem]{Corollary}

\theoremstyle{definition}
\newtheorem{definition}[theorem]{Definition}
\newtheorem{construction}[theorem]{Construction}
\newtheorem{setup}[theorem]{Setup}
\newtheorem{example}[theorem]{Example}
\newtheorem{discussion}[theorem]{Discussion}

\newtheorem{remark}[theorem]{Remark}

\newtheorem{background}[theorem]{Background}

\numberwithin{equation}{section}

\usepackage[a4paper, margin=1in]{geometry}
\makeatletter
\def\thm@space@setup{%
  \thm@preskip=\parskip \thm@postskip=0pt
}
\makeatother

\usepackage[scr=boondoxo]{mathalfa}

\newcommand{\Polytope}{P}
\newcommand{\Rb}{R_{\partial}}
\newcommand{\Sb}{S}
\newcommand{\blr}{\bm{\lambda}} 
\newcommand{\bmm}{\bm{\mu}} 
\newcommand{\lr}{\lambda} 
\newcommand{\mm}{\mu} 
\newcommand{\mg}{\mathscr{g}}
\newcommand{\mf}{\mathscr{f}}
\newcommand{\mproj}{\mathscr{p}}
\newcommand{\ironed}{\operatorname{ironed}}
\newcommand{\lcm}{\operatorname{lcm}}
\newcommand{\toric}{\operatorname{toric}}
\newcommand{\sX}{\mathfrak{X}}
\newcommand{\cX}{\mathcal{X}}
\newcommand{\Hess}{\operatorname{Hess}}
\newcommand{\tri}{\Lambda}
\newcommand{\trifil}{\Delta}

\newcommand{\mI}{\mathcal{I}}
\renewcommand{\epsilon}{\varepsilon}
\newcommand{\ups}{\upsilon}

\newcommand{\CC}{\mathbb{C}}
\newcommand{\mC}{\mathcal{C}}

\newcommand{\Ell}{\mathbb{E}}

\newcommand{\QQ}{\mathbb{Q}}

\newcommand{\mL}{\mathcal{L}}

\newcommand{\RR}{\mathbb{R}}

\newcommand{\PP}{\mathbb{P}}

\newcommand{\ZZ}{\mathbb{Z}}
\newcommand{\CP}{\mathbb{CP}}

\newcommand{\OP}[1]{\mathrm{#1}}
\newcommand{\til}{\widetilde}

\newcommand{\mX}{\mathcal{X}}
\newcommand{\mY}{\mathcal{Y}}
\newcommand{\mZ}{\mathcal{Z}}
\newcommand{\mE}{\mathcal{E}}

\newcommand{\edge}{\mathfrak{e}}
\newcommand{\pt}{\mathfrak{p}}
\newcommand{\qt}{\mathfrak{q}}

\newcommand{\exc}{\operatorname{exc}}

\title{Weighted Seshadri constants \\and ellipsoid embeddings}
\author{Jonny Evans\footnote{This work was supported by EPSRC
Standard Grant EP/W015749/1.}}

\newcommand{\Addresses}{{
  \bigskip
  \footnotesize
  J.~D.~Evans, \textsc{University of Lancaster}\par\nopagebreak
  \texttt{j.d.evans@lancaster.ac.uk}
}}
\begin{document}

\maketitle

\thanks{\begin{center}\it Dedicated to Kai Cieliebak, on his five-dozenth birthday.\end{center}}

\begin{abstract}
  We explore Seshadri constants associated to weighted blow-ups
  of complex projective varieties and demonstrate how to use this
  notion to construct symplectic embeddings of ellipsoids. We
  illustrate the utility of this point of view by providing
  constructions of full fillings of \(\CP^2\) by ellipsoids
  corresponding to all of the exceptional (post-Fibonacci) steps
  of the McDuff--Schlenk staircase and some non-obvious
  embeddings of ellipsoids in ellipsoids.
\end{abstract}

\section{Introduction}

\subsection{Seshadri constants and symplectic geometry}
Recall that:
\begin{itemize}
\item the {\em Gromov width} of a symplectic manifold
  \((Z,\omega)\) is the quantity \(\pi r^2\) where \(r\) is the
  supremum over all radii of symplectically embedded balls in
  \((Z,\omega)\);
\item the {\em Seshadri constant} of a projective variety \(Z\)
  at a point \(p\) with respect to an ample divisor \(D\) is
  defined to be the supremum
  \(\sup\{\varepsilon\,:\, g^*D-\varepsilon C\text{ is
    ample}\}\), where \(g\colon Y\to Z\) is the blow-up of \(Z\)
  at \(p\) and \(C=\exc(g)\) is the exceptional divisor.
\end{itemize}
It has been known since the work of McDuff and Polterovich
\cite{McDuffPolterovich} that there is a close connection
between the Gromov width of a K\"{a}hler manifold and its
Seshadri constants: the Seshadri constant gives a lower bound on
the Gromov width (see Lazarsfeld's book {\cite[Theorem
  5.1.22]{Lazarsfeld}}). One can generalise this to the problem
of embedding several symplectic balls pairwise disjointly, which
relates to the {\em multi-point} Seshadri constant. Biran
\cite{Biran4d,BiranStability} (for the case when
\([\omega]\in H^2(Z;\QQ)\)) and Buse, Hind and Opshtein
\cite{BHO} (in general) proved that one can fully fill any
symplectic 4-manifold by balls of equal size provided one has
enough balls. In his ECM survey \cite{BiranECM}, Biran pointed
out that one should be able to use {\em weighted blow-ups}
instead of ordinary blow-ups and thereby study symplectic
embeddings of {\em ellipsoids}.

\begin{definition}
  Given a vector \(\bm{\sigma}\) of positive real numbers, the symplectic ellipsoid
  \(\Ell(\bm{\sigma})\) is \begin{equation}\label{eq:ellipsoid}
    \Ell(\bm{\sigma}) = \left\{\bm{z}\in\CC^n\,:\,
      \sum_{j=1}^n\frac{|z_j|^2}{\sigma_j}\leq
      1\right\}\subset\CC^n
  \end{equation} where \(\CC^n\) is equipped with the standard
  symplectic form \(\sum_{k=1}^ndx_k\wedge dy_k\) and
  \(z_k=x_k+iy_k\). In the case when \(n=2\), we will refer to
  \(\sigma_2/\sigma_1\) as the {\em slope} of the ellipsoid. If
  \(\Ell(\bm{\sigma})\) embeds symplectically into a symplectic
  4-manifold \((Z,\omega)\) then we refer to the ratio
  \(\OP{vol}(\Ell(\bm{\sigma}))/\OP{vol}(Z)\) as the {\em packing
    ratio}.
\end{definition}

The connection between weighted blow-up and ellipsoid embeddings
was explored by McDuff \cite{McDuffEllipsoids}; McDuff showed
that the problem of embedding a symplectic ellipsoid into the
ball (or \(\CP^2\)) is equivalent to an associated ball packing
problem. In fact, similarly to Biran's results, it is known that
most symplectic manifolds admit full fillings by symplectic
ellipsoids of sufficiently large slope; see for example Buse and
Hind {\cite[Theorem 1.2]{BH}} (when the cohomology class of the
symplectic form is rational) and Edtmair {\cite[Theorem B]{Edt}}
(for 4-manifolds in general). For example, \(\CP^2\) can be
fully filled by any ellipsoid of slope at least
\(\frac{289}{36}\) (see {\cite[Corollary
  1.2.4]{McDuffSchlenk}}). The complete answer of which
ellipsoids embed symplectically into \(\CP^2\) was worked out by
McDuff and Schlenk in their seminal paper
\cite{McDuffSchlenk}. The size of the biggest ellipsoid
embedding into \(\CP^2\) depends strongly on the slope of the
ellipsoid: the set of possible pairs \((a,b)\in\RR^2\) for which
there exists an embedding of \(\Ell(a,b)\) forms a
staircase-like shape. This staircase has infinitely many
piecewise-linear steps; there is an infinite sequence of steps
whose coordinates are related to the Fibonacci numbers, and
which accumulate when the slope approaches the fourth power of
the Golden Ratio. There are also nine ``exceptional steps''
beyond this slope. For slopes outside these steps, \(\CP^2\) can
be fully filled (i.e. the packing ratio can be made arbitrarily
close to \(1\)).

Whilst much of the subsequent work on ellipsoid embeddings has used
McDuff's equivalence with ball-packings as a basis, the idea of
working directly with weighted blow-ups has been used to great effect
by Entov and Verbitsky \cite{EntovVerbitsky} to show that full
ellipsoid packings exist for {\em Campana simple manifolds}
(K\"{a}hler manifolds which are not unions of their proper
subvarieties). In the current paper, we adopt this point of view and
show that it can also be used to find all of the ellipsoid embeddings
in \(\CP^2\) corresponding to the post-Fibonacci steps in the
McDuff--Schlenk staircase.

\subsection{Inflation}
Biran's main tool for constructing multiple balls and McDuff's
main tool for constructing ellipsoids is {\em symplectic
  inflation}, where one modifies a symplectic form along some
curve in a 4-manifold; see for example {\cite[Proof of
  Proposition 2.1, Step 2]{McDuffEllipsoids}}. Inflation along
unicuspidal and sesquicuspidal plane curves is used by McDuff
and Siegel in {\cite[Theorem A, E]{McDuffSiegelInfiniteStairs}}
to construct the full ellipsoid embeddings in \(\CP^2\)
corresponding to the Fibonacci portion of the
staircase. Opshtein \cite{OpshteinEllipsoid} showed that
ellipsoids with slope \(p/q\) could be constructed by inflating
along a curve with ``multi-cusp'' singularities (points where
several irreducible branches each modelled analytically on
\(y^p=x^q\)). In each case, one needs to pay careful attention
to the local model for the inflationary curve.

Already in his ECM survey {\cite[Section 2.1]{BiranECM}}, Biran
pointed out that inflation is a symplectic analogue of using the
Nakai--Moishezon criterion to find an ample divisor. In the current
paper, we bypass inflation and just use Nakai--Moishezon, giving a
purely algebro-geometric construction of symplectic ellipsoid
embeddings.

\subsection{Construction of ellipsoids}
To a weighted blow-up \(g\colon Y\to Z\) (see Definition
\ref{dfn:weighted_blowup}) and an ample divisor \(D\) on \(Z\),
we will associate a {\em weighted Seshadri constant}
\(\varepsilon(Z,D;g)\) (see Definition
\ref{dfn:seshadri_constant}: this is the reciprocal of the
Cutkosky--Ein--Lazarsfeld \(s\)-invariant {\cite[Definition
  1.1]{CutkoskyEinLazarsfeld}} of a certain ideal). One can use
a lower bound on the weighted Seshadri constant to produce
ellipsoids of a certain size. This is made precise in Theorem
\ref{thm:ellipsoids_galore} below; essentially the same result
appears in the paper of Entov and Verbitsky
\cite{EntovVerbitsky} with a complete and detailed proof in the
masters thesis of Gudiev \cite{Gudiev}. It was also proved in a
different way by Luef and Wang {\cite[Theorem A]{LuefWang}}, who
also give an equivalent definition of weighted Seshadri
constants in terms of plurisubharmonic functions with specified
local behaviour {\cite[Definition 1.5 and Theorem
  3.2]{LuefWang}}. We include a slightly different\footnote{The
  idea in all the proofs is to produce K\"{a}hler forms which
  interpolate between standard forms on \(\CC^n\) and its
  weighted blow-up. We use Guillemin--Abreu theory to reduce
  this to interpolation between convex functions on the moment
  polytope; both Entov--Verbitsky/Gudiev and Luef--Wang use
  Demailly's regularised maximum of plurisubharmonic functions.}
proof here for completeness.

\begin{theorem}\label{thm:ellipsoids_galore}
  Let \(Z\) be a smooth complex projective variety, \(p\in Z\) be a
  point and \(D\subset Z\) be an irreducible ample \(\QQ\)-divisor;
  let \(\zeta\) be the associated symplectic form Poincar\'{e}-dual to
  \(\pi D\) (see Remark \ref{rmk:ample}). Let \(g\colon Y\to Z\) be a
  weighted blow-up with weight vector
  \(\bm{a}=(a_1,\ldots,a_n)\). Then \((Z,\zeta)\) admits a symplectic
  embedding of the ellipsoid
  \(\Ell\left(\frac{\varepsilon}{a_1},\ldots,\frac{\varepsilon}{a_n}\right)\)
  for any rational \(\varepsilon<\varepsilon(Z,D;g)\).
\end{theorem}

To apply the theorem, you need a suitable divisor \(D\), which
can be hard to find; this plays the role of the inflationary
curve. We will illustrate the use of this theorem by
constructing all of the exceptional steps in the McDuff--Schlenk
staircase and constructing some new ellipsoid embeddings into
ellipsoids. See Section \ref{sct:examples} for a wide array of
applications.

\subsection{Advantages}
\label{sct:advantages}
There are several advantages to working in the category of
complex projective varieties. First, the weighted blow-up is
almost always singular, and the intersection theory of curves on
singular projective varieties is well-developed and convenient to
use. One could further resolve and work with a smooth manifold,
but this ends up clouding the simplicity of computations with
additional irrelevant terms. Second, one doesn't need to make
any assumptions on what the inflationary divisors look like, or
how they intersect the exceptional divisor of the weighted
blow-up: one simply applies the Nakai--Moishezon criterion to
find an ample divisor and hence a symplectic form. Finally, we
observe that the original construction of the exceptional steps
by McDuff and Schlenk was quite involved: they had to enumerate
{\em all} possible obstructive curves and find the strongest
obstructions in each given interval of slopes. By contrast, as
soon as we find a curve \(D\) to which we can apply our
construction, we find an ellipsoid.

\subsection{Disadvantages and difficulties}
\label{sct:disadvantages}
There are also several disadvantages. The main disadvantage is
that the method is wholly unsuitable for constructing ellipsoids
whose slopes lie outside the steps of the McDuff--Schlenk
staircase, that is when there is a one-parameter family of
slopes which yield full fillings. To construct something
arbitrarily close to the volume constraint one would need to use
a certain infinite sequence of divisors \(D\) and weighted
blow-ups adapted to them, so that the associated weighted
Seshadri constant gets closer and closer to a limiting value but
never quite achieves it; see Example \ref{exm:slope_10}. It
seems likely that this infinite sequence of divisors is
precisely what one needs to prove the Nagata conjecture.

When working with weighted blow-ups there are also subtleties:
there is not a single operation which, given a point and a
slope, produces a canonical weighted blow-up. Indeed, one can
construct weighted blow-ups by iteratedly blowing up at
infinitely-near points and then contracting all but the last
exceptional curve; the choices of which points one blows up give
moduli for the construction; we explore this subtlety in Section
\ref{sct:farey}. Rather than being a difficulty, this is really an
advantage: we can make a ``general'' choice of points to
blow up, which can be helpful for constructing the right divisor
\(D\).

\subsection{Remarks on the proof of Theorem \ref{thm:ellipsoids_galore}}

Theorem \ref{thm:ellipsoids_galore} depends on some intermediate
results (Proposition \ref{prp:ellipsoids_1} and Corollary
\ref{cor:ellipsoids_2}) which are weighted analogues of a result of
McDuff and Polterovich's {\cite[Corollary 2.1.D]{McDuffPolterovich}}
for ordinary blow-ups. The strategy of our proof is very similar, but
there are subtleties in modifying the proof to handle weighted
blow-ups and ellipsoids:
\begin{itemize}
\item[1.] The McDuff--Polterovich proof uses K\"{a}hler forms
  constructed by pulling back the standard K\"{a}hler form on
  \((\CC^\times)^n\) via a map of the form
  \(\bm{z}\mapsto\frac{h(|\bm{z}|)}{|\bm{z}|}\bm{z}\) for some
  strictly increasing function \(h\). These maps are called
  {\em monotone embeddings}. If one uses
  \(\sqrt{\sum_{j=1}^na_j|z_j|^2}\) instead of \(|\bm{z}|\)
  then the pulled back form is not K\"{a}hler (it has
  \((2,0)\) and \((0,2)\)-parts), and it would be harder to
  check tameness. On a related note, the Hamiltonian circle
  action generated by
  \(H(\bm{z})=\frac{1}{2}\sum_{j=1}^na_j|z_j|^2\) extends to a
  holomorphic \(\CC^\times\)-action with weights
  \(a_1,\ldots,a_n\), but this does not carry level sets of
  \(H\) (ellipsoids) to level sets of \(H\) unless
  \(\bm{a}=(1,\ldots,1)\).
\item[2.] The weighted blow-up of \(\CC^n\) sits inside
  \(\CC^n\times\PP(\bm{a})\) where \(\PP(\bm{a})\) is a weighted
  projective space. The ``standard'' symplectic form on the ordinary
  blow-up can be written as a linear combination of two contributions:
  the pullback of the Fubini--Study form from the projection to
  \(\PP(1,\ldots,1)\), and the pullback of the standard symplectic
  form on \(\CC^n\). However, the symplectic form on the weighted
  blow-up cannot be written in this form: the inclusion of the
  weighted blow-up into the product has vanishing derivatives along
  the exceptional locus, so the pullback of the standard symplectic
  form has degeneracies along the normal directions to the exceptional
  locus of the weighted blow-up. Note that this is {\em in addition}
  to the complication that the weighted blow-up is singular.
\end{itemize}
To get around both of these issues, we use the
Guillemin--Abreu theory of toric K\"{a}hler metrics and
symplectic potentials on orbifolds to find ``standard''
symplectic forms and to interpolate between standard forms on
the weighted blow-up and weighted blow-down.

\subsection{Examples}

In Section \ref{sct:examples_nodal}, we use the nodal cubic curve in
\(\CP^2\) as our divisor to construct embeddings of
\(\Ell(\sigma_1,\sigma_2)\) into \(\CP^2\) with:
\begin{itemize}
\item slope \(s\coloneqq \sigma_2/\sigma_1\in
  \left(\frac{7+\sqrt{45}}{2},7\right]\)
    and packing ratios up to \(\frac{9s}{(1+s)^2}\), and
  \item slope \(s\coloneqq \sigma_2/\sigma_1\in
  \left(7,\frac{64}{9}\right]\)
    and packing ratios up to \(\frac{64s}{9}\).
\end{itemize}
These represent the optimal ellipsoid embeddings for the first
post-Fibonacci step in the McDuff--Schlenk staircase. In Section
\ref{sct:examples_beyond}, we use more complicated curves of higher
degree to construct the remaining post-Fibonacci steps; these curves
appear in the work of Dumnicki, Harbourne, K\"{u}ronya, Ro\'{e} and
Szemberg \cite{DumnickiEtAl} as ``supraminimal curves''. In that
context these curves are responsible for steps in a different but
related staircase: the graph of the function \(\hat{\mu}\) which is
another way of getting at what we are calling the weighted Seshadri
constant. Indeed, this paper grew out of my attempts to understand the
relationship between their work and that of McDuff--Schlenk
\cite{McDuffSchlenk}.

In Section \ref{sct:examples_unicuspidal}, we explain how to reproduce
the results of McDuff and Siegel \cite{McDuffSiegelInfiniteStairs} on ellipsoid
embeddings via inflation along unicuspidal curves in our language,
which in particular recovers the Fibonacci part of the staircase. In
Section \ref{sct:examples_outside}, we discuss the difficulties of
using our method to produce embeddings when the staircase is not
piecewise-linear. Finally, in Section \ref{sct:examples_ellipsoids},
we give some examples of ellipsoid embeddings in other ellipsoids.

\subsection{Outline}

After recapping some prerequisites and fixing conventions
(Section \ref{sct:preliminaries}), we define what we mean by
weighted blow-up and the associated Seshadri constants (Section
\ref{sct:weighted_seshadri}). Then there follow two sections in
which we explain how to iron out a K\"{a}hler form along the
exceptional locus of a weighted blow-up (Section
\ref{sct:ironing}) and how to extend an ironed form over the
blow-down (Section \ref{sct:kaehler}). Subsection
\ref{sct:kaehler_proof} contains the proof of Theorem
\ref{thm:ellipsoids_galore}. Section \ref{sct:surfaces} works
out in detail how to calculate the weighted Seshadri constant in
the case of complex surfaces and Section \ref{sct:examples}
applies this to a range of examples.

\subsection{Acknowledgements}

This paper grew out of my attempts to process talks and
conversations at the 2025 Workshops on Singular Algebraic Curves
and Quantitative Symplectic Embeddings in Geneva and Les
Mar\'{e}cottes. I would like to thank the organisers (Grisha
Mikhalkin, Quim Ro\'{e}, Felix Schlenk and Kyler Siegel) for
arranging this highly stimulating workshop, and to thank all the
speakers and participants for their inspiring talks and
penetrating observations. I would also like to thank Nikolas
Adaloglou, Jo\'{e} Brendel, Johannes Hauber, Dusa McDuff, Leonid
Polterovich, Quim Ro\'{e}, Felix Schlenk, Joel Schmitz, Kyler
Siegel and Giancarlo Urz\'{u}a for their subsequent interest in
and helpful discussions on this topic, and their comments on
this paper, and especially Felix Schlenk and Johannes Hauber for
a detailed reading, finding many typos. I would also like to
thank Misha Entov for telling me about the thesis of Mark
Gudiev, which also led me to his papers with Verbitsky, and Xu
Wang for informing me about his paper with Franz Luef. My
research is supported by EPSRC Standard Grant EP/W015749/1. For
the purpose of open access, the author has applied a Creative
Commons Attribution (CC BY) licence to any Author Accepted
Manuscript version arising.

\section{Preliminaries}
\label{sct:preliminaries}
\subsection{Notation, normalisations and conventions}

\begin{remark}[Broadcasting]
  Throughout this paper, we will write expressions like \(\bmm^{-1}\),
  \(\log(\bmm)\), or \(\frac{\partial G}{\partial\bmm}\) where
  \(\bmm=(\mm_1,\ldots,\mm_n)\) is a vector to mean the vector
  obtained by applying the relevant operation elementwise to \(\bmm\)
  to obtain a new vector like \((\mm_1^{-1},\ldots,\mm_n^{-1})\),
  \((\log(\mm_1),\ldots,\log(\mm_n))\), or \((\partial G/\partial
  \mm_1,\ldots,\partial G/\partial \mm_n)\). This practice, known as
     {\em broadcasting}, will simplify notation.
\end{remark}

\begin{remark}[Moment maps]
  The standard moment map \(\CC^n\to\RR^n\) sends \(\bm{z}\) to
  \[\left(\frac{1}{2}|z_1|^2,\ldots,\frac{1}{2}|z_n|^2\right).\]
  The moment image of \(\CC^n\) is the positive orthant,
  \(\RR_{\geq 0}^n\). The moment image of the ellipsoid
  \(\Ell(\bm{\sigma})\) is the subset
  \[\left\{\bm{\mu}\in\RR_{\geq 0}^n \,:\,
      \sum_{j=1}^n\sigma_i^{-1}\mu_i \leq \frac{1}{2}\right\}\]
  This has a slanted face with outward-pointing normal
  \(\bm{\sigma}^{-1}=(\sigma_1^{-1},\ldots,\sigma_n^{-1})\). In
  this paper, we will assume that a positive multiple of this
  normal, say \(R\bm{\sigma}^{-1}\), is a primitive integer
  vector \(\bm{a}\):
  \[\bm{\sigma}^{-1}=R^{-1}\bm{a}.\] It
  will often be more convenient to work with this normal, in which
  case we will write \(\Ell(R\bm{a}^{-1})\). In terms of \(R\) and
  \(\bm{a}\), the moment image of \(\Ell(R\bm{a}^{-1})\) and its
  slanted face are:
  \[\trifil_{\bm{a},\,R}\coloneqq \left\{\bm{\mu}\in\RR_{\geq 0}^n\,:\,
      \bm{a}\cdot\bm{\mu}\leq \frac{R}{2}\right\},\qquad
    \tri_{\bm{a},\,R}\coloneqq \left\{\bm{\mu}\in\RR_{\geq 0}^n\,:\,
      \bm{a}\cdot\bm{\mu}=\frac{R}{2}\right\}\]
\end{remark}

\begin{remark}[Symplectic area and affine length]
  Recall that the preimage of an edge of a moment polytope under
  a moment map is a symplectic sphere of symplectic area
  \(2\pi\ell\) where \(\ell\) is the affine length of the
  edge. We will normalise the Fubini--Study form on \(\CP^n\) as
  the symplectic reduction of the unit sphere in \(\CC^{n+1}\),
  so that it gives a line area \(\pi\): with this normalisation,
  the moment image is
  \(\tri_{(1,\ldots,1),\,1}\subseteq\RR^{n+1}\), a \(n\)-simplex
  whose edges have affine length \(1/2\).
\end{remark}

\begin{remark}[Symplectic forms from ample divisors]\label{rmk:ample}
  Recall that if \(Z\) is a complex projective variety and \(D\)
  is an ample \(\QQ\)-Cartier \(\QQ\)-divisor on \(Z\) then we
  obtain a K\"{a}hler form on the smooth locus of \(Z\) as
  follows. Let \(N>0\) be an integer such that \(ND\) is a very
  ample Cartier divisor and let \(L\) be the associated
  holomorphic line bundle. Since \(L\) is very ample, evaluation
  of sections defines an embedding of \(Z\) into
  \(\PP(H^0(Z;L)^\vee)\). Pulling back the Fubini--Study form
  along this embedding and then rescaling it by \(1/N\) yields a
  K\"{a}hler form whose cohomology class is Poincar\'{e}-dual to
  \(\pi D\).
\end{remark}

\begin{remark}[Orbifold Kodaira embedding]\label{rmk:orbifold_kodaira_embedding}
  We refer the reader to the lucid account of Ross and Thomas
  \cite{RossThomas} for a helpful user's guide to orbifolds. If
  \(Z\) is an orbifold then one can ask whether a K\"{a}hler
  form on the smooth locus extends to a K\"{a}hler orbifold
  differential form over the orbifold locus. Recall that an
  orbifold differential form is the same thing as a coherent
  choice of \(\Gamma\)-invariant differential form on each local
  uniformising chart \(\CC^n\to\CC^n/\Gamma\). Note that the
  K\"{a}hler form constructed in Remark \ref{rmk:ample} is
  degenerate along the orbifold locus when pulled back along a
  local uniformising chart \(\CC^n\to\CC^n/\Gamma\). One can fix
  this by working instead with an orbi-ample orbibundle \(L\)
  {\cite[Section 2.4 and Definition 2.7]{RossThomas}}. An
  orbibundle is locally just an equivariant line bundle on each
  orbifold chart; in particular, the stabiliser subgroup at a
  point \(z\in Z\) acts on the line \(L_z\). If we define the
  order \(\OP{ord}(Z)\) to be the least common multiple of the
  sizes of point-stabilisers then \(L^{\otimes\OP{ord}(Z)}\) is
  an honest line bundle. Ross and Thomas define an orbibundle to
  be {\em locally ample} if this action is faithful for all
  \(z\in Z\), and {\em orbi-ample} if it is locally ample and
  \(L^{\otimes\OP{ord}(Z)}\) is ample. Given an orbi-ample line
  bundle, one can construct an orbifold embedding
  \(Z\to \PP(\bm{n})\) into a weighted projective space
  \(\PP(\bm{n})\) built out of sections of powers of \(L\) (see
  {\cite[Proposition 2.11]{RossThomas}}) and the pullback of the
  Fubini--Study orbifold K\"{a}hler form {\cite[Definition
    3.3]{RossThomas}} yields an orbifold K\"{a}hler form on
  \(Z\).
\end{remark}

\begin{background}[Intersection numbers]
  We will frequently work with intersection numbers between divisors on
  singular complex projective surfaces; see {\cite[1.1.C]{Lazarsfeld}}
  for an explanation of how to make sense of this. If \(f\colon X\to Y\)
  is a morphism of schemes then we write \(f^*B\) for the total
  transform of a divisor \(B\subset Y\) and \(f^{-1}_*B\) for its proper
  transform. We will make use of the push-pull formula: if \(f\colon
  X\to Y\) is a proper morphism of surfaces and \(A\subset X\) and
  \(B\subset Y\) are curves then
  \begin{equation}\label{eq:push_pull}
    f_*(A\cdot f^*B)=f_*A\cdot B.
  \end{equation} See \cite[Lemma 82.19.4, Tag 0EQT]{Stacks} for
  a more general statement at the level of Chow groups. For example,
  this tells us that:
  \begin{itemize}
  \item \(f^{-1}_*B\cdot f^*B = B^2\) since
    \(f_*f^{-1}_*B=B\);
  \item if \(A\) is contracted by \(f\) then \(A\cdot f^*B=0\)
    since \(f_*A=0\).
  \end{itemize}
\end{background}

\section{Weighted Seshadri constants}
\label{sct:weighted_seshadri}
\subsection{Seshadri constants}

\begin{setup}\label{pg:setup}
  Let \(Y\) and \(Z\) be normal complex projective varieties and
  let \(g\colon Y\to Z\) be a projective birational morphism
  with irreducible exceptional locus \(C=\exc(g)\). Let \(D\) be
  an ample \(\QQ\)-divisor on \(Z\). Note that, by
  {\cite[Chapter II, Theorem 7.17]{Hartshorne}}, \(g\) is the
  blow-up of \(Z\) along a coherent ideal sheaf \(\mI\).
\end{setup}

\begin{definition}\label{dfn:seshadri_constant}
  In the context of Setup \ref{pg:setup}, let
  \(\Upsilon_\varepsilon\coloneqq g^*D-\varepsilon C\) and define the
          {\em Seshadri constant of \((Z,D)\) along \(g\)} to be
  \[\varepsilon(Z,D;g)\coloneqq \sup\left\{\varepsilon\geq
      0\,:\,\Upsilon_\varepsilon\text{ is ample}\right\}.\] This
  Seshadri constant is precisely the reciprocal of the
  Cutkosky--Ein--Lazarsfeld \(s\)-invariant of  \(\mI\), see
  {\cite[Definition 1.1]{CutkoskyEinLazarsfeld}}. 
\end{definition}

\begin{remark}\label{rmk:nef}
  Recall that a divisor \(\Upsilon\) is called {\em nef} if
  \(\Upsilon\cdot A\geq 0\) for all irreducible curves
  \(A\). One can equivalently define
  \[\varepsilon(Z,D;g)\coloneqq \max\left\{\varepsilon\geq
      0\,:\,\Upsilon_\varepsilon\text{ is nef}\right\}.\] To see
  that this is equivalent to the definition using ampleness
  observe that \(g^*D-\epsilon C\) is ample for sufficiently
  small \(\epsilon\) by {\cite[Chapter II, Proposition
    7.10(b)]{Hartshorne}}). By Kleiman's theorem {\cite[Theorem
    1.4.23]{Lazarsfeld}} the nef cone of a projective scheme is
  the closure of its ample cone, so the path
  \(\epsilon\mapsto g^*D-\epsilon C\) leaves the ample cone at
  the last instant it is nef.
\end{remark}

\subsection{Weighted blow-ups}

In this paper, we will focus on a specific class of ideal
sheaves defining non-reduced points; we will call the
corresponding blow-ups {\em weighted blow-ups}.

\begin{definition}
  Consider the cone \(\sigma\) spanned by the standard basis
  vectors \(\bm{e}_1,\ldots,\bm{e}_n\). The associated affine
  toric variety is \(\mZ=\CC^n\). Fix a primitive integer vector
  \(\bm{a}\in\ZZ^n\) with positive entries. Let \(\rho\) be the
  ray pointing in the \(\bm{a}\)-direction, and, for each
  \(j=1,\ldots,n\), let \(\sigma_j\) be the cone spanned by
  \(\bm{a}\) and all the basis vectors \(\bm{e}_k\) except
  \(\bm{e}_j\). The cones \(\sigma_1,\ldots,\sigma_n\) now form
  a fan \(\Sigma_{\mY}\) subdividing \(\sigma\). See Figure
  \ref{fig:toric_picture} for illustrations when \(n=2,3\).
\end{definition}

\begin{remark}\label{rmk:polytope}
  The fan \(\Sigma_{\mY}\) is the inward normal fan to the
  polytope
  \[P_\mY(R) \coloneqq \{\bm{\mu}\in\RR_{\geq 0}^n \,:\,
    \bm{a}\cdot\bm{\mu}\geq R/2\}\] for any \(R>0\). This
  polytope is obtained from the positive orthant
  \(\RR^n_{\geq 0}\) by chopping off a corner to obtain a new
  face \(\tri_{\bm{a},\,R}\) with inward normal \(\bm{a}\). See
  Figure \ref{fig:toric_picture} for illustrations when
  \(n=2,3\).
\end{remark}

\begin{definition}
  Let \(\mY\) be the normal toric variety associated to
  \(\Sigma_{\mY}\) and let \(\mg\colon \mY\to \mZ\) be the toric
  morphism coming from the inclusion of the cones
  \(\sigma_j\subseteq \sigma\). The variety \(\mY\) is a complex
  orbifold, with singularities at the toric fixed points, and
  the exceptional locus of \(\mg\) is the toric divisor \(\mC\)
  associated to the ray \(\rho\). We call \(\mg\) the {\em toric
    weighted blow-up with weights \(\bm{a}\)}.
\end{definition}

\begin{figure}[htb]
  \centering
  \begin{tikzpicture}
    \node at (-2,1) {Fans \(\Sigma_{\mY}\):};
    \node at (1,3) {\(n=2\)};
    \draw[->] (0,0) -- (1,0) node [below] {\(\bm{e}_1\)};
    \draw (1,0) -- (2,0);
    \draw[->] (0,0) -- (0,1) node [left] {\(\bm{e}_2\)};
    \draw (0,1) -- (0,2);
    \draw[very thick,->] (0,0) -- (1,1.5) node [below right] {\(\bm{a}\)};
    \draw (1,1.5) -- (4/3,2);
    \node at (1.2,0.6) {\(\sigma_2\)};
    \node at (0.5,1.3) {\(\sigma_1\)};
    \begin{scope}[shift={(0,-3)}]
      \node at (-2,1) {Polytopes \(P_{\mY}\):};
      \draw (0,2) -- (0,1) -- (1.5,0) -- (2,0);
      \draw[very thick,->] (0.75,0.5) -- ++ (1,1.5) node [below right] {\(\bm{a}\)};
      \node at (0.2,0.2) {\(\tri_{\bm{a},\,R}\)};
    \end{scope}
  \end{tikzpicture}\qquad
  \begin{tikzpicture}[z={(90:15mm)},x={(-135:15mm)},y={(-10:15mm)}]
    \node at (0,0,1.5) {\(n=3\)};
    \draw[->] (0,0,0) -- (1,0,0) node (e_1) {};
    \node at (e_1) [left] {\(\bm{e}_1\)};
    \draw[->] (0,0,0) -- (0,1,0) node (e_2) {};
    \node at (e_2) [right] {\(\bm{e}_2\)};
    \draw[->] (0,0,0) -- (0,0,1) node (e_3) {};
    \node at (e_3) [left] {\(\bm{e}_3\)};
    \node (a) at (0.8,1,0.9) {};
    \filldraw[lightgray,opacity=0.5,draw=black] (0,0,0) -- (a.center) -- (e_1.center) -- cycle;
    \filldraw[lightgray,opacity=0.5,draw=black] (0,0,0) -- (a.center) -- (e_2.center) -- cycle;
    \filldraw[lightgray,opacity=0.5,draw=black] (0,0,0) -- (a.center) -- (e_3.center) -- cycle;
    \filldraw[lightgray,opacity=0.5,draw=black] (0,0,0) -- (e_2.center) -- (e_1.center) -- cycle;
    \filldraw[lightgray,opacity=0.5,draw=black] (0,0,0) -- (e_3.center) -- (e_1.center) -- cycle;
    \filldraw[lightgray,opacity=0.5,draw=black] (0,0,0) -- (e_2.center) -- (e_3.center) -- cycle;
    \node (sigma_2) at (1,0,0.7) {\(\sigma_2\)};
    \draw[->] (sigma_2) -- (0.5,0.2,0.5);
    \node (sigma_3) at (1,0.7,0) {\(\sigma_3\)};
    \draw[->] (sigma_3) -- (0.5,0.5,0.2);
    \node (sigma_1) at (0,0.7,1) {\(\sigma_1\)};
    \draw[->] (sigma_1) -- (0.2,0.6,0.5);
    \draw[very thick,->] (0,0,0) -- (a.center);
    \node at (a) [right] {\(\bm{a}\)};
    \begin{scope}[shift={(0,0,-2)}]
      \draw (0,0,1) -- (0,0,3/9) -- (3/8,0,0) -- (1,0,0);
      \draw (0,0,1) -- (0,0,3/9) -- (0,3/10,0) -- (0,1,0);
      \draw (0,1,0) -- (0,3/10,0) -- (3/8,0,0) -- (1,0,0);
      \node at (0.6,0,0.6) {\(\tri_{\bm{a},\,R}\)};
      \begin{scope}[shift={(3*8/245,3*10/245,3*9/245)}]
        \draw[very thick,->] (0,0,0) -- (0.8,1,0.9) node (a2) {};
        \node at (0.8,1.1,0.8) {\(\bm{a}\)};
      \end{scope}
    \end{scope}
  \end{tikzpicture}
  \caption{The fans \(\Sigma_{\mY}\) (above) and dual polytopes
    \(P_{\mY}\) (below) for the toric model of the weighted
    blow-up for \(n=2\) (left) and \(n=3\)
    (right).}\label{fig:toric_picture}
\end{figure}
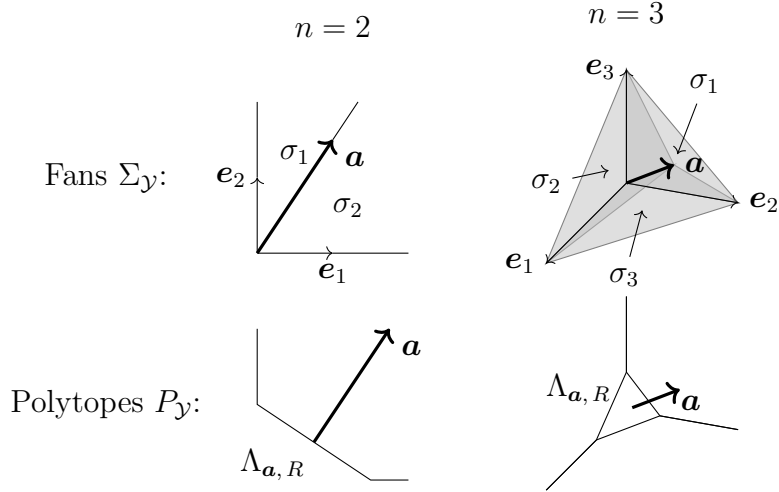

A weighted blow-up will be a blow-up locally analytically
modelled on \(\mg\colon\mY\to\mZ\). To make this precise, let us
define a family of subsets of \(\mZ\) and \(\mY\); at this
point, these subsets could equally be balls or polydiscs rather
than ellipsoids, but it will be more convenient to use
ellipsoids later when we come to equip our varieties with
K\"{a}hler forms.

\begin{definition}
  Given \(\bm{a}\), define the function \(H\colon\mZ\to\RR\) by
  \(H(\bm{z})=\sum_{j=1}^na_j|z_j|^2\) whose sublevel sets
  \(H^{-1}([0,R])\) are the ellipsoids
  \(\mZ(R)\coloneqq \Ell(R\bm{a}^{-1})\). Let
  \(\mY(R)=\mg^{-1}(\mZ(R))\).
\end{definition}

\begin{definition}\label{dfn:weighted_blowup}
  Let \(Z\) be a smooth complex projective variety and
  \(p\in Z\) be a point. Let \(\mI\) be an ideal sheaf supported
  at \(p\) and let \(g\colon Y=\OP{Bl}_{\mI}(Z)\to Z\) be the
  corresponding blow-up. We say that \(g\) is a {\em weighted
    blow-up} with weights \(\bm{a}\) if there exists a
  commutative diagram of analytic morphisms

  \begin{center}
    \begin{tikzpicture}
      \node (V) at (0,0) {\(\mY(R)\)};
      \node (W) at (0,-2) {\(\mZ(R)\)};
      \node (Y) at (2,0) {\(Y\)};
      \node (Z) at (2,-2) {\(Z\)};
      \draw[->] (V) -- (Y) node [midway,above] {\(\jmath\)};
      \draw[->] (V) -- (W) node [midway,left] {\(\mg\)};
      \draw[->] (W) -- (Z) node [midway,above] {\(\iota\)};
      \draw[->] (Y) -- (Z) node [midway,right] {\(g\)};
    \end{tikzpicture}
  \end{center}

  where \(\iota\) and \(\jmath\) are embeddings and \(\mg\) is
  the toric weighted blow-up with weights \(\bm{a}\).
\end{definition}

\begin{remark}
  Weighted blow-ups are more subtle than the ordinary blow-up,
  which is the special case \(\bm{a}=(1,\ldots,1)\). For
  example:
  \begin{itemize}
  \item the ordinary blow-up depends on the holomorphic
    embedding \(\iota\) only through the value \(\iota(0)\),
    whereas more general weighted blow-ups depend heavily on
    this choice.
  \item the result of weighted blow-up with any weights other
    than \((1,\ldots,1)\) is always singular.
  \end{itemize}
\end{remark}

\begin{remark}[Alternative descriptions]
  Using Cox's homogeneous coordinate ring description of the toric
  variety associated to the fan \(\Sigma\), we see that \(\mY\) is
  the GIT quotient of \(\CC^{n+1}\) (with coordinates
  \(z_0,z_1,\ldots,z_n\)) by the \(\CC^\times\)-action with
  weights \(-1,a_1,\ldots,a_n\), that is \(t\in\CC^\times\)
  acts as
  \[(t^{-1}z_0,t^{a_1}z_1,\ldots,t^{a_n}z_n).\]
  The unstable locus is \(z_1=\cdots=z_n=0\), so
  \[\mY = \left(\CC\times(\CC^n\setminus 0)\right)/\CC^\times.\]
  The morphism \(\mg\colon\mY\to \mZ\)
  is
  \[\mg([z_0:z_1:\cdots:z_n]) =
    (z_0^{a_1}z_1,\ldots,z_0^{a_n}z_n)\] and we see that
  \(\mg^{-1}(\bm{0})=\{z_0=0\}\) is a copy of the weighted
  projective space \(\PP(\bm{a})\). There is also a well-defined
  projection \(\mproj\colon\mY\to\PP(\bm{a})\) defined by
  \[\mproj([z_0\,:\,z_1\,:\,\cdots\,:\,z_n])=
    [z_1\,:\,\cdots\,:\, z_n]\in\PP(\bm{a}).\]
  The morphism \((\mg,\mproj)\colon \mY\to
  \mZ\times\PP(\bm{a})\) is an injection whose image is the
  subscheme
  \[\left\{(z_1,\ldots,z_n,[w_1\,:\, \cdots\,:\, w_n])\,:\,
    z_j^{a_k}w_k^{a_j}=z_k^{a_j}w_j^{a_k}\text{ for }1\leq
    j,k\leq n\right\}.\] This is often how the weighted blow-up is
  defined, but has the disadvantage that it is usually not
  normal: it is singular along the whole of \(\mC\). Indeed,
  \((\mg,\mproj)\) is its normalisation, in particular \(\mY\) is smooth in
  codimension \(1\).
\end{remark}

\begin{lemma}\label{lma:C_intersection}
  The divisor \(\left(\prod_{j=1}^na_j\right)\mC\) is
  Cartier. If \(T_{ij}\) is the toric curve corresponding to the
  wall \(\tau_{ij}\) separating the top-dimensional cones
  \(\sigma_i\) and \(\sigma_j\) then,
  \[\mC\cdot T_{ij} = -\frac{1}{\lcm(a_i,a_j)}.\]
\end{lemma}
\begin{proof}
  Recall (e.g. from {\cite[Theorem 4.2.8]{CoxLittleSchenck}})
  that a multiple \(m\mC\) is Cartier if and only if it admits
  {\em Cartier data}, namely a choice of vectors
  \(\bm{m}_{\sigma_j}\) for each top-dimensional cone
  \(\sigma_j\) such that \(\bm{m}_{\sigma_j}\cdot \bm{e}_k=0\)
  for all \(k\neq j\) and \(\bm{m}_{\sigma_j}\cdot
  \bm{a}=-m\). We see that \(\left(\prod_{j=1}^na_j\right)\mC\) is
  Cartier, with Cartier data
  \(\bm{m}_{\sigma_j}=-\left(\prod_{k\neq
      j}a_k\right)\bm{e}_j\).

  According to {\cite[Proposition 6.3.8]{CoxLittleSchenck}}, the
  intersection number
  \(-\left(\prod_{j=1}^na_j\right)\mC\cdot T_{ij}\) is given by
  \((\bm{m}_{\sigma_i}-\bm{m}_{\sigma_j})\cdot \bm{u}\), where
  \(\bm{u}\) is an integer vector in \(\sigma_j\) whose
  projection to the quotient of \(\RR^n\) by \(\tau_{ij}\)
  generates the integer lattice. Since \(\tau_{ij}\) contains
  all vectors \(\bm{e}_k\) with \(k\neq i,j\), we can first
  project to the span of \(\bm{e}_i\) and \(\bm{e}_j\). Under
  this projection,
  \begin{itemize}
  \item \(\bm{a}\) projects to \((a_i,a_j)\), and
  \item \(\bm{u}\) projects to an integer point \((u_i,u_j)\) whose
    integer affine displacement \(a_ju_i-a_iu_j\) from the
    projection of \(\bm{a}\) is minimal, that is, equal to
    \(\gcd(a_i,a_j)\) by B\'{e}zout's identity.
  \end{itemize}
  This means that
  \((\bm{m}_{\sigma_i}-\bm{m}_{\sigma_j})\cdot
  \bm{u}=-\left(\prod_{k\neq
      i,j}a_k\right)(a_j\bm{e}_i-a_i\bm{e}_j)\cdot \bm{u} =
  -\left(\prod_{k\neq i,j}a_k\right)\gcd(a_i,a_j)\). Putting this
  together gives the intersection number as stated.
\end{proof}

\begin{remark}
  Note that \(\mC\) is actually ``orbi-Cartier'' in the sense
  that it is the vanishing set of a section of an orbifold line
  bundle. Namely, let \(\mL\) be the quotient of
  \(\CC\times\left(\CC^n\setminus 0\right)\times\CC\) by
  \(\CC^\times\) acting with weights \((-1,a_1,\ldots,a_n,-1)\);
  this has an orbibundle projection \(\mL\to\mY\) and
  \([z_0:z_1:\cdots:z_n]\mapsto [z_0:z_1:\cdots:z_n:z_0]\)
  defines a section which vanishes precisely along \(\mC\). This
  line bundle is locally ample in the sense of Remark
  \ref{rmk:orbifold_kodaira_embedding}. If \(g\colon Y\to Z\) is
  a weighted blow-up with \(\exc(g)=C\) and
  \(N(g^*D-\epsilon C)\) is an ample Cartier divisor then, as
  remarked in Remark \ref{rmk:orbifold_kodaira_embedding}, the
  resulting K\"{a}hler form will be degenerate at the orbifold
  points. We can fix this by using instead
  \(N(g^*D-\epsilon C)+C=N(g^*D-(\epsilon-1/N)C)\). This is now
  orbi-ample, and since we can take \(N\) to be as large as we
  like, we can find an orbi-ample \(g^*D-\epsilon' C\) as close
  as we like to an ample \(g^*D-\epsilon C\). In other words, we
  can assume that \(\Upsilon_\varepsilon\) is orbi-ample in the
  definition of the Seshadri constant, and that we are working
  with genuine orbifold K\"{a}hler forms.
\end{remark}

\section{Ironing}
\label{sct:ironing}
Given a K\"{a}hler manifold, McDuff and Polterovich have a
procedure for ``ironing'' a symplectic form to make it flat near a
point whilst keeping it tame (see {\cite[Lemma
  5.5.B]{McDuffPolterovich}}). To blow down, we need to be able
to iron our symplectic forms to make them ``standard'' along
the exceptional locus. We need some lemmas, whose proofs closely
follow those of McDuff and Polterovich.

For convenience, let us order our basis of \(\RR^n\) so that the
entries of the vector \(\bm{a}\) satisfy
\[a_1\leq a_2\leq\cdots\leq a_n.\] We will continue to write
\(\mg\colon\mY\to\mZ\) for the toric weighted blow-up of
\(\mZ=\CC^n\) with weights \(\bm{a}\), and we also write
\(\omega_\mZ\) for the standard symplectic form on \(\mZ\).

\subsection{Concentrating volume near the origin}

\begin{lemma}[Compare with Part 1 of the proof of {\cite[Lemma
    5.5.B]{McDuffPolterovich}}]\label{lma:monotone_function} For
  any \(R>0\), \(\kappa>1\) and \(1>\epsilon>0\), there exists a
  K\"{a}hler form \(\tau_{\epsilon,\kappa}\) on \(\mZ(R)\) such
  that \(\tau_{\epsilon,\kappa}=\kappa^2\omega_\mZ\) on
  \(\mZ\left(\frac{R \epsilon^2 a_1}{4\kappa^2 a_n}\right)\) and
  \(\tau_{\epsilon,\kappa}=\epsilon^2\omega_\mZ\) near
  \(\partial\mZ(R)\).
\end{lemma}
\begin{proof}
  The ellipsoid \(\mZ(R)\) is sandwiched between the spheres
  of radii \(\sqrt{R/a_n}\) (inner) and \(\sqrt{R/a_1}\)
  (outer). Consider a smooth function
  \(h\colon[0,\sqrt{R/a_1}]\to [0,\sqrt{R/a_1}]\) with
  positive derivative such that (see Figure
  \ref{fig:monotone_function}):
  \begin{itemize}
  \item \(h(t)=t\) for \(t>\sqrt{\frac{R}{2a_n}}\)
  \item \(h(t)=\kappa t/\epsilon\) for
    \(t<\frac{\epsilon}{\kappa}\sqrt{\frac{R}{3a_n}}\).
  \end{itemize}
  Define the monotone
  embedding (in the sense of McDuff and Polterovich
  {\cite[p.425]{McDuffPolterovich}})
  \(\bm{z}\mapsto \frac{h(|\bm{z}|)}{|\bm{z}|}\bm{z}\). The
  pullback of \(\epsilon^2\omega_\mZ\) along this monotone
  embedding is the desired K\"{a}hler form \(\tau_{\epsilon,\kappa}\) which
  coincides with:
  \begin{itemize}
  \item \(\epsilon^2\omega_\mZ\) on a spherical shell which contains
    \(\partial\mZ(R)\),
  \item \(\kappa^2\omega_\mZ\) on the ball of radius
    \(\frac{\epsilon}{\kappa}\sqrt{\frac{R}{3a_n}}\), which
    contains the ellipsoid
    \(\mZ\left(\frac{R\epsilon^2
        a_1}{4\kappa^2a_n}\right)\).\qedhere
  \end{itemize}
\end{proof}

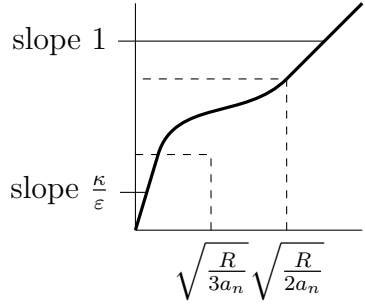
\begin{figure}[htb]
  \centering
  \begin{tikzpicture}
    \draw (0,0) -- (3,0);
    \draw (0,0) -- (0,3);
    \draw[dashed] (2,0) -- (2,2) -- (0,2);
    \draw[dashed] (1,0) -- (1,1) -- (0,1);
    \draw[very thick] (0,0) -- (0.3,1) to[out=73.3,in=-135] (2,2) -- (3,3);
    \node at (1,0) [below] {\(\sqrt{\frac{R}{3a_n}}\)};
    \node at (2,0) [below] {\(\sqrt{\frac{R}{2a_n}}\)};
    \node (a) at (-1,0.5) {slope \(\frac{\kappa}{\epsilon}\)};
    \draw (a) -- (0.15,0.5);
    \node (b) at (-1,2.5) {slope \(1\)};
    \draw (b) -- (2.5,2.5);
  \end{tikzpicture}
  \caption{The strictly increasing function used in the proof of
  Lemma \ref{lma:monotone_function}.}\label{fig:monotone_function}
\end{figure}

\subsection{Ironing near the exceptional locus}

\begin{lemma}[Compare with Part 2 of the proof of {\cite[Lemma
    5.5.B]{McDuffPolterovich}}]\label{lma:ironing}
  Let \(\upsilon\) and \(\gamma\) be two cohomologous symplectic
  forms on the orbifold \(\mY(R)\) which tame its complex
  structure. Then there exist constants \(0<\epsilon<\kappa\)
  and a new symplectic form \(\upsilon'\) on \(\mY(R)\) with
  the following properties:
  \begin{itemize}
  \item[(1)] \(\upsilon\) and \(\upsilon'\) are cohomologous,
  \item[(2)] \(\upsilon'\) and \(\upsilon\) agree near
    \(\partial\mY(R)\),
  \item[(3)] on
    \(\mY\left(\frac{R\epsilon^2a_1}{36\kappa^2a_n}\right)\),
    \(\upsilon'\) coincides with
    \(\gamma+(\kappa^2-\epsilon^2)\mg^*\omega_\mZ\),
  \item[(4)] \(\upsilon'\) still tames the complex structure.
  \end{itemize}
\end{lemma}
\begin{proof}
  Let:
  \begin{itemize}
  \item[(a)] \(\epsilon>0\) be such that
    \(\upsilon-\epsilon^2\mg^*\omega_\mZ\) still tames \(i\).
  \item[(b)] \(\beta\) be a primitive for \(\upsilon-\gamma\), that is
    \(\upsilon-\gamma=d\beta\).
  \item[(c)] \(\rho\) be a cut-off function on \(\mZ(R)\) such
    that \[\rho=\begin{cases}1&\text{ on }\mZ(R/9)\\ 0
        &\text{ on }\mZ(R)\setminus\mZ(4R/9).\end{cases}\]
  \end{itemize}
  For any \(\kappa>\epsilon\),
  let
  \[\rho_{\epsilon,\kappa}(z) = \rho\left(\frac{2\kappa z}{\epsilon}
      \sqrt{\frac{a_n}{a_1}}\right).\] Recall that in the
  notation \(\mZ(R)\), the quantity \(R\) is the value of a
  quadratic Hamiltonian on \(\mZ=\CC^n\), and hence scales
  quadratically when we rescale on \(\mZ\). Since \(\rho\) is
  supported on \(\mZ(4R/9)\), and since \(\rho_{\epsilon,\kappa}\) is
  obtained from \(\rho\) by rescaling the argument
  by \((2\kappa/\epsilon)\sqrt{a_n/a_1}\),
  the cut-off function \(\rho_{\epsilon,\kappa}\) is supported on
  \(\mZ\left(\frac{R\epsilon^2a_1}{9\kappa^2a_n}\right)\)
  and is equal to \(1\) on
  \(\mZ\left(\frac{R\epsilon^2a_1}{36\kappa^2a_n}\right)\).

  Let
  \[\upsilon' = \upsilon - \epsilon^2\mg^*\omega_\mZ +
    \mg^*\tau_{\epsilon,\kappa} - d(\mg^*(\rho_{\epsilon,\kappa})\,\beta),\] where
  \(\tau_{\epsilon,\kappa}\) is the symplectic form on \(\mZ\)
  constructed in Lemma \ref{lma:monotone_function}. This form
  satisfies the conclusions of the lemma:
  \begin{itemize}
  \item[(1)] \(\upsilon'\) and \(\upsilon\) differ by an exact form
    (note that since \(H^2(\mZ;\RR)=0\), anything pulled back
    along \(\mg\) is exact), so they are cohomologous.
  \item[(2)] Near \(\partial\mY(R)\), we have
    \(\mg^*\tau_{\epsilon,\kappa}=\epsilon^2\mg^*\omega_\mZ\) and
    \(\mg^*\rho_{\epsilon,\kappa}\equiv 0\), so \(\upsilon'=\upsilon\).
  \item[(3)] On
    \(\mY\left(\frac{R\epsilon^2a_1}{36\kappa^2a_n}\right)\),
    we have \(\tau_{\epsilon,\kappa}=\kappa^2\omega_\mZ\) and
    \(\rho_{\epsilon,\kappa}\equiv 1\) so
    \[\upsilon' = \upsilon + (\kappa^2 - \epsilon^2)\mg^*\omega_\mZ -
      d\beta = \gamma + (\kappa^2 - \epsilon^2)\mg^*\omega_\mZ.\]
  \item[(4)] On
    \(\mY(R) \setminus
    \mY\left(\frac{R\epsilon^2a_1}{9\kappa^2a_n}\right)\) we
    have \(\mg^*\rho_{\epsilon,\kappa}=0\) so
    \(\upsilon' = \left(\upsilon -
      \epsilon^2\mg^*\omega_\mZ\right) +
    \mg^*\tau_{\epsilon,\kappa}\), which is a sum of two taming
    forms (thanks to our choice of \(\epsilon\); see (a)) and so
    tame. On
    \(\mY\left(\frac{R\epsilon^2a_1}{36\kappa^2a_n}\right)\) we
    have seen that
    \(\upsilon'=\gamma+(\kappa^2-\epsilon^2)\mg^*\omega_\mZ\),
    which is tame. The shell
    \(\mY\left(\frac{R\epsilon^2a_1}{9\kappa^2a_n}\right)\setminus
    \mY\left(\frac{R\epsilon^2a_1}{36\kappa^2a_n}\right)\) on
    which it remains for us to check tameness is contained in
    \(\mY\left(\frac{R\epsilon^2a_1}{4\kappa^2a_n}\right)\), so
    we certainly have
    \(\mg^*\tau_{\epsilon,\kappa}\equiv
    \kappa^2\mg^*\omega_\mZ\). Therefore, for any vector \(V\),
    \begin{align*}
      \upsilon'(V,iV)&=(1-\rho_{\epsilon,\kappa})\upsilon(V,iV) +
                     (\kappa^2-\epsilon^2)\mg^*\omega_Z(V,iV) +\\
                   & \qquad +
                     \rho_{\epsilon,\kappa}\gamma(V,iV) -
                     (d(\mg^*\rho_{\epsilon,\kappa})\wedge\beta)(V,iV)\\
                   &\geq
                     \left(\kappa^2-\epsilon^2\right)|\mg_*V|^2 -
                     \frac{2\kappa}{\epsilon}\sqrt{\frac{a_n}{a_1}}
                     (d(\mg^*\rho)\wedge\beta)(V,iV)
    \end{align*}
    since \(0\leq \rho_{\epsilon,\kappa}\leq 1\), and
    \(\upsilon\) and \(\gamma\) tame \(i\). Now
    \(\left|d(\mg^*\rho)\wedge\beta(V,iV)\right|\leq
    c|\mg_*V|^2\) for some constant \(c\) independent of
    \(\kappa\), since
    \(\mY\left(\frac{R\epsilon^2a_1}{9\kappa^2a_n}\right)
    \setminus
    \mY\left(\frac{R\epsilon^2a_1}{36\kappa^2a_n}\right)\) is
    bounded and since \(\mg_*\neq 0\) on this subset. Therefore
    \[\upsilon'(V,iV)\geq \left(\kappa^2-\epsilon^2-\frac{2\kappa
          c}{\epsilon}\sqrt{\frac{a_n}{a_1}}\right)|\mg_*V|^2,\]
    which is positive provided \(\kappa\) is chosen large
    enough.\qedhere
  \end{itemize}
\end{proof}

\section{K\"{a}hler forms}
\label{sct:kaehler}
\subsection{Review of Guillemin--Abreu theory}

We will make extensive use of the Guillemin--Abreu theory of
toric K\"{a}hler geometry. We briefly recap what we need from
this theory; the reader can learn more from the original paper
by Guillemin \cite{Guillemin} or the subsequent work by Abreu
\cite{Abreu2,Abreu1} which also treats the orbifold
setting. Recall that a K\"{a}hler structure on a manifold \(X\)
comprises a complex structure \(J\) and a symplectic form
\(\omega\) which are compatible in the sense that, at every
point \(p\),
\[\omega(Jv,Jw)=\omega(v,w)\] for all tangent vectors
\(v,w \in T_pX\) and \(\omega(v,Jv) > 0\) for all nonzero
tangent vectors \(v\) (this second condition is called {\em
  tameness}: \(\omega\) tames \(J\)). If \(X\) is an orbifold
locally modelled on a quotient \(\CC^n/\Gamma\) by the action of
a finite group \(\Gamma\subset GL(n,\CC)\) then we can still
talk about K\"{a}hler structures: they just come from
\(\Gamma\)-invariant K\"{a}hler structures on \(\CC^n\).

Let \(\ell_1,\ldots,\ell_m\colon\RR^n\to\RR\) be linear maps and
\(R_1,\ldots,R_m\) be real numbers. The linear
inequalities \(\ell_i(\bmm)\geq R_i/2\) cut out a polytope
\[\Polytope=\left\{\bmm\in\RR^n\,:\,
    \ell_i(\bmm)\geq R_i/2\text{ for }i=1,\ldots,m\right\}.\] Given
this polytope, we can construct:
\begin{itemize}
\item the normal toric variety \(\cX\) (a complex orbifold)
  associated with the inward normal fan to \(\Polytope\). This
  admits a holomorphic \((\CC^\times)^n\)-action which has a
  free Zariski-open orbit. We identify this Zariski-open torus
  with \(\CC^n/(2\pi\ZZ)^n\) via the exponential map, writing
  complex coordinates \(\blr +i\bm{\theta}\) with
  \(\blr =(\lr_1,\ldots,\lr_n)\in\RR^n\) and
  \(\bm{\theta}=(\theta_1,\ldots,\theta_n)\in\RR^n/(2\pi
  \ZZ)^n\), so that the action of
  \(\bm{t}=(t_1,\ldots,t_n)\in T^n\subseteq (\CC^\times)^n\) is
  given by \(\blr +i(\bm{\theta}+\bm{t})\). We refer to these as
  {\em complex coordinates}. The notation \(\bm{\lambda}\) is
  intended to evoke ``log-radius''.
\item a toric symplectic orbifold \((\sX,\omega)\) together with
  a moment map \(\bmm\colon \sX\to \Polytope\). This construction is
  due to Lerman and Tolman \cite{LermanTolman}, following
  Delzant's construction in the case of manifolds
  \cite{Delzant}. This admits a Hamiltonian action of the torus
  \(T^n\) which is free over the interior
  \(\Polytope^\circ\subseteq\Polytope\) of the moment polytope. The
  \(\bmm\)-preimage of \(\Polytope^\circ\) is symplectomorphic to
  \(\Polytope^\circ\times T^n\) equipped with the symplectic form
  \(\omega=\sum_{j=1}^n d\mm_i\wedge d\theta_i\) where
  \((\mm_1,\ldots,\mm_n)\) are linear coordinates on \(\RR^n\) and
  \((\theta_1,\ldots,\theta_n)\in T^n=\RR^n/(2\pi\ZZ)^n\) are
  angle coordinates such that the circle action generated by
  \(\mm_j\) rotates \(\theta_j\) at unit speed. We refer to these
  as {\em symplectic} (or {\em action-angle}) {\em coordinates}.
\end{itemize}
The orbifolds \(\cX\) and \(\sX\) are diffeomorphic, and the
Guillemin--Abreu theory tells us how to combine the symplectic
and complex geometry to obtain a (torus-invariant) K\"{a}hler structure on either
side. There are many torus-invariant K\"{a}hler structures, and
choosing one is equivalent to choosing a {\em symplectic
  potential}:

\begin{definition}[Symplectic potential]
  A symplectic potential is a strictly convex function
  \(G\colon \Polytope\to\RR\) of the form
  \[G(\bmm)=\sum_{i=1}^mL(\ell_i(\bmm)-R_i/2)+h(\bmm)\]
  where:
  \begin{itemize}
  \item \(L(t)=\frac{1}{2}(t\log(2t)-t)\).
  \item \(h(\bmm)\) is a smooth function on \(\Polytope\).
  \item the limit of
    \(\left(\det\Hess(G)\right) \cdot
    \prod_{i=1}^m(\ell_i(\bmm)-R_i/2)\) as you approach
    \(\partial\Polytope\) is strictly positive (note that the
    product vanishes along \(\partial\Polytope\), but the
    Hessian determinant has poles).
  \end{itemize}
\end{definition}

\begin{remark}
  Guillemin and Abreu use the function \(\frac{1}{2}t\log t\)
  instead of our \(L\). These functions differ only by a linear
  term, which can be absorbed into \(h\). This is also
  responsible for the factor of \(1/2\) in \(R_i/2\) which will
  pervade our equations. The advantage of our choice is that it
  yields the standard symplectic form and moment map on
  \(\CC^n\).
\end{remark}

\begin{construction}
  We use a symplectic potential to define a torus-invariant
  complex structure on \(\sX\) compatible with \(\omega\). In
  terms of the symplectic coordinates \((\bmm,\bm{\theta})\) on
  the open set \(\bmm^{-1}(P^\circ)\), the complex structure is
  given by the block matrix
  \[J = \begin{pmatrix}0 & -\OP{Hess}(G)^{-1}\\ \OP{Hess}(G) &
      0\end{pmatrix}.\] Integrability of \(J\) is equivalent to
  the fact that the off-diagonal blocks are Hessian matrices for
  some function \(G\). Abreu {\cite[Appendix A]{Abreu1}}
  explains why this complex structure extends over the toric
  boundary. The complex orbifold \((\sX,J)\) is biholomorphic to
  \((\cX,i)\) via an explicit biholomorphism
  \(\Psi\colon \sX\to \cX\). In symplectic and complex
  coordinates on \(\bmm^{-1}(P^\circ)\), \(\Psi\) can be written as:
  \[\Psi(\bmm,\bm{\theta})=\psi(\bmm)+i\bm{\theta},\qquad
    \psi(\bmm) = \frac{\partial G}{\partial \bmm}.\] Moreover,
  the pullback \((\Psi^{-1})^*\omega\) on
  \(\bmm^{-1}(P^\circ)\subseteq \cX\) can be written as
  \(2i\partial\bar{\partial}F\) for a {\em K\"{a}hler} potential
  \[F(\blr +i\bm{\theta})=\blr
    \cdot\psi^{-1}(\blr)-G(\psi^{-1}(\blr)).\] In other words,
  \(F\) is the {\em Legendre transform} of \(G\) and \(\lr_j\)
  is the conjugate variable to \(\mm_j\). The moment map for the
  torus action is given by projection to \(\bmm\) in symplectic
  coordinates and by
  \(\psi^{-1}(\blr)=\frac{\partial F}{\partial\blr}\) in complex
  coordinates; that is, the moment map is given by the conjugate
  variables to \(\blr\).
\end{construction}

\begin{remark}
  Changing the symplectic potential by a linear term doesn't
  affect the Hessian; it does affect the K\"{a}hler potential
  and the moment map.
\end{remark}

\subsection{Standard symplectic vector space}

\begin{example}\label{exm:Cn}
  Let
  \(\Polytope_\mZ=\RR_{\geq 0}^n\), that is \(\ell_j(\bmm)=\mm_j\) and
  \(R_j=0\) for \(j=1,\ldots,n\). The associated complex
  manifold is \(\mZ=\CC^n\). Take the symplectic potential
  \(G_\mZ=\sum_{j=1}^nL(\mm_j)\). The conjugate variables and
  Legendre transform are given by
  \begin{gather*}
    \blr = \psi_\mZ(\bmm)= \frac{\partial G_\mZ}{\partial \bmm} =
    \frac{1}{2}\log(2\bmm),\qquad \bmm=\frac{1}{2}e^{2\blr}\\
    F_\mZ =
    \blr\cdot\bmm-G_\mZ(\bmm)=\frac{1}{4}\sum_{j=1}^ne^{2\lr_j}.
  \end{gather*}
  If we write \(\bm{z}=e^{\blr+i\bm{\theta}}\) then
  \(F_\mZ(\bm{z})=\frac{1}{4}\sum_{j=1}^n|z_j|^2\) and
  \(\omega_\mZ=2i\partial\bar{\partial}
  F_\mZ=\frac{i}{2}\sum_{j=1}^n dz_j\wedge d\bar{z}_j\) is the
  standard symplectic form. The moment map is
  \(\bmm=\left(\frac{1}{2}|z_1|^2,\ldots,\frac{1}{2}|z_n|^2\right)\).

  Given a vector \(\bm{a}\) and a constant \(R>0\), define the subsets
  \[\trifil_{\bm{a},\,R}=\left\{\bmm\in\RR^n\,:
      \,\bm{a}\cdot\bmm\leq
      \frac{R}{2}\right\},\qquad \tri_{\bm{a},\,R}=\left\{\bmm\in\RR^n\,:
      \,\bm{a}\cdot\bmm=
      \frac{R}{2}\right\}.\] Note that \(\tri_{\bm{a},\,R}=\tri_{\bm{a}/R,\,1}\). The ellipsoid
  \[\Ell(R\bm{a}^{-1})=\left\{\bm{z}\,:\,
      \sum_{j=1}^na_j|z_j|^2\leq R\right\}\]
  projects under the moment map \(\bmm\) to the subset
  \(\trifil_{\bm{a},\,R}\).

  In conjugate coordinates, we introduce the notation:
  \[\mem_{\bm{a},\,R}= \left\{\blr\in\RR^n\,:\,
      \sum_{j=1}^na_je^{2\lr_j}\leq R\right\},\quad
    \resh_{\bm{a},\,R} = \left\{\blr\in\RR^n\,:\,
      \sum_{j=1}^na_je^{2\lr_j}=R\right\},\] so that
  \(\mem_{\bm{a},\,R}=\psi_\mZ(\trifil_{\bm{a},\,R})\) and
  \(\resh_{\bm{a},\,R}=\psi_\mZ(\tri_{\bm{a},\,R})\). This means
  that \(\blr(\Ell(R\bm{a}^{-1}))=\mem_{\bm{a},\,R}\)
  and
  \(\blr(\partial\Ell(R\bm{a}^{-1}))=\resh_{\bm{a},\,R}\).

  The choice of notation\footnote{The Hebrew letters \(\resh\)
    (resh) and \(\mem\) (mem).} reflects the shape of the
  subsets when \(n=2\); see Figure \ref{fig:resh_and_mem}. Note
  that \(\resh_{\bm{a},\,R}\) is a translation of
  \(\resh_{\bm{1},\,1}\): if we write \(\bm{1}\) for the
  vector \((1,\ldots,1)\) then
  \begin{equation}\label{eq:resh_translate}\resh_{\bm{a},\,R} =
    \resh_{\bm{1},\,1}-\frac{1}{2}\log\left(\bm{a}\right)
    +\frac{1}{2}\log(R)\bm{1}.\end{equation}
\end{example}

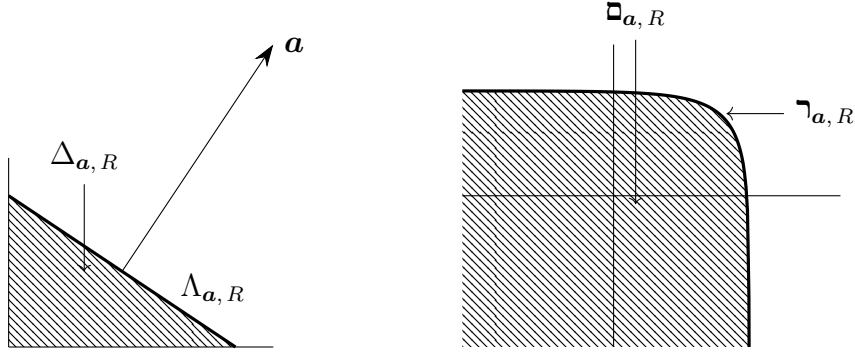
\begin{figure}[htb]
  \centering
  \begin{tikzpicture}
    \draw (0,0) -- (0,2.5);
    \draw (0,0) -- (3.5,0);
    \fill[pattern=north west lines] (0,2) -- (3,0) -- (0,0) -- cycle;
    \draw[very thick] (0,2) -- (3,0) node [pos=0.9,above=2mm] {\(\tri_{\bm{a},\,R}\)};
    \draw[-{Stealth[length=3mm, width=2mm]}] (1.5,1) -- ++ (2,3) node [right] {\(\bm{a}\)};
    \node (nd) at (1,2.5) {\(\trifil_{\bm{a},\,R}\)};
    \draw[->] (nd) -- (1,1);
    \begin{scope}[shift={(8,2)}]
      \draw (0,-2) -- (0,2);
      \draw (-2,0) -- (3,0);
      \begin{scope}[shift={(1.791759469,1.386294361)}]
        \draw[very
        thick,variable=\x,samples=200,domain=-3.791759:-0.0000001] plot
        ({\x},{0.5*ln(1-exp(2*\x))}) -- (0,-3.38629);
        \node (resh) at (1,-0.3) {\(\resh_{\bm{a},\,R}\)};
        \draw[->] (resh) -- (-0.3,-0.3);
        \node (mem) at (-1.5,1) {\(\mem_{\bm{a},\,R}\)};
        \draw[->] (mem) -- (-1.5,-1.5);
        \begin{scope}
          \path[clip] plot[samples=200,domain=-3.791759:-0.0000001]
          ({\x},{0.5*ln(1-exp(2*\x))}) -- (0,-3.38629) -- (-3.791759,-3.38629) -- cycle;
          \fill[pattern=north west lines] (-4,0) -- (0,0) -- (0,-4) -- (-4,-4) -- cycle;
        \end{scope}
      \end{scope}
    \end{scope}
  \end{tikzpicture}
  \caption{The projections of the ellipsoid
    \(\Ell(R\bm{a}^{-1})\subseteq\CC^2\) under \(\bmm\) (left)
    and \(\blr\) (right) from Example \ref{exm:Cn}. In this
    example, \(\bm{a}=(2,3)\), \(R=12\) and 
    \(\bm{\sigma}=R\bm{a}^{-1}=(6,4)\).}\label{fig:resh_and_mem}
\end{figure}

\subsection{Standard model of toric weighted blow-up}

\begin{example}
  As in Remark \ref{rmk:polytope}, let
  \[\Polytope_\mY(R_0)=\{\bmm=(\mm_1,\ldots,\mm_n)\in\Polytope_\mZ\,:\,
    \ell_0(\bmm)\geq R_0/2\},\] where
  \(\ell_0(\bmm)=\bm{a}\cdot\bmm\) and \(R_0>0\) is a positive
  constant; this is obtained from the orthant \(\Polytope_\mZ\)
  by chopping off a corner, which introduces a new face
  \(\tri_{\bm{a},\,R_0}\) with inward normal \(\bm{a}\). The
  associated complex orbifold \(\mY\) is the toric weighted
  blow-up of \(\mZ\) at the origin with weights
  \(\bm{a}\). Write \(\mg\colon\mY\to\mZ\) for the weighted
  blow-up; the exceptional locus \(\mC\) of \(\mg\) is a
  weighted projective space \(\PP(a_1,\ldots,a_n)\). The
  symplectic orbifold associated to \(\Polytope_\mY(R_0)\) also
  contains \(\mC\); the moment image of \(\mC\) is the new face
  \(\tri_{\bm{a},\,R_0}\).

  Choose a vector \(\bm{b}\in\RR^n\) and take the symplectic
  potential
  \[G^{R_0,\bm{b}}_\mY(\bmm)=\sum_{j=1}^nL(\mm_j) +
    L(\ell_0(\bmm) - R_0/2) + \bm{b}\cdot\bmm;\] then
  we obtain a K\"{a}hler form
  \(\omega^{R_0,\bm{b}}_\mY\) on \(\mY\). The conjugate
  variables are given by
  \[\blr=\psi^{R_0,\bm{b}}_\mY(\bmm)=\frac{\partial G^{R_0,\bm{b}}_\mY}{\partial\bmm} =
    \frac{1}{2}\log(2\bmm)+ \frac{1}{2}\log(2\ell_0(\bmm) -
    R_0)\bm{a}+\bm{b}.\] To invert this and find
  \(\bmm\) in terms of \(\blr\), note that
  \begin{equation}\label{eq:conjugate_variables} \mm_j =
    \frac{1}{2}\cdot\frac{\exp(2(\lr_j-b_j))}{(2\ell_0(\bmm)-R_0)^{a_j}}
  \end{equation}
  and since \(\ell_0(\bmm)=\bm{a}\cdot\bmm\), we get
  \[\ell_0(\bmm) =
    \frac{1}{2}\sum_{j=1}^n\frac{a_j\exp(2(\lr_j -
      b_j))}{(2\ell_0(\bmm) - R_0)^{a_j}}.\] Clearing
  denominators, this gives a polynomial equation for the value
  of \(\ell_0(\bmm)\) in terms of \(\blr\). Assuming this is
  solved for \(\ell_0(\bmm)\), then Equation
  \eqref{eq:conjugate_variables} gives a formula for the
  conjugate variables \(\bmm\) in terms of
  \(\ell_0(\bmm)\). Whilst it seems infeasible to solve this
  polynomial explicitly, it is nonetheless easy to use Equations
  \eqref{eq:resh_translate} and \eqref{eq:conjugate_variables}
  to find the image under \(\psi^{R_0,\bm{b}}_\mY\) of
  \(\tri_{\bm{a},\,R}=\{\bmm\,:\,\ell_0(\bmm)=R/2\}\) for any
  given value of \(R>R_0\):
  \[\psi^{R_0,\bm{b}}_\mY(\tri_{\bm{a},\,R})=\resh_{\bm{a},\,R} +
    \bm{b}+\frac{1}{2}\bm{a}\log(R-R_0).\] Fix an \(\Rb>R_0\)
  and recall that \(\mY(\Rb)=\mg^{-1}(\mZ(\Rb))\).
\end{example}

\begin{definition}\label{dfn:symp_pot}
  Define \(G^{R_0}_{\mY(\Rb)}\) to be the symplectic potential
  on \(\Polytope_\mY(R_0)\) given
  by
  \[G^{R_0,\bm{b}}_{\mY}\qquad \text{with} \qquad\bm{b} =
    -\frac{1}{2}\bm{a}\log(\Rb-R_0).\] Writing
  \begin{equation}
    \label{eq:leg_trans}
    \psi_{\mY(\Rb)}^{R_0}(\bmm)\coloneqq \psi_{\mY}^{R_0,\bm{b}}(\bmm)
    = \frac{1}{2}\log(2\bmm) -
    \frac{1}{2}\bm{a}\log\left(\frac{\Rb - R_0}{2\ell_0(\bmm) -
        R_0}\right),
  \end{equation}
  we have
  \(\psi^{R_0}_{\mY(\Rb)}(\tri_{\bm{a},\,\Rb}) =
  \resh_{\bm{a},\,\Rb}\). We will write
  \(\omega^{R_0}_{\mY(\Rb)}\) for the resulting K\"{a}hler form
  on \(\mY\) (indeed, we will usually restrict attention to the
  domain \(\mY(\Rb)\subset\mY\) whose boundary
  \(\partial\mY_{\Rb}\) projects via \(\blr\) to
  \(\resh_{\bm{a},\,\Rb}\)).
\end{definition}

\begin{remark}
  Note that \(\mY(R)\) is also defined for \(R\leq
  R_0\). However, if we choose some \(\Rb>R_0\) and equip
  \(\mY(\Rb)\) with the K\"{a}hler form
  \(\omega^{R_0}_{\mY(\Rb)}\) then we see that the image of
  \(\mY(\Rb)\) under the moment map is
  \(\Polytope_\mY(R_0)\cap\trifil_{\bm{a},\,\Rb}\), so that
  \(\mY(\Rb)\) is the result of symplectically cutting an
  ellipsoid \(\Ell(R_0\bm{a}^{-1})\) out of
  \(\Ell(\Rb\bm{a}^{-1})\). The \(\psi^{R_0}_{\mY(\Rb)}\)-images
  of the parallel lines \(\tri_{\bm{a},\,R}\) with \(R_0<R<\Rb\)
  are then
  \begin{equation}\label{eq:psi_images}
    \psi^{R_0}_{\mY(\Rb)}(\tri_{\bm{a},\,R})= \resh_{\bm{a},\,\Rb} -
    \frac{1}{2}\bm{a}\log\left(\frac{\Rb - R_0}{R -
      R_0}\right),\end{equation} which fill out the whole of
  \(\mem_{\bm{a},\,\Rb}\); see Figure \ref{fig:leg_transf_R_0}. I
  highlight this because the main difference between complex and
  symplectic (weighted) blow-up is that in the symplectic case one
  usually thinks of ``cutting out a ball'' whereas in the complex case
  one ``cuts out a point''.  This example shows how to equip the ball
  punctured at a point with a symplectic form in such a way that it
  appears that a whole ball has been excised.
\end{remark}

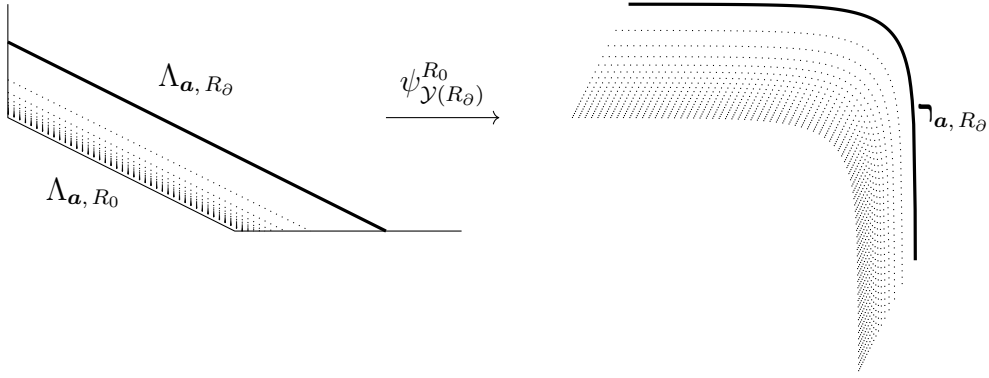
\begin{figure}[htb]
  \centering
  \begin{tikzpicture}
    \draw (0,3) -- (0,1.5) -- (3,0) -- (6,0);
    \draw[very thick] (0,1.5+1) -- (3+2,0);
    \foreach \z in {1,2,...,20} {
      \draw[dotted] (0,{1.5+1/\z}) -- ({3+2/\z},0);
    }
    \draw[->] (5,1.5) -- (6.5,1.5) node [midway,above] {\(\psi^{R_0}_{\mY(\Rb)}\)};
    \node at (2.5,2) {\(\tri_{\bm{a},\,\Rb}\)};
    \node at (1,0.5) {\(\tri_{\bm{a},\,R_0}\)};
    \begin{scope}[shift={(12,3)}]
      \node at (0.5,-1.5) {\(\resh_{\bm{a},\,\Rb}\)};
      \draw[very thick,variable=\x,samples=200,domain=-3.791759:-0.0000001] plot ({\x},{0.5*ln(1-exp(2*\x))}) -- (0,-3.38629);
      \foreach \y in {2,...,20} {
        \begin{scope}[shift={({-0.25*ln(\y)},{-0.5*ln(\y)})}]
          \draw[dotted,variable=\x,samples=200,domain=-3.791759:-0.0000001] plot ({\x},{0.5*ln(1-exp(2*\x))}) -- (0,-3.38629);
        \end{scope}
      }
    \end{scope}
  \end{tikzpicture}
  \caption{The Legendre transform \(\psi^{R_0}_{\mY(\Rb)}\) sends \(\tri_{\bm{a},\,R}\) in the moment polytope to a translate of \(\resh_{\bm{a},\,\Rb}\). The thickly drawn level set show the boundary of \(\mY(\Rb)\); the dotted level sets in the moment polytope accumulate at \(\tri_{\bm{a},\,R_0}\); their images under the Legendre transform do not accumulate, but rather escape logarithmically slowly to the bottom left.}\label{fig:leg_transf_R_0}
\end{figure}

\subsection{Implanting a weighted blow-up}

For the purposes of performing the weighted blow-up, we need a
symplectic form on \(\mY\) which coincides with the flat
K\"{a}hler form \(\mg^*\omega_\mZ\) outside \(\mY(R)\) for
sufficiently small \(R\). In fact, we will only need the
construction in the other direction (blowing down), but we
explain weighted blowing up for completeness.

\begin{definition}\label{dfn:chi_delta}
  Given \(\delta>0\), let \(\chi_\delta\) be the convex
  \(C^1\) function defined piecewise by
  \begin{equation}\label{eq:chi_delta}
    \chi_\delta(t)=\begin{cases}\frac{1}{2}(t\log(t/\delta)+\delta-t)&\text{
        if }t\leq \delta\\ 0&\text{ if }t\geq
      \delta. \end{cases}
  \end{equation} Given \(0<\eta<\delta\), let
  \(\chi_{\delta,\eta}\) be the convex smoothing of
  \(\chi_\delta\) given by the construction of Ghomi
  {\cite[Section 2]{Ghomi}}. This is a convex \(C^\infty\)
  function which coincides with \(\chi_\delta\) outside the
  interval \((\delta-\eta,\delta+\eta)\). Note that if
  \(t\leq\delta\) then
  \(\chi_\delta(t)=L(t)+\frac{1}{2}(\delta-t\log(2\delta))\).
\end{definition}

\begin{figure}[htb]
  \centering
  \begin{tikzpicture}
    \fill[lightgray] (0.8,1) -- (0.8,0) -- (1.2,0) -- (1.2,1) -- cycle;
    \draw[dotted] (1,0) -- (1,1);
    \draw[<->] (0.8,1) -- (1.2,1) node [midway,above] {\(2\eta\)};
    \draw (0,0) -- (0,1) node [left] {\(\chi_\delta(t)\)};
    \draw (0,0) -- (3,0) node [below] {\(t\)};
    \node at (1,0) [below] {\(\delta\)};
    \draw[very thick,variable=\x,smooth,domain=0.0001:1] plot (\x,{0.5*(\x*ln(\x/1)+1-\x)});
    \draw[very thick] (1,0) -- (3,0);
  \end{tikzpicture}
  \caption{The function \(\chi_\delta(t)\) defined in Equation
    \eqref{eq:chi_delta}. Its smoothing looks almost identical
    but differs slightly on the shaded strip.}\label{fig:chi_delta}
\end{figure}
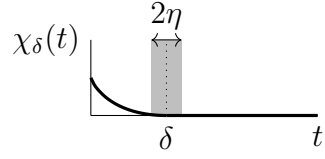

\begin{example}
  Given \(R_0>0\) and \(0<\eta<\delta\), consider the
  symplectic potential
  \[G_{\mY}^{\ironed}=\sum_{j=1}^nL(\mm_j) +
    \chi_{\delta,\eta}(\ell_0(\bmm)-R_0/2)\] on
  \(\Polytope_\mY(R_0)\) and let \(\omega^{\ironed}_{\mY}\) be
  the associated K\"{a}hler form on \(\mY\). See Figure
  \ref{fig:G_ironed_Y} for an illustration of
  \(G^{\ironed}_{\mY}\). This symplectic potential coincides
  with:
  \begin{itemize}
  \item \(G^{R_0,\bm{0}}_{\mY}\) on
    \(\Polytope_\mY(R_0)\cap\{\ell_0\leq R_0+2(\delta-\eta)\}\),
    in other words on a neighbourhood of the face
    \(\tri_{\bm{a},\,R_0}\), and
  \item with \(G_\mZ\) on
    \(\Polytope_\mY(R_0)\cap\{\ell_0\geq R_0+2(\delta+\eta)\}\).
  \end{itemize}
  As a consequence, the diffeomorphism
  \(\psi^{\ironed}_\mY(\bmm)=\frac{\partial
    G^{\ironed}_{\mY}}{\partial\bmm}\) coincides with
  \(\psi_\mZ\) on
  \(\Polytope_\mY(R_0)\cap\{\ell_0\geq R_0+2(\delta+\eta)\}\), which
  means that \(\omega_{\mY}^{\ironed}\) agrees with
  \(\mg^*\omega_\mZ\) on
  \(\mY\setminus\mY(R_0+2(\delta+\eta))\).
\end{example}

\begin{lemma}
  Given a complex manifold \(Z\) and a weighted blow up
  
  \begin{center}
    \begin{tikzpicture}
      \node (V) at (0,0) {\(\mY(R)\)};
      \node (W) at (0,-2) {\(\mZ(R)\)};
      \node (Y) at (2,0) {\(Y\)};
      \node (Z) at (2,-2) {\(Z\)};
      \draw[->] (V) -- (Y) node [midway,above] {\(\jmath\)};
      \draw[->] (V) -- (W) node [midway,left] {\(\mg\)};
      \draw[->] (W) -- (Z) node [midway,above] {\(\iota\)};
      \draw[->] (Y) -- (Z) node [midway,right] {\(g\)};
    \end{tikzpicture}
  \end{center}

  with weights \(\bm{a}\), suppose that \(Z\) admits a tame
  symplectic form \(\omega\) such that
  \(\iota^*\omega=\omega_\mZ\). Then, for any \(R_0<R\), there
  exists a symplectic form on \(Y\) taming the complex structure
  and making \(C=\exc(g)\) symplectomorphic to the toric
  orbifold with moment image \(\tri_{\bm{a},\,R_0}\).
\end{lemma}
\begin{proof}
  We use \(\mg^*\omega\) on the complement of
  \(\jmath(\mY(R))\). On \(\jmath(\mY(R))\) we use
  \((\jmath^{-1})^*\omega^{\ironed}_{\mY}\) with \(\delta,\eta\)
  chosen so that \(R_0+2(\delta+\eta)<R\). The
  symplectic form \(\omega^{\ironed}_{\mY}\) agrees with
  \(\mg^*\omega_{\mZ}\) in a neighbourhood of
  \(\partial\mY(R)\), so this piecewise definition matches
  smoothly. The symplectic form \(\omega^{\ironed}_{\mY}\)
  coincides with \(\omega^{R_0,\bm{0}}_\mY\) in a
  neighbourhood of \(\mC\), and so makes \(C\) symplectomorphic
  to the toric orbifold with moment image
  \(\tri_{\bm{a},\,R_0}\).
\end{proof}

\begin{figure}[htb]
  \centering
  \begin{tikzpicture}
    \draw (0,3) -- (0,1) -- (1.5,0) -- (5,0);
    \draw[dotted] (2.5,0) -- ++ (-3.333*0.75,3.333*0.5);
    \draw[dotted] (3.5,0) -- ++ (-4.666*0.75,4.666*0.5);
    \draw[dotted] (1.5,0) -- ++ (0.75,-0.5) node [pos=1.5,sloped] {\(R_0\)};
    \draw[dotted] (2.5,0) -- ++ (2*0.75,-2*0.5);
    \draw[draw=none] (3,0) -- ++ (3.5*0.75,-3.5*0.5) node [midway,sloped] {\(R_0+2(\delta\pm\eta)\)};
    \draw[dotted] (3.5,0) -- ++ (2*0.75,-2*0.5);
    \draw[draw=none] (1.5,2) -- ++ (1.5,-1) node [midway,sloped] {\(G^{\ironed}_{\mY}=G_{\mZ}\)};
    \node (g) at (-2,1) {\(G^{\ironed}_{\mY}=G^{R_0,\bm{0}}_{\mY}\)};
    \draw[->] (g) -- (1,0.7);
    \node (ls) at (-1,-1) {level sets of \(2\ell_0\)};
    \draw[->] (ls) -- (1.8,-0.2);
    \draw[->] (ls) -- (1.8+1,-0.2);
  \end{tikzpicture}
  \caption{A schematic of the symplectic potential
    \(G^{\ironed}_{\mY}\) when \(n=2\).}\label{fig:G_ironed_Y}
\end{figure}
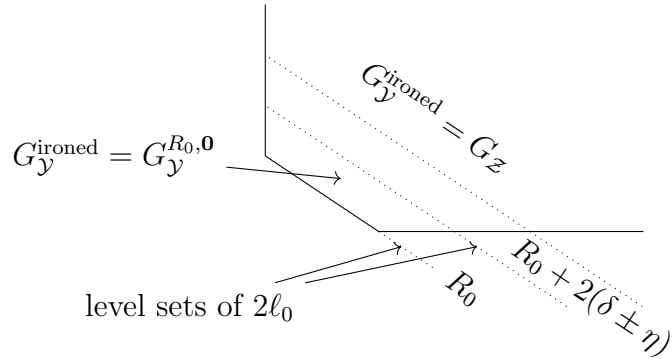

\begin{remark}
  By {\cite[Lemma 5.5.B]{McDuffPolterovich}}, if we have an
  almost K\"{a}hler manifold \((Z,\omega,J)\) and a holomorphic
  embedding \(\iota\colon \mZ(R)\to Z\) for some \(R\), we can
  ``iron'' the symplectic form near to \(\iota(0)\) to get a new
  symplectic form \(\omega'\) taming \(J\) such that
  \(\iota^*\omega'\) coincides with a multiple of \(\omega_\mZ\)
  on some \(\mZ(R')\subseteq\mZ(R)\) and \(\omega'\) coincides
  with \(\omega\) on the complement of
  \(\iota(\mZ(R))\). Therefore we can always perform the
  weighted blow-up along a sufficiently small ellipsoid after
  ironing.
\end{remark}

\subsection{Implanting a weighted blow-down}

We need one more K\"{a}hler form, this time on \(\mZ\), for the
purposes of blowing down.

\begin{definition}
  Given \(\delta>0\), let \(\zeta_\delta\) be the convex \(C^1\)
  function defined piecewise by
  \begin{equation}\label{eq:zeta_delta}
    \zeta_\delta(t)=\begin{cases}
      \frac{1}{2}(t\log 2\delta-\delta)&\text{ if }t\leq \delta\\
      L(t)&\text{ if }t\geq \delta.
      \end{cases}
  \end{equation}
  Given \(0<\eta<\delta\), let \(\zeta_{\delta,\eta}\) be the
  convex smoothing of \(\zeta_\delta\) given by the construction
  of Ghomi {\cite[Section 2]{Ghomi}}. This is a convex
  \(C^\infty\) function which coincides with \(\zeta_\delta\)
  outside the interval
  \((\delta-\eta,\delta+\eta)\).
\end{definition}

\begin{figure}[htb]
  \centering
  \begin{tikzpicture}
    \fill[lightgray] (0.8,1) -- (0.8,-1) -- (1.2,-1) -- (1.2,1) -- cycle;
    \draw[dotted] (1,-0.5) -- (1,1);
    \draw[<->] (0.8,1) -- (1.2,1) node [midway,above] {\(2\eta\)};
    \draw (0,-1) -- (0,1) node [left] {\(\zeta_\delta(t)\)};
    \draw (-1,0) -- (3,0) node [below] {\(t\)};
    \node at (1,-0.5) [below] {\(\delta\)};
    \draw[very thick,variable=\x,smooth,domain=-1:1] plot (\x,{0.5*(-1+ln(2)*\x)});
    \draw[very thick,variable=\x,smooth,domain=1:3] plot (\x,{0.5*(\x*ln(2*\x)-\x)});
  \end{tikzpicture}
  \caption{The function \(\zeta_\delta(t)\) defined in Equation
    \eqref{eq:zeta_delta}. Its smoothing looks almost identical
    but differs slightly on the shaded strip.}\label{fig:zeta_delta}
\end{figure}
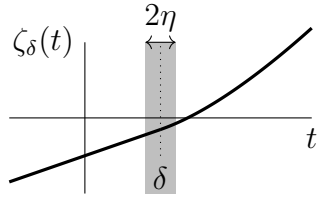

\begin{example}\label{exm:ironed_Z}
  Let \(0<R_0<\Rb\) and \(0<\eta<\delta\) be numbers
  such that \(R_0+2(\delta+\eta)<\Rb\). Let
  \[G^{\ironed}_{\mZ(\Rb)} = \sum_{j=1}^nL(\mm_i) +
    \zeta_{\delta,\eta}\left(\ell_0(\bmm) - \frac{R_0}{2}\right)
    - \frac{1}{2}\bm{a}\cdot\bmm\log(\Rb-R_0),\] let
  \(\psi^{\ironed}_{\mZ(\Rb)}(\bmm)=\frac{\partial
    G^{\ironed}_{\mZ(\Rb)}}{\partial\bmm}\) be the
  diffeomorphism to the conjugate variables, let
  \(F^{\ironed}_{\mZ(\Rb)}\) be the Legendre-dual K\"{a}hler
  potential and let \(\omega^{\ironed}_{\mZ(\Rb)}\) be the
  associated K\"{a}hler form. See Figure \ref{fig:G_ironed_Z}
  for an illustration of \(G^{\ironed}_{\mZ(\Rb)}\).

  Recall the symplectic potential \(G^{R_0}_{\mY(\Rb)}\) from
  Definition \ref{dfn:symp_pot}. We have
  \begin{equation}\label{eq:matching_potentials}
    G^{\ironed}_{\mZ(\Rb)}=G^{R_0}_{\mY(\Rb)}\qquad
    \text{on}\qquad\trifil_{\bm{a},\,\Rb} \setminus
    \trifil_{\bm{a},\,R_0+2(\delta+\eta)}.
  \end{equation}
  By Equation \eqref{eq:psi_images}, this means that, for any
  \(R\in [R_0+2(\delta+\eta),\Rb]\):
  \begin{equation}\label{eq:ironed_images}
    \psi_{\mZ(\Rb)}^{\ironed}(\tri_{\bm{a},\,R}) =
    \resh_{\bm{a},\,R} -
    \frac{1}{2}\bm{a}\log\left(\frac{\Rb - R_0}{R -
        R_0}\right).
  \end{equation}
  In particular,
  \(\psi_{\mZ(\Rb)}^{\ironed}(\tri_{\bm{a},\,\Rb})=\resh_{\bm{a},\,\Rb}\).
  On the other hand, by taking \(\delta\) sufficiently small, we
  can make
  \(\log\left(\frac{\Rb-R_0}{2(\delta+\eta)}\right)\) as
  large as we want, so for any \(0<\Sb<\Rb\) we can ensure
  that
  \begin{equation}\label{eq:ironed_subset}
    \psi^{\ironed}_{\mZ(\Rb)}(\trifil_{\bm{a},\,R_0
      + 2(\delta+\eta)}) \subseteq \mem_{\bm{a},\,\Sb/2}.
  \end{equation}
  Moreover, Equation \eqref{eq:matching_potentials} implies that
  \(\omega_{\mZ(\Rb)}^{\ironed}\) pulls back to
  \begin{equation}\label{eq:matching_forms}
    \mg^*\omega_{\mZ(\Rb)}^{\ironed}=\omega^{R_0}_{\mY(\Rb)}\qquad
    \text{on} \qquad \blr^{-1}\left(\psi^{\ironed}_{\mZ(\Rb)}(\trifil_{\bm{a},\,\Rb} \setminus
      \trifil_{\bm{a},\,R_0+2(\delta+\eta)})\right)\subseteq\mY(\Rb),
  \end{equation}
  but, if we take \(\delta\) sufficiently small then,
  \(\mem_{\bm{a},\,\Rb}\setminus\mem_{\bm{a},\,\Sb/2}\subseteq
  \psi^{\ironed}_{\mZ(\Rb)}(\trifil_{\bm{a},\,\Rb} \setminus
  \trifil_{\bm{a},\,R_0+2(\delta+\eta)})\) by Equation
  \eqref{eq:ironed_subset}. In particular, for any given
  \(0<\Sb<\Rb\), we can choose \(\delta\) sufficiently small to
  ensure that
  \begin{equation}\label{eq:ironed_outside}
    \mg^*\omega^{\ironed}_{\mZ(\Rb)} =
    \omega^{R_0}_{\mY(\Rb)} \qquad\text{on}
    \qquad\mY(\Rb)\setminus \mY(\Sb/2).
  \end{equation}
\end{example}

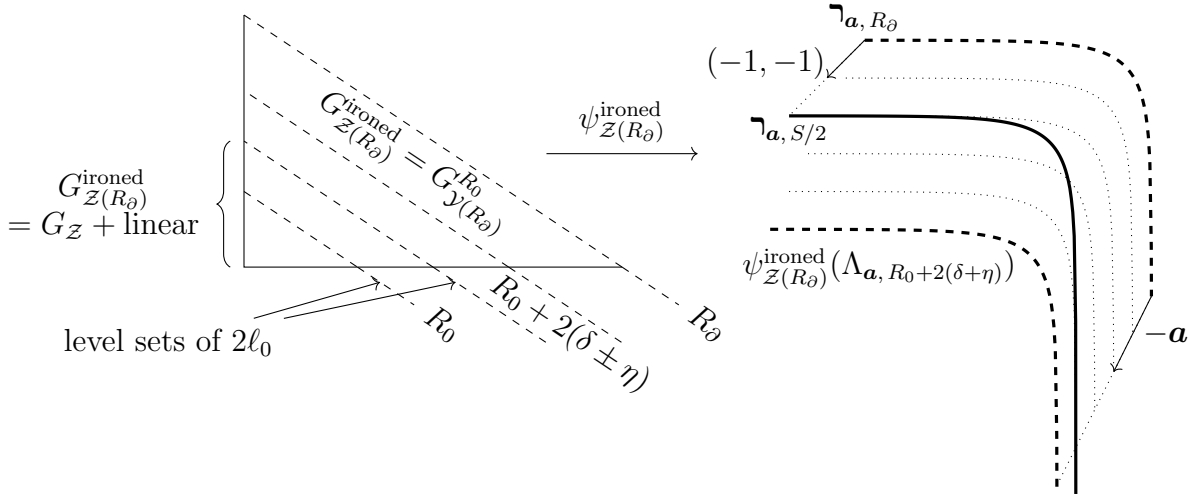
\begin{figure}[htb]
  \centering
  \begin{tikzpicture}
    \draw (0,10/3) -- (0,0) -- (5,0);
    \draw[dashed] (0,1) -- (1.5,0);
    \draw[dashed] (0,10/3) -- (5,0);
    \draw[dashed] (5,0) -- ++ (0.75,-0.5) node [pos=1.5,sloped] {\(\Rb\)};
    \draw[dashed] (2.5,0) -- ++ (-3.333*0.75,3.333*0.5);
    \draw[dashed] (3.5,0) -- ++ (-4.666*0.75,4.666*0.5);
    \draw[dashed] (1.5,0) -- ++ (0.75,-0.5) node [pos=1.5,sloped] {\(R_0\)};
    \draw[dashed] (2.5,0) -- ++ (2*0.75,-2*0.5);
    \draw[draw=none] (3,0) -- ++ (3.5*0.75,-3.5*0.5) node [midway,sloped] {\(R_0+2(\delta\pm\eta)\)};
    \draw[dashed] (3.5,0) -- ++ (2*0.75,-2*0.5);
    \draw[draw=none] (1.5,1.9) -- ++ (1.5,-1) node [midway,sloped] {\(G^{\ironed}_{\mZ(\Rb)}=G^{R_0}_{\mY(\Rb)}\)};
    \draw [decorate,decoration={brace,amplitude=5pt,mirror,raise=1ex}] (0,3.333*0.5) -- (0,0) node[midway,xshift=-4.5em,align=center]{\(G^{\ironed}_{\mZ(\Rb)}\)\\\(=G_{\mZ}+\text{linear}\)};
    \node (ls) at (-1,-1) {level sets of \(2\ell_0\)};
    \draw[->] (ls) -- (1.8,-0.2);
    \draw[->] (ls) -- (1.8+1,-0.2);
    \draw[->] (4,1.5) -- (6,1.5) node [midway,above] {\(\psi^{\ironed}_{\mZ(\Rb)}\)};
    \begin{scope}[shift={(12,3)}]
      \draw[dotted] (-3.791759,0) -- ++ (-1,-1);
      \draw[->] (-3.791759,0) -- ++ (-0.5,-0.5);
      \node at (-5.1,-0.3) {\((-1,-1)\)};
      \draw[dotted] (0,-3.38629) -- ++ (-1.25,-2.5);
      \draw[->] (0,-3.38629) -- ++ (-0.5,-1) node [midway,right] {\(-\bm{a}\)};
      \draw[very thick,dashed,variable=\x,samples=200,domain=-3.791759:-0.0000001] plot ({\x},{0.5*ln(1-exp(2*\x))}) -- (0,-3.38629);
      \node at (-3.8,0.3) {\(\resh_{\bm{a},\,\Rb}\)};
      \begin{scope}[shift={(-1,-1)}]
        \draw[very thick,variable=\x,samples=200,domain=-3.791759:-0.0000001] plot ({\x},{0.5*ln(1-exp(2*\x))}) -- (0,-5);
        \node at (-3.8,-0.2) {\(\resh_{\bm{a},\,\Sb/2}\)};
      \end{scope}
      \foreach \y in {1,2,3,4} {
        \begin{scope}[shift={({-0.25*\y},{-0.5*\y})}]
          \draw[dotted,variable=\x,samples=200,domain=-3.791759:-0.0000001] plot ({\x},{0.5*ln(1-exp(2*\x))}) -- (0,-3.38629);
        \end{scope}
      }
      \begin{scope}[shift={({-0.25*5},{-0.5*5})}]
        \draw[very thick,dashed,variable=\x,samples=200,domain=-3.791759:-0.0000001] plot ({\x},{0.5*ln(1-exp(2*\x))}) -- (0,-3.38629-0.1);
      \end{scope}
      \node at (-3.6,-2.6) [below] {\(\psi^{\ironed}_{\mZ(\Rb)}(\tri_{\bm{a},\,R_0+2(\delta+\eta)})\)};
    \end{scope}
  \end{tikzpicture}
  \caption{A schematic of the symplectic potential
    \(G^{\ironed}_{\mZ(\Rb)}\) when \(n=2\). On the right it shows the
    image of the moment polygon under the Legendre transform; the
    dotted lines are the images of level sets of \(2\ell_0\)
    strictly between \(R_0+2(\delta+\eta)\) and \(\Rb\), obtained by
    translating \(\resh_{\bm{a},\,\Rb}\) in the
    \(-\bm{a}\)-direction. The solid line is
    \(\resh_{\bm{a},\,\Sb/2}\) (obtained by translating
    \(\resh_{\bm{a},\,\Rb}\) in the \((-1,-1)\)-direction), to
    illustrate Equation
    \eqref{eq:ironed_subset}.}\label{fig:G_ironed_Z}
\end{figure}

\begin{lemma}\label{lma:weighted_blowdown}
  Given a complex manifold \(Z\) and a weighted blow up
  
  \begin{center}
    \begin{tikzpicture}
      \node (V) at (0,0) {\(\mY(\Rb)\)};
      \node (W) at (0,-2) {\(\mZ(\Rb)\)};
      \node (Y) at (2,0) {\(Y\)};
      \node (Z) at (2,-2) {\(Z\)};
      \draw[->] (V) -- (Y) node [midway,above] {\(\jmath\)};
      \draw[->] (V) -- (W) node [midway,left] {\(\mg\)};
      \draw[->] (W) -- (Z) node [midway,above] {\(\iota\)};
      \draw[->] (Y) -- (Z) node [midway,right] {\(g\)};
    \end{tikzpicture}
  \end{center}

  with weights \(\bm{a}\), suppose that \(Y\) admits a
  symplectic form \(\upsilon\) such that, for some values
  \(R_0,\Sb\in(0,\Rb)\) and positive coefficients \(r_1,r_2\),
  the form \(\jmath^*\upsilon\) coincides with
  \(r_1\omega^{R_0}_{\mY(\Rb)}+r_2\mg^*\omega_\mZ\) on
  \(\mY(\Sb)\). Then there exists a symplectic structure
  \(\omega\) on \(Z\) such that:
  \begin{itemize}
  \item[(1)] \(\omega\) tames the complex structure,
  \item[(2)] \(g^*\omega=\upsilon\) on the complement of
    \(\jmath(\mY(\Rb))\),
  \item[(3)] on \(\mZ(\Sb)\), the form \(\iota^*\omega\)
    coincides with
    \(r_1\omega^{\ironed}_{\mZ(\Rb)}+r_2\omega_{\mZ}\).
  \end{itemize}
\end{lemma}

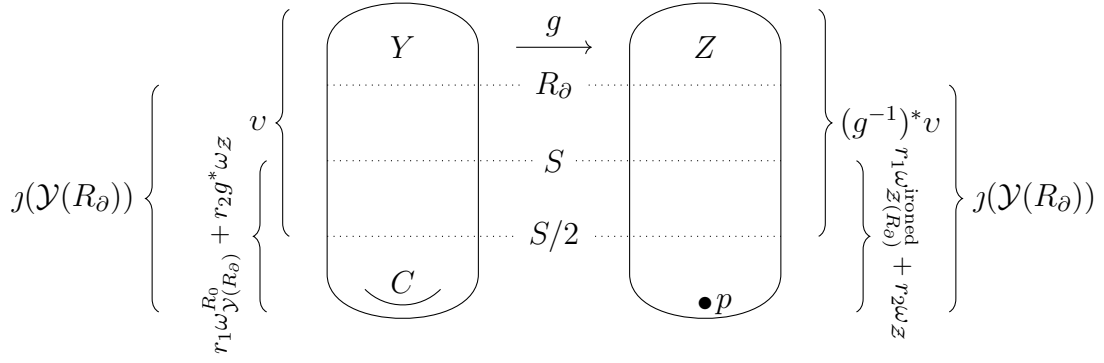
\begin{figure}[htb]
  \centering
  \begin{tikzpicture}
    \draw (0,0) -- (0,1.5) to[out=90,in=90] (2,1.5) -- (2,-1.5) to[out=-90,in=-90] (0,-1.5) -- cycle;
    \draw (4,0) -- (4,1.5) to[out=90,in=90] (6,1.5) -- (6,-1.5) to[out=-90,in=-90] (4,-1.5) -- cycle;
    \draw[dotted] (0,1) -- (6,1);
    \draw[dotted] (0,0) -- (6,0);
    \draw[dotted] (0,-1) -- (6,-1);
    \node[fill=white] at (3,1) {\(\Rb\)};
    \node[fill=white] at (3,0) {\(\Sb\)};
    \node[fill=white] at (3,-1) {\(\Sb/2\)};
    \draw [decorate,decoration={brace,amplitude=5pt,mirror,raise=4ex}] (-1.5,1) -- (-1.5,-2) node[midway,xshift=-4.5em]{\(\jmath(\mY(\Rb))\)};
    \draw [decorate,decoration={brace,amplitude=5pt,mirror,raise=4ex}] (7.5,-2) -- (7.5,1) node[midway,xshift=4.5em]{\(\jmath(\mY(\Rb))\)};
    \node at (1,1.5) {\(Y\)};
    \node at (5,1.5) {\(Z\)};
    \draw [decorate,decoration={brace,amplitude=5pt,mirror,raise=0ex}] (-0.5,2) -- (-0.5,-1) node[midway,xshift=-1em]{\(\upsilon\)};
    \draw [decorate,decoration={brace,amplitude=5pt,mirror,raise=0ex}] (-0.8,0) -- (-0.8,-2);
    \draw[draw=none] (-1,-2) -- (-1,-0.2) node[above,midway,sloped] {\footnotesize\(r_1\omega^{R_0}_{\mY(\Rb)}+r_2g^*\omega_{\mZ}\)};
    \draw [decorate,decoration={brace,amplitude=5pt,mirror,raise=0ex}] (6.5,-1) -- (6.5,2) node[midway,xshift=2.3em]{\((g^{-1})^*\upsilon\)};
    \draw [decorate,decoration={brace,amplitude=5pt,mirror,raise=0ex}] (7,-2) -- (7,0);
    \draw[draw=none] (7.2,-0.2) -- (7.2,-2) node[above,midway,sloped] {\footnotesize\(r_1\omega^{\ironed}_{\mZ(\Rb)}+r_2\omega_{\mZ}\)};
    \draw[->] (2.5,1.5) -- (3.5,1.5) node[midway,above] {\(g\)};
    \draw (0.5,-1.7) to[out=-45,in=-135] (1.5,-1.7);
    \node (p) at (5,-1.9) {\(\bullet\)};
    \node at (p) [right] {\(p\)};
    \node at (1,-1.9) [above] {\(C\)};
  \end{tikzpicture}
  \caption{A schematic for the proof of Lemma \ref{lma:weighted_blowdown}.}\label{fig:weighted_blowdown}
\end{figure}

\begin{proof}
  By Equation \eqref{eq:ironed_outside}, we can choose
  \(0<\eta<\delta\) sufficiently small to ensure that
  \(\mg^*\omega^{\ironed}_{\mZ(\Rb)}\) coincides with
  \(\omega^{R_0}_{\mY(\Rb)}\) on
  \(\mY(\Rb)\setminus\mY(\Sb/2)\). Define the symplectic form
  \(\omega\) on \(Z\) as follows:
  \begin{itemize}
  \item On \(Z\setminus\iota(\mZ(\Sb/2))\), we use
    \(\omega=(g^{-1})^*\upsilon\). In particular, on
    \(Y\setminus\jmath(\mY(\Rb))\), we have
    \(g^*\omega=\upsilon\), proving (2).
  \item On \(\iota(\mZ(\Sb))\), we use
    \((\iota^{-1})^*\left(r_1\omega^{\ironed}_{\mZ(\Rb)} +
      r_2\omega_\mZ\right)\).
  \end{itemize}
  On the overlap \(\mZ(\Sb)\setminus\mZ(\Sb/2)\), these forms
  coincide (Equation \eqref{eq:ironed_outside}). Since we are
  using K\"{a}hler forms to interpolate, \(\omega\) still tames
  the complex structure, proving (1).
\end{proof}

\subsection{Finding ellipsoids}

\begin{proposition}\label{prp:finding_ellipsoids}
  Let \(R_0,\Sb\in (0,\Rb)\). If \(0<\eta<\delta\) are chosen
  sufficiently small then the K\"{a}hler manifold
  \((\mZ(\Sb/2),r_1\omega^{\ironed}_{\mZ(\Rb)}+r_2\omega_\mZ)\)
  contains a symplectic ellipsoid \(\Ell(r_1R_0\bm{a}^{-1})\).
\end{proposition}
\begin{proof}
  We choose \(\delta\) and \(\eta\) small enough that
  \(\psi^{\ironed}_{\mZ(\Rb)}(\trifil_{\bm{a},\,R_0 +
    2(\delta+\eta)}) \subseteq \mem_{\bm{a},\,\Sb/2}\) as in
  Equation \eqref{eq:ironed_subset}. The toric K\"{a}hler form
  \(r_1\omega^{\ironed}_{\mZ(\Rb)}+r_2\omega_{\mZ}\) comes from
  the the K\"{a}hler potential\footnote{We work with the linear
    combination of K\"{a}hler potentials rather than the linear
    combination of symplectic potentials. Note that rescaling
    the symplectic potential does not preserve the condition
    that \(G(\bmm)=L(\ell(\bmm)-R_i/2)+\text{smooth }h(\bmm)\) along
    the boundary, and the Legendre transform is not additive
    anyway (the Legendre transform of a sum of functions is the
    {\em infimal convolution} of the individual Legendre
    transforms).} \[F=r_1F^{\ironed}_{\mZ(\Rb)}+r_2F_{\mZ}.\]
  The moment map \(\bmm\colon \mZ\to\RR_{\geq 0}^n\) is then
  given by
  \[\bmm=\frac{\partial
      F}{\partial\blr} =
    r_1\left(\psi^{\ironed}_{\mZ(\Rb)}\right)^{-1}(\blr) +
    r_2\left(\psi_{\mZ}\right)^{-1}(\blr).\]  
  Choose \(\bmm\in\trifil_{\bm{a},\,r_1R_0}\). We know that
  \(\left(\psi_{\mZ}\right)^{-1}(\blr)=\frac{1}{2}e^{2\blr}\),
  so
  \(\left(\psi^{\ironed}_{\mZ(\Rb)}\right)^{-1}(\blr) =
  \frac{1}{r_1}\left(\bmm -
    \frac{r_2}{2}e^{2\blr}\right)\). Writing \(\bmm' =
  \frac{1}{r_1}\left(\bmm - \frac{r_2}{2}e^{2\blr}\right)\), we
  have
  \[\bm{a}\cdot\bmm' =
    \frac{\bm{a}\cdot\bmm}{r_1}-\frac{r_2}{2r_1}\bm{a}\cdot
    e^{2\blr} \leq \frac{R_0}{2},\] since
  \(\bmm\in\trifil_{\bm{a},\,r_1R_0}\). This means that
  \(\bmm'\in\trifil_{\bm{a},\,R_0}\subseteq
  \trifil_{\bm{a},\,R_0+2(\delta+\eta)}\), so
  \(\blr \in \psi^{\ironed}_{\mZ(\Rb)}\left(\trifil_{\bm{a},\,R_0
      + 2(\delta+\eta)}\right)\), which is contained in
  \(\mem_{\bm{a},\,S/2}\) by our choice of \(\delta\) and
  \(\eta\). This shows that the moment image of \(\mZ(S/2)\)
  contains \(\trifil_{\bm{a},\,R_0}\), so that \(\mZ(S/2)\)
  contains a symplectically embedded copy of
  \(\Ell(r_1R_0\bm{a}^{-1})\).
\end{proof}

\subsection{Proof of Theorem \ref{thm:ellipsoids_galore}}
\label{sct:kaehler_proof}
We begin with some preliminary results.

\begin{proposition}\label{prp:ellipsoids_1}
  Given a complex manifold \(Z\) and a weighted blow up
  
  \begin{center}
    \begin{tikzpicture}
      \node (V) at (0,0) {\(\mY(\Rb)\)};
      \node (W) at (0,-2) {\(\mZ(\Rb)\)};
      \node (Y) at (2,0) {\(Y\)};
      \node (Z) at (2,-2) {\(Z\)};
      \draw[->] (V) -- (Y) node [midway,above] {\(\jmath\)};
      \draw[->] (V) -- (W) node [midway,left] {\(\mg\)};
      \draw[->] (W) -- (Z) node [midway,above] {\(\iota\)};
      \draw[->] (Y) -- (Z) node [midway,right] {\(g\)};
    \end{tikzpicture}
  \end{center}

  with weights \(\bm{a}\), suppose that \(Y\) admits a
  symplectic form \(\upsilon\) such that:
  \begin{itemize}
  \item \(\upsilon\) tames the complex structure, and
  \item the pullback of \(\upsilon\) to \(C=\exc(g)\) is
    cohomologous to the toric symplectic form coming from the
    moment polytope \(\tri_{\bm{a},\,R}\) for some \(R\) (not
    necessarily smaller than \(\Rb\)).
  \end{itemize}
  Then there exists a symplectic structure \(\omega\) on \(Z\)
  such that:
  \begin{itemize}
  \item[(1)] \(\omega\) tames the complex structure,
  \item[(2)] \(g^*\omega=\upsilon\) on the complement of
    \(\jmath(\mY(\Rb))\),
  \item[(3)] \((Z,\omega)\) admits a symplectic embedding of the
    ellipsoid \(\Ell(R\bm{a}^{-1})\).
  \end{itemize}
\end{proposition}
\begin{proof}
  Choose a \(0<R_0<\Rb\) and let \(r_1\) be such that
  \(r_1R_0=R\). Then
  \(r_1\omega^{R_0}_{\mY(\Rb)}\) is a K\"{a}hler form on
  \(\mY(\Rb)\) whose pullback to \(\mC=\exc(\mg)\) is
  cohomologous to that of \(\jmath^*\upsilon\). Note that we
  can't simply use \(\omega^{R}_{\mY(\Rb)}\) since the
  construction of \(\omega^{R}_{\mY(\Rb)}\) required
  \(R<\Rb\) and we don't know if that holds.
  
  Now we appeal to Lemma \ref{lma:ironing} and iron \(\upsilon\)
  on \(\jmath(\mY(\Rb))\) to get a new taming form
  \(\upsilon'\) such that:
  \begin{itemize}
  \item \(\upsilon'\) coincides with \(\upsilon\) on the
    complement of \(\jmath(\mY(\Rb))\), and
  \item the pullback \(\jmath^*\upsilon'\) has the property
    that, for some \(\Sb\) and \(r_2\), its restriction to
    \(\mY(\Sb)\) coincides with
    \(r_1\omega^{R_0}_{\mY(\Rb)}+r_2\mg^*\omega_\mZ\).
  \end{itemize}
  By Lemma \ref{lma:weighted_blowdown}, we can equip \(Z\) with
  a blown-down symplectic form satisfying (1) and (2) and which
  contains a copy of
  \((\mZ(\Sb/2),r_1\omega^{\ironed}_{\mZ(\Rb)}+r_2\omega_{\mZ})\). By
  Proposition \ref{prp:finding_ellipsoids}, this contains a copy
  of \(\Ell(r_1R_0\bm{a}^{-1})=\Ell(R\bm{a}^{-1})\) as required.
\end{proof}

\begin{corollary}\label{cor:ellipsoids_2}
  Let \(Z\) be a smooth complex projective variety of dimension
  \(n\) equipped with a K\"{a}hler form \(\zeta\), and let
  \(g\colon Y\to Z\) be a weighted blow-up with weights
  \(\bm{a}\) and exceptional locus \(\exc(g)=C\subseteq
  Y\). Suppose that \(Y\) admits a K\"{a}hler form \(\ups\) such
  that:
  \begin{itemize}
  \item the pullback of \(\upsilon\) to \(C=\exc(g)\) is
    cohomologous to the toric symplectic form coming from the
    moment polytope \(\tri_{\bm{a},\,R}\) for some \(R\),
    and
  \item the restrictions of
    \(\ups\) and \(g^*\zeta\) to \(Y\setminus C\) are
    cohomologous.
  \end{itemize}
  Then \((Z,\zeta)\) admits a symplectic embedding of the
  ellipsoid \(\Ell(R\bm{a}^{-1})\).
\end{corollary}
\begin{proof}
  By Proposition \ref{prp:ellipsoids_1}, \(Z\) admits a
  symplectic form \(\omega\) taming the complex structure such
  that \((Z,\omega)\) contains a copy of the desired
  ellipsoid. But since \(\zeta\) and \(\omega\) both tame the
  same complex structure, the 2-forms \(t\omega+(1-t)\zeta\) are
  all symplectic. Their restrictions to the complement of the
  ellipsoid where the blow-up is happening are cohomologous by
  assumption, and the ellipsoid is contractible, so they are
  cohomologous on \(Z\). We can then apply Moser's trick to this
  path of symplectic forms; we see that they are
  symplectomorphic. Therefore \((Z,\omega)\) also contains a
  copy of \(\Ell(R\bm{a}^{-1})\).
\end{proof}

We finally explain how Theorem \ref{thm:ellipsoids_galore}
follows from Corollary \ref{cor:ellipsoids_2}.

\begin{proof}[Proof of Theorem \ref{thm:ellipsoids_galore}]
  If \(\varepsilon<\varepsilon(Z,D;g)\) then
  \(g^*D-\varepsilon C\) is an ample \(\QQ\)-divisor on \(Y\),
  where \(C=\exc(g)\). Let \(\upsilon\) be the associated
  K\"{a}hler form Poincar\'{e} dual to \(\pi(g^*D-\epsilon
  C)\). Certainly \(\upsilon\) and \(g^*\zeta\) are cohomologous
  on the complement of \(C\). The cohomology class of
  \(\upsilon\) restricted to a neighbourhood of \(C\) is
  Poincar\'{e}-dual to \(-\pi\epsilon C\). The divisor \(C\) is
  isomorphic to the weighted projective space \(\PP(\bm{a})\);
  this is the normal toric variety associated to the fan of
  inward normals of the polytope \(\tri_{\bm{a},\,R}\) for any
  \(R\). This polytope has \(\binom{n}{2}\) edges
  \(\edge_{ij}\), \(1\leq i,j\leq n\), of affine length
  \(\frac{R}{2\lcm(a_i,a_j)}\), and corresponding toric boundary
  curves \(T_{ij}\). Working in the local model for the toric
  weighted blow-up, Lemma \ref{lma:C_intersection} implies that
  \(C\cdot T_{ij} = - \frac{1}{\lcm(a_i,a_j)}\). The cohomology
  class \(-\pi\epsilon C\) of the symplectic form on this
  neighbourhood is determined by its integral over any of the
  curves \(T_{ij}\), namely
  \[\int_{T_{ij}}\upsilon = \frac{\pi\epsilon}{\lcm(a_i,a_j)}.\]
  Consider the toric symplectic form \(\upsilon_{\toric}\) on
  the local model of the toric weighted blow-up, then we get
  \[\int_{T_{ij}}\upsilon_{\toric}=
    \frac{2\pi R}{2\lcm(a_i,a_j)}.\] This means that, upon
  restricting to a neighbourhood of \(C\),
  \[[\upsilon] = \frac{\epsilon}{R}[\upsilon_{\toric}],\] so
  these cohomology classes match precisely when
  \(R=\epsilon\). The result then follows from Corollary
  \ref{cor:ellipsoids_2}.
\end{proof}

\section{Weighted Seshadri constants for surfaces}
\label{sct:surfaces}
We now focus on the case of complex surfaces. We will free up
some subscripts by switching from writing \(\bm{a}=(a_1,a_2)\)
to writing \(\bm{a}=(\alpha,\beta)\) for the weights of a
weighted blow-up.

\subsection{Local model}

Let \(\mg\colon\mY\to\mZ\) be the toric weighted blow-up of
\(\mZ=\CC^2\) at the origin with weights
\((\alpha,\beta)\). This is associated to the fan \(\Sigma\)
with three rays:
\[\RR_{\geq 0}(\alpha,\beta),\qquad \RR_{\geq 0}(0,1),\qquad
  \RR_{\geq 0}(1,0).\] The fan \(\Sigma\) is the inward normal
fan to the polygon (see Figure \ref{fig:fan_and_polygon}):
\[\Polytope_\mY(R_0)=\{(\mm_1,\mm_2)\in\RR^2_{\geq 0}\,:\,
  \alpha\mm_1+\beta\mm_2\geq R_0/2\}.\]

\begin{figure}[htb]
  \centering
  \begin{tikzpicture}
    \node at (-1,3) {\(\Sigma'\)};
    \draw (0,0) -- (3,0);
    \draw (0,0) -- (0,3);
    \draw[->] (0,0) -- (1,0) node [below] {\(\rho_0\)};
    \draw[->] (0,0) -- (0,1) node [left] {\(\rho_{m+1}\)};
    \draw[->] (0,0) -- (2,3) node [above] {\(\rho_{i_1}\)};
    \draw (0,0) -- (1,1) node [right] {\(\rho_1\)};
    \draw[dashed] (0,0) -- (2,2.5) node [right] {\(\vdots\)};
    \draw[dashed] (0,0) -- (1.5,3);
    \draw[draw=none] (0.7,2.4) -- (2,3.4) node[above,midway,sloped] {\(\cdots\)};
    \draw (0,0) -- (0.6,2) node [above] {\(\rho_m\)};
    \begin{scope}[shift={(5,0)}]
      \node at (-1,3) {\(\Sigma\)};
      \draw (0,0) -- (3,0);
      \draw (0,0) -- (0,3);
      \draw[->] (0,0) -- (1,0) node [below] {\(\bm{e}_1\)};
      \draw[->] (0,0) -- (0,1) node [left] {\(\bm{e}_2\)};
      \draw[->] (0,0) -- (2,3) node [above] {\(\bm{a}=(\alpha,\beta)\)};
    \end{scope}
    \begin{scope}[shift={(10,0)}]
      \node at (-1,3) {\(\Polytope_{\mY(R_0)}\)};
      \fill[lightgray] (0,3) -- (0,1) -- (1.5,0) -- (3,0) -- (3,3) -- cycle;
      \draw (0,3) -- (0,1) -- (1.5,0) -- (3,0);
      \node (pt) at (0,1) {\(\bullet\)};
      \node (qt) at (1.5,0) {\(\bullet\)};
      \node at (pt) [left] {\(\pt\)};
      \node at (qt) [below] {\(\qt\)};
      \node at (0.75,0.5) [below left] {\(\tri_{\bm{a},\,R_0}\)};
    \end{scope}
  \end{tikzpicture}
  \caption{The polygon \(P_{\mY}(R_0)\), its inward normal fan
    \(\Sigma\), and the subdivision \(\Sigma'\) of \(\Sigma\)
    corresponding to the minimal resolution
    \(\mf\colon \mX\to \mY\).}\label{fig:fan_and_polygon}
\end{figure}

The associated toric surface \(\mY\) has (up to) two orbifold
points living over the vertices of \(P_\mY(R_0)\) and contains a
curve \(\mC\) connecting these points. The self-intersection of
\(\mC\) is \(\mC^2 = -1/(\alpha\beta)\) by Lemma
\ref{lma:C_intersection}.

It will sometimes be convenient to work on the minimal
resolution \(\mf\colon \mX\to \mY\). This is the toric surface
associated to a subdivision \(\Sigma'\) of \(\Sigma\), see
{\cite[Section 2.6]{Fulton}}. This fan has rays
\(\rho_0,\rho_1,\ldots,\rho_{m+1}\), ordered anticlockwise,
where \(\rho_0=\RR_{\geq 0}\bm{e}_1\),
\(\rho_{i_1}=\RR_{\geq 0}\bm{a}\) and
\(\rho_{m+1}=\RR_{\geq 0}\bm{e}_2\) (see Figure
\ref{fig:fan_and_polygon}). The rays \(\rho_1,\ldots,\rho_m\)
correspond to embedded rational curves \(\mE_1,\ldots,\mE_m\) in
\(\mX\). Write \(\bm{u}_i=(\alpha_i,\beta_i)\) for the primitive
integer vector pointing along the ray \(\rho_i\) and recall the
operation \(\bm{x}\wedge\bm{y}=x_1y_2-x_2y_1\) on vectors
\(\bm{x},\bm{y}\in\RR^2\). We have
\(\bm{u}_i\wedge \bm{u}_{i+1} = 1\) for \(i=0,\ldots,m\) and
\(\mE_i^2 = -b_i\) where \(b_i=\bm{u}_{i-1}\wedge \bm{u}_{i+1}\)
for \(i=1,\ldots,m\). We have
\begin{gather}\label{eq:alphas_betas}
  \frac{\beta_i}{\beta_{i-1}}=[b_{i-1},\ldots,b_1],\qquad
  \frac{\alpha_i}{\alpha_{i+1}}=[b_{i+1},\ldots,b_m],\\
  \nonumber \text{where
  }[a,b,c,\ldots]=a-\frac{1}{b-\frac{1}{c-\cdots}}.
\end{gather}
Note that the proper transform \(\mf^{-1}_*\mC\) is
\(\mE_{i_1}\) and
\((\alpha_{i_1},\beta_{i_1})=(\alpha,\beta)\). The curves
\(E_1,\ldots,E_{i_1-1}\) (respectively \(E_{i_1+1},\ldots,E_m\))
project to the point labelled \(\pt\) (respectively \(\qt\)) in
Figure \ref{fig:fan_and_polygon}.

\begin{remark}
  As usual, we will use script/nonscript letters to denote the
  local/non-local versions of this, so if \(g\colon Y\to Z\) is
  a weighted blow-up and \(f\colon X\to Y\) is its minimal
  resolution then we will write \(E_1,\ldots,E_m\) and \(C\) for
  the exceptional curves of \(h\coloneqq g\circ f\) and \(g\)
  respectively, and \(E_{i_1}=f^{-1}_*C\).
\end{remark}

\begin{figure}[htb]
  \centering
  \begin{tikzpicture}
    \node (a) at (0,0) {\(\bullet\)};
    \node (b) at (1,0) {\(\bullet\)};
    \node (c) at (2,0) {\(\cdots\)};
    \node (d) at (3,0) {\(\bullet\)};
    \node (e) at (4,0) {\(\cdots\)};
    \node (f) at (5,0) {\(\bullet\)};
    \draw (a) -- (b);\draw (b) -- (c); \draw (c) -- (d);
    \draw (d) -- (e); \draw (e) -- (f);
    \node at (a) [above] {\(E_1\)};
    \node at (b) [above] {\(E_2\)};
    \node at (d) [above] {\(E_{i_1}\)};
    \node at (f) [above] {\(E_m\)};
    
    \node at (a) [below] {\(-b_1\)};
    \node at (b) [below] {\(-b_2\)};
    \node at (d) [below] {\(-b_{i_1}\)};
    \node at (f) [below] {\(-b_m\)};
  \end{tikzpicture}
  \caption{The exceptional curves of the minimal resolution of a
    weighted blow-up, with their self-intersections below. The
    curve \(E_{i_1}\) is the proper transform of
    \(C\).}\label{fig:resolution_chain}
\end{figure}

\begin{remark}\label{rmk:minus_1}
  Note that since \(f\colon X\to Y\) is a minimal resolution, we
  have \(E_i^2\leq -2\) for all \(i\neq i_1\). Since \(Z\) is
  smooth at \(p\), \(h=g\circ f\) is {\em not} a minimal
  resolution (there's nothing to resolve) which means that
  \(E_{i_1}^2=-1\).
\end{remark}

\subsection{The Farey tree}
\label{sct:farey}

Even if \(Z\) is toric and \(p\) is a toric fixed point, not all
weighted blow-ups at \(p\) are globally toric: there is a lot of
freedom in forming weighted blow-ups, and, in this section, we
will discuss the moduli space of weighted blow-ups at
\(p\).

\begin{discussion}
  Since \(E_{i_1}\) is a \(-1\)-curve (Remark
  \ref{rmk:minus_1}), we may contract it. Then one of the
  neighbouring curves \(E_{i_2}\) must become a
  \(-1\)-curve. Contracting \(E_{i_2}\) and continuing in this
  manner, we contract all the remaining curves in some order
  \(E_{i_3},\ldots,E_{i_m}\) until we get back to
  \(Z\). Reversing this argument, we see that any weighted
  blow-up can be obtained by iteratedly blowing up at points
  infinitely near to \(p\): write \(F_1=E_{i_m}\) for the first
  exceptional curve, then blow up a point on \(F_1\) t obtain an
  exceptional curve \(F_2=E_{i_{m-1}}\), then blow up a point on
  \(F_2\) to obtain an exceptional curve \(F_3=E_{i_{m-2}}\),
  and so on. Provided we end up with a linear chain of curves,
  we will call such a blow-up process a {\em Farey process}.

  Ensuring the curves remain in a linear chain constrains which
  blow-ups we can perform. We call the intersection points
  between curves in the chain {\em satellite points}. If the
  most recent exceptional curve \(F_i\) has two satellite points
  \(p_i\) and \(q_i\) then, in order that our exceptional set
  remains a linear chain, we can only blow up at a satellite
  point. If, however, \(F_i\) has only one satellite point then
  we can either blow up at this satellite point \(p_i\) or else
  at any other ({\em free}) point \(q_i\) of \(F_i\). Once we
  have blown up at a satellite point, we encounter no further
  free points, but there is always a choice between two
  satellite points. We think of this as a decision tree (see
  Figure \ref{fig:farey}).
\end{discussion}

\begin{figure}[htb]
  \centering
  \begin{tikzpicture}
    \begin{scope}[shift={(0,0)}]
      \draw (-0.2,0) -- (1.2,0.5) node [pos=0.3] (a) {};
      \node at (-0.2,0.25) {\(F_1\)};
      \node at (a) [below] {\(-2\)};
      \draw[very thick] (0.8,0.5) -- (2.2,0) node [pos=0.7] (b) {};
      \node at (b) [above] {\(F_2\)};
      \node at (1.5,0.3) [below left] {\(-1\)};
      \node (p2) at (1,0.4) {\(\bullet\)};
      \node (q2) at (2,0.05) {\(\circ\)};
      \node at (1,0.45) [above] {\(p_2\)};
      \node at (2,0.05) [below] {\(q_2\)};
    \end{scope}
    \begin{scope}[shift={(2.5,2)}]
      \draw (-0.2,0) -- (1.2,0.5) node [pos=0.3] (c1) {};
      \node at (c1) [above] {\(F_1\)};
      \node at (c1) [below] {\(-2\)};
      \draw (0.8,0.5) -- (2.2,0) node [pos=0.6] (c2) {};
      \node at (c2) [above] {\(F_2\)};
      \node at (c2) [below left] {\(-2\)};
      \draw[very thick] (1.8,0) -- (3.2,0.5) node [pos=0.5] (c3) {};
      \node at (c3) [above] {\(F_3\)};
      \node at (c3) [below right] {\(-1\)};
      \node (p3b) at (2,0.05) {\(\bullet\)};
      \node at (3,0.42) {\(\circ\)};
      \node at (2,0.05) [below] {\(p_3\)};
      \node (q3b) at (3,0.45) [above] {\(q_3\)};
    \end{scope}
    \begin{scope}[shift={(-3,2)}]
      \draw (-0.2,0) -- (1.2,0.5) node [pos=0.3] (d1) {};
      \draw[very thick] (0.8,0.5) -- (2.2,0) node [pos=0.6] (d2) {};
      \draw (1.8,0) -- (3.2,0.5) node [pos=0.5] (d3) {};
      \node at (d1) [above] {\(F_1\)};
      \node at (d1) [below] {\(-3\)};
      \node at (d2) [above] {\(F_3\)};
      \node at (d2) [below left] {\(-1\)};
      \node at (d3) [above] {\(F_2\)};
      \node at (d3) [below right] {\(-2\)};
      \node at (1,0.4) {\(\bullet\)};
      \node (q3a) at (2,0.05) {\(\bullet\)};
      \node (p3a) at (1,0.4) [above] {\(p_3\)};
      \node at (2,0.05) [below] {\(q_3\)};
    \end{scope}
    \draw[->] (d2) to[out=-90,in=180] (p2);
    \node at (-1,1) [left] {satellite};
    \draw[->] (c3) to[out=-90,in=0] (q2);
    \node at (4.7,1) [right] {free};
    \draw[dashed,->] (-3,4) -- (p3a) node [pos=0.5,left] {sat.};
    \draw[dashed,->] (-0.5,4) -- (q3a) node [pos=0.3,right] {sat.};
    \draw[dashed,->] (4,4) -- (p3b) node [pos=0.3,left] {sat.};
    \draw[dashed,->] (6,4) -- (q3b) node [midway,right] {free};
    \begin{scope}[shift={(0.5,-2)}]
      \draw[very thick] (-0.2,0) -- (1.2,0) node [pos=0.4] (f1) {};
      \node at (f1) [above] {\(F_1\)};
      \node at (f1) [below] {\(-1\)};
      \node (q1) at (0.8,0) {\(\circ\)};
    \end{scope}
    \draw[->] (1.5,0.15) -- (q1) node [midway,right] {free};
    \begin{scope}[shift={(9,-2)}]
      \node (11) at (-1,0) {\((1,1)\)};
      \node (21) at (0,2) {\((1,2)\)};
      \node (31) at (1,4) {\((1,3)\)};
      \node (32) at (-1,4) {\((2,3)\)};
      \draw[->] (21) -- (11);
      \draw[->] (31) -- (21);
      \draw[->] (32) -- (21);
      \draw[dashed,->] (-2,6) -- (32);
      \draw[dashed,->] (-1,6) -- (32);
      \draw[dashed,->] (1,6) -- (31);
      \draw[dashed,->] (2,6) -- (31);
    \end{scope}
  \end{tikzpicture}
  \caption{Left: The choices involved in making a weighted
    blow-up. Satellite points are drawn solid \(\bullet\), free
    points as empty circles \(\circ\). The weighted blow-up is
    obtained by contracting all the curves \(F_i\) except the
    final one (drawn thickly). Right: The Farey tree.}
  \label{fig:farey}
\end{figure}
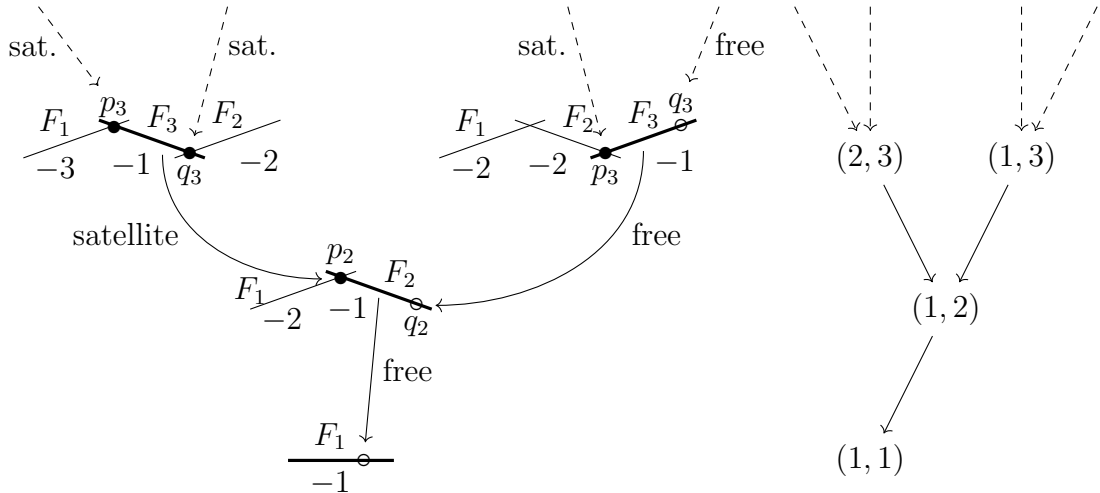

\begin{definition}[Farey tree]
  If we label the nodes of this decision tree with the weights
  of the corresponding weighted blow-up, we obtain the
  well-known {\em Farey tree} of primitive integer vectors; see
  Figure \ref{fig:farey}. The easiest way to reconstruct these
  labels is to first label each node by {\em three} integer
  vectors \(\bm{u},\bm{v},\bm{w}\), starting with the root which
  is labelled \((1,0),(1,1),(0,1)\). Moving up from the node
  \(\bm{u},\bm{v},\bm{w}\), if we move rightward in the decision
  tree, the next node is labelled
  \(\bm{v},\bm{v}+\bm{w},\bm{w}\) and if we move leftward then
  the next node is labelled
  \(\bm{u},\bm{u}+\bm{v},\bm{v}\). Then if a node is labelled
  \(\bm{u},\bm{v},\bm{w}\), we omit the \(\bm{u}\) and
  \(\bm{w}\) and just keep the label \(\bm{v}\). Here, we think
  of the first step as a rightward step, so the first node above
  the root is labelled \((1,2)\).
\end{definition}

\begin{remark}
  We will always write our chains \(E_1,\ldots,E_m\) from left
  to right and think of satellite points being ordered with
  \(p_i\) on the left and \(q_i\) on the right (so
  \(p_i\in E_{k-1}\cap E_k\) and \(q_i\in E_k\cap E_{k+1}\) for
  some \(k\)).
\end{remark}

\begin{definition}\label{dfn:guided}
  The choices of free points in a Farey process can be fixed
  ahead of time by picking the germ of an irreducible analytic
  curve \(A\) passing through \(p\). At the first step, we blow
  up the intersection of \(F_1\) with the proper
  transform\footnote{We will abusively write \(A\) for any
    subsequent proper transform of \(A\).}  of \(A\). At the
  second step, we take \(q_2\) to be the unique intersection
  point \(A\cap F_2\). If we blow up at \(q_2\) and write \(A\)
  for the proper transform of \(A\), we can take \(q_3\) to be
  the unique intersection point \(A\cap F_3\), and so on. We
  call this the {\em Farey process guided by \(A\)}.
\end{definition}

\begin{remark}\label{rmk:guided_exists}
  Given a Farey process, we can always choose the analytic curve
  germ\footnote{This germ is purely local: it does not need to
    be part of a closed holomorphic curve in \(X\).} at the end
  passing transversely through a non-satellite point of the
  right-most curve \(E_m\) in the chain. Contracting down to
  \(Z\) we obtain a smooth, irreducible curve germ \(A\) whose
  guided Farey process coincides with the given process.
\end{remark}

\begin{example}\label{exm:guided_wb}
  \begin{itemize}
  \item[(1)] The toric weighted blow-up is guided by (either)
    one of the toric boundary components passing through \(p\).
  \item[(2)] Let \(Z=\CP^2\), let \(p=[0:0:1]\), and take
    \(A\) to be the nodal cubic curve
    \(\{[x:y:z]\,:\,x^3+y^3=xyz\}\). Under the Farey process of
    weight \((\alpha,\beta)\) guided by (either) one of the
    branches of \(A\), the proper transform of \(A\) is
    always an embedded, irreducible rational curve which
    intersects \(E_1\) and \(E_m\) once transversely. This fails
    to be toric if \(\beta/\alpha>2\): if we move right once in
    the Farey tree then the point \(q_2\) where the proper
    transform of \(A\) intersects \(F_2\) is not a toric
    fixed point, and the blow-up at \(q_2\) is not toric.
  \end{itemize}
\end{example}

\begin{lemma}\label{lma:local_model}
  Suppose \(Z\) is a smooth complex projective surface and let
  \(g\colon Y\to Z\) be a morphism of projective surfaces whose
  exceptional locus is an irreducible rational curve \(C\). Let
  \(f\colon X\to Y\) be the minimal resolution. If the
  exceptional locus of \(h=g\circ f\) is a chain of curves
  \(E_1,\ldots,E_m\) then \(g\) is a weighted blow-up.
\end{lemma}
\begin{proof}
  Note that \(f_*^{-1}C\) is an embedded \(-1\)-curve in \(X\)
  or else \(f\colon X\to Y\) is not the minimal resolution. We
  can contract \(-1\)-curves on \(X\) one at a time starting
  with \(f^{-1}_*C\) until we get to \(Z\), so \(X\) is obtained
  from \(Z\) by a Farey process. Let \(A\) be a smooth curve
  germ through \(p\) guiding this Farey process as in Remark
  \ref{rmk:guided_exists}. Since \(A\) is smooth, in suitable
  local analytic coordinates \(\iota\colon \mZ(R)\to Z\) we can
  take it to be \(y=0\). But since this is part of the toric
  boundary in \(\mZ(R)\), it guides the toric blow-up of
  \(\mZ(R)\).
\end{proof}

\begin{remark}
  In fact, we do not need to assume that \(Y\) is projective:
  since \(X\) is obtained by iterated blow-up from \(Z\), we see
  that \(X\) is projective. Now let \(D\) be an ample divisor on
  \(Z\). The divisor \(L=f^*(g^*D-\varepsilon C)\) satisfies
  \(L\cdot E_j=0\) for \(j\neq i_1\) and
  \(L\cdot E_{i_1}=\varepsilon/(\alpha\beta)>0\), so for large
  \(N\), the linear system \(|NL|\) defines a morphism to
  \(\PP(|NL|)\) which factors as
  \(X\stackrel{f}{\to} Y\stackrel{i}{\to}\PP(|NL|)\) for some
  projective embedding \(i\) of \(Y\).
\end{remark}

\begin{remark}\label{rmk:quasimonomial}
  If \(Z=\CP^2\) and we are {\em given} a Zariski-open affine
  coordinate chart (not just analytic) then we can rotate so
  that the tangent space of any smooth guiding curve \(A\) at
  \(p=[0:0:1]\) is \(y=0\). Then the germ of \(A\) is equivalent
  to the curve \(y=\xi(x)\) for some analytic function \(\xi\).
\end{remark}

\begin{remark}
  We will usually be more interested in ``sufficiently general''
  Farey processes, where any free choice is made at a
  ``sufficiently general'' point of \(F_i\). Recall that the
  divisor \(F_m\subseteq X\) gives a divisorial valuation on the
  function field of \(Z\), and there is a notion of ``very
  general quasimonomial valuation'' (see \cite{DumnickiEtAl}),
  which means that the coefficients of the power series
  \(\xi(x)\) from Remark \ref{rmk:quasimonomial} are chosen to
  lie in the complement of a countable union of algebraic
  sets. In practice, however, the following easily checkable
  condition will suffice.
\end{remark}
\begin{definition}\label{dfn:general_adapted}
  Consider a Farey process whose result is
  \(X \xrightarrow{f} Y\xrightarrow{g} Z\) and let
  \(h=g\circ f\). Given a rational curve \(A\subseteq Z\) we
  write \(\dbtilde{A}=h^{-1}_*A\) for its proper transform. We
  say that this Farey process is {\em general with respect to
    \(A\)} if \(\dbtilde{A}\) is an embedded rational curve with
  \(\dbtilde{A}^2\geq -1\). We say it is {\em adapted to \(A\)}
  if it is general with respect to \(A\) and
  \(\dbtilde{A}^2=-1\).
\end{definition}

\begin{remark}\label{rmk:general_adapted}
  If we were to deform the arbitrary choices of blow-up points
  slightly then the homology class \([\dbtilde{A}]\) would still
  contain an embedded rational curve, so this condition is stable
  under small perturbations, hence ``general''. Adaptedness is a
  condition that will be useful when we are trying to maximise the
  weighted Seshadri constant \(\varepsilon(Z,A;g)\).
\end{remark}

\begin{example}\label{exm:nodal_cubic}
  Let \(A\) be a nodal cubic curve from Example
  \ref{exm:guided_wb}(2). If we perform a Farey process guided
  by \(A\) then the self-intersection of the curve
  \(\dbtilde{A}\) is equal to \(6-r\) where \(r\) is the number
  of times we move right in the Farey tree before moving
  left. Equivalently, this Farey process is general with respect
  to \(A\) provided \(\beta/\alpha\leq 7\) and adapted to \(A\)
  if \(6<\beta/\alpha\leq 7\). By contrast, the toric Farey
  process is adapted to \(A\) if and only if
  \(\beta/\alpha\leq 2\).
\end{example}

\subsection{Computing weighted Seshadri constants}

\begin{lemma}\label{lma:positive_gammas}
  Let \(X\stackrel{f}{\to} Y\stackrel{g}{\to} Z\) be the minimal
  resolution \(f\) of a weighted blow-up \(g\) of \(Z\) at a smooth
  point \(p\). Let \(C=\exc(g)\) be the exceptional locus. Define the
  numbers \(\gamma_i\) for \(i=1,\ldots,m\) by
  \begin{equation}\label{eq:gammas}
    f^*C=f^{-1}_*C+\sum_{i\neq i_1}\gamma_iE_i=\sum_{i=1}^m\gamma_iE_i,
  \end{equation}
  with \(\gamma_{i_1}=1\). The numbers \(\gamma_i\) are all
  positive. Indeed, they are given by
  \begin{equation}\label{eq:gammas_formula}
    \gamma_i=\begin{cases}\beta_i/\beta_{i_1}&\text{if
    }i<i_1,\\ 1& \text{if }i=i_1,\\\alpha_i/\alpha_{i_1}&\text{if
    }i>i_1,\end{cases}
  \end{equation}
  where \(\alpha_i\) and \(\beta_i\) are defined by Equation
  \eqref{eq:alphas_betas}.
\end{lemma}
\begin{proof}
   We can compute these numbers as follows. Let \(M_{ij}=E_i\cdot
   E_j\) be the intersection matrix of the curves \(E_i\), let
   \(\phi_j=E_j\cdot f^*C\) and let
   \(\bm{\gamma}=(\gamma_1,\ldots,\gamma_m)^T\) and
   \(\bm{\phi}=(\phi_1,\ldots,\phi_m)^T\). Intersecting both sides of
   Equation \eqref{eq:gammas} with \(E_j\), we find that
   \(\bm{\phi}=M\bm{\gamma}\). Since \(f_*E_j=0\) if \(j\neq i_1\) and
   \(f_*E_{i_1}=C\), we have
  \[\phi_j=E_j\cdot f^*C=\begin{cases}0&\text{if }j\neq i_1\\ C^2&\text{if }j=i_1\end{cases}.\] So \(\bm{\gamma}\) is
  \(C^2\) times the \(i_1\)th column of the matrix \(M^{-1}\). The
  entries of \(M^{-1}_{ij}\) are given by
  \[M^{-1}_{ij}=\begin{cases}-\beta_i\alpha_j &\text{if }i\leq
      j\\ -\beta_j\alpha_i&\text{if }i>j\end{cases}\] (see
  {\cite[Proposition 2.1]{OrevkovZaidenberg}}) where the numbers
  \(\alpha_i\) and \(\beta_i\) are defined as in Equation
  \eqref{eq:alphas_betas}. Since \(\gamma_{i_1}=1\) and
  \((\alpha_{i_1},\beta_{i_1})=(\alpha,\beta)\) this tells us
  that \(C^2=-\frac{1}{\alpha\beta}\) (which we already know
  from Lemma \ref{lma:C_intersection}).
\end{proof}

\begin{lemma}\label{lma:delta}
  Let \(g\colon Y\to Z\) be a weighted blow-up with weights
  \((\alpha,\beta)\) at a smooth point \(p\in Z\) and let \(D\subseteq
  Z\) be an irreducible ample \(\QQ\)-divisor. Let \(\til{D}=g^{-1}_*D\) and
  define \(\delta\in\QQ\) by
  \begin{equation}\label{eq:delta}
    g^*D=\til{D}+\delta C.
  \end{equation}
  Then \(\delta=\alpha\beta C\cdot \til{D}>0\).   
\end{lemma}
\begin{proof}
  Since \(g_*C=0\), we have \(C\cdot g^*D = 0\), so
  \(0=C\cdot \til{D} + \delta C^2\). But
  \(C^2=-1/(\alpha\beta)<0\), so
  \(\delta=\alpha\beta C\cdot \til{D}>0\).
\end{proof}

\begin{definition}\label{dfn:mu}
  Let \(\mu(Z,D;g)=\min\left(\delta,\frac{D^2}{C\cdot \til{D}}\right)\).
\end{definition}

\begin{lemma}\label{lma:compute}
  The weighted Seshadri constant satisfies
  \[\mu(Z,D;g)\leq
    \varepsilon(Z,D;g)\leq \frac{D^2}{C\cdot\til{D}}.\]
  In particular, if \(\frac{D^2}{C\cdot\til{D}}\leq \delta\) then
  \(\varepsilon(Z,D;g)=\mu(Z,D;g)=\frac{D^2}{C\cdot \til{D}}\).
\end{lemma}
\begin{proof}
  As noted in Remark \ref{rmk:nef}, the weighted Seshadri constant is
  the maximal \(\varepsilon>0\) for which the divisor
  \(\Upsilon_\varepsilon=g^*D-\varepsilon C\) is nef, that is,
  \(\Upsilon_\varepsilon\) intersects all curves non-negatively. We have
  \[\Upsilon_\varepsilon\cdot \til{D} = (g^*D-\varepsilon C)\cdot
  \til{D} = D^2-\varepsilon C\cdot\til{D},\] so \(\varepsilon\leq
  D^2/C\cdot \til{D}\). This proves the upper bound.

  To prove the lower bound we need to show that if
  \(\varepsilon\leq \min\left(\delta,\frac{D^2}{C\cdot \til{D}}\right)\)
  then \(\Upsilon_\varepsilon\) is nef. Let \(A\subseteq Y\) be an
  irreducible curve; we must show that \(\Upsilon_\varepsilon\cdot A\geq
  0\).
  \begin{itemize}
  \item If \(A=C\) then \(\Upsilon_\varepsilon\cdot A =
    (g^*D-\varepsilon C)\cdot C = \varepsilon/(\alpha\beta)>0\), since
    \(g^*D\cdot C=D\cdot g_*C=0\) and \(C^2=-1/(\alpha\beta)\).
  \item If \(A=\til{D}\) then \(\Upsilon_\varepsilon\cdot A =
    (g^*D-\varepsilon C)\cdot \til{D} = D^2-\varepsilon C\cdot
    \til{D}\), which is non-negative since \(\varepsilon\leq
    \frac{D^2}{C\cdot\til{D}}\).
  \item If \(A\neq C,\til{D}\) then \(A\cdot \til{D}\geq 0\) and
    \(A\cdot C\geq 0\). Computing, we get \[\Upsilon_\varepsilon\cdot
    A=A\cdot \til{D}+(\delta-\varepsilon)A\cdot C.\] Since
    \(\varepsilon\leq \delta\), all terms on the right-hand side are
    non-negative, so this implies \(\Upsilon_\varepsilon\cdot A\geq 0\).
  \end{itemize}
  This shows that \(\Upsilon_\varepsilon\) is nef.
\end{proof}

\begin{remark}\label{rmk:ineff}
  Note that \(\Upsilon_\varepsilon\) is {\em effective} (i.e. a
  non-negative linear combination of curves) if and only if
  \(\varepsilon\leq \delta\). If \(\delta<\frac{D^2}{C\cdot\til{D}}\)
  then \(\mu(Z,D;g)=\delta\) gives us a lower bound on the weighted
  Seshadri constant, but \(\varepsilon(Z,D;g)\) could in principle be
  larger. This will usually mean that when we produce ellipsoids in
  the regime \(\delta<\frac{D^2}{C\cdot\til{D}}\), they
  will not be optimal (they could possibly be made larger). In this
  case, we say we are in the {\em ineffective regime}.
\end{remark}
\begin{remark}\label{rmk:obstr}
  If \(\frac{D^2}{C\cdot \til{D}}\leq \delta\) then Lemma
  \ref{lma:compute} computes the weighted Seshadri constant: as
  \(\varepsilon\) approaches its maximal value, the ``symplectic
  area'' \(\Upsilon_\varepsilon\cdot\til{D}\) goes to zero. Indeed, as
  we will see, this is precisely the symplectic area of \(\til{D}\)
  for a symplectic form associated to the projective embedding coming
  from the linear system of a large multiple of
  \(\Upsilon_\varepsilon\). Now pass to the minimal resolution \(f\colon
  X\to Y\) and take the proper transform
  \(\dbtilde{D}=f^{-1}_*\til{D}\). Suppose that (the reduced curve
  underlying\footnote{We will usually take \(D\) to be a rational
  multiple of a reduced curve, so this just means ``drop the
  coefficient''.}) \(\dbtilde{D}\) has non-zero Gromov--Witten
  invariant (for example, it could be an embedded rational curve of
  square \(-1\)). Since such a curve needs to have positive symplectic
  area, it will provide an obstruction to taking \(\varepsilon\) any
  larger. For this reason, if \(\frac{D^2}{C\cdot\til{D}}\leq
  \delta\), we say that we are in the {\em potentially obstructive
    regime}, and if moreover \(\dbtilde{D}\) has nonzero
  Gromov--Witten invariant, we say that we are in the {\em obstructive
    regime}.
\end{remark}

Evidently, the number \(C\cdot \til{D}\) is the most important thing
to calculate in order to apply this lemma. This can be achieved using
the following lemma.

\begin{lemma}
  Suppose that \(X\stackrel{f}{\to} Y\stackrel{g}{\to} Z\) is the
  minimal resolution \(f\) of a weighted blow-up \(g\) at a smooth
  point \(p\in Z\) with weights \((\alpha,\beta)\) and \(D\subseteq Z\)
  is an irreducible ample \(\QQ\)-divisor. Let
  \(\gamma_1,\ldots,\gamma_m\) be the numbers defined by Equation
  \eqref{eq:gammas_formula}. Then, writing \(\dbtilde{D}\coloneqq
  f^{-1}_*\til{D}\):
  \begin{equation}\label{eq:C_dot_D}
    C\cdot \til{D} = E_{i_1}\cdot \dbtilde{D} +
    \frac{1}{\beta}\sum_{i<i_1}\beta_i E_i\cdot \dbtilde{D} +
    \frac{1}{\alpha}\sum_{i>i_1}\alpha_i E_i\cdot \dbtilde{D}.
  \end{equation}
\end{lemma}
\begin{proof}
  Because \(C\cdot \til{D} = f^*C\cdot
  f^{-1}_*\til{D}\) and \(f^*C=\sum_{i=1}^m\gamma_iE_i\), we get
  \[C\cdot \til{D}=\sum_{i=1}^m\gamma_i E_i\cdot
  \dbtilde{D}.\] Now substituting in the values of \(\gamma_i\) from
  Equation \eqref{eq:gammas_formula} gives the result (recall that
  \(f^{-1}_*C=E_{i_1}\), \(\alpha=\alpha_{i_1}\) and
  \(\beta=\beta_{i_1}\)).
\end{proof}

\section{Examples}
\label{sct:examples}
\subsection{Nodal cubic}
\label{sct:examples_nodal}
Let \(A\) be the nodal cubic curve in the plane defined by the
equation \(\{xyz=x^3+y^3\}\subseteq\CP^2\). We will make a
weighted blow-up with weights \((\alpha,\beta)\) with slope
\(\beta/\alpha\) in the interval \((6,8]\) which is adapted to
\(A\) in the sense of Definition \ref{dfn:general_adapted}. This
will give us large values of the weighted Seshadri constant
\(\varepsilon(\CP^2,\frac{1}{3}A;g)\) which will in turn enable
us to embed large ellipsoids whose slopes lie within this
interval. More precisely, we will prove the following:

\begin{proposition}\label{prp:first_step}
  Equip \(\CP^2\) with the Fubini--Study form which gives a line area
  \(\pi\).
  \begin{itemize}
  \item[(a)] For all \(s\in \left(\frac{7+\sqrt{45}}{2},7\right]\) and for all
    \(\sigma_1<\frac{3s}{1+s}\) and \(\sigma_2 < \frac{3}{1+s}\),
    the ellipsoid \(\Ell(\sigma_1,\sigma_2)\) embeds symplectically in
    \(\CP^2\).
  \item[(b)] For all \(s\in \left[7,\frac{64}{9}\right]\) and for all
    \(\sigma_1<\frac{3s}{8}\) and \(\sigma_2<\frac{3}{8}\), the ellipsoid
    \(\Ell(\sigma_1,\sigma_2)\) embeds symplectically in \(\CP^2\).
  \end{itemize}
\end{proposition}

\begin{remark}
  These are the optimal ellipsoid embeddings for the first exceptional
  step in the McDuff-Schlenk staircase \cite{McDuffSchlenk}: one
  cannot embed larger ellipsoids with these slopes. Note that
  \(\frac{7+\sqrt{45}}{2}\approx 6.85410\) is the fourth power of the
  Golden Ratio.
\end{remark}

\begin{proof}
  (a) The nodal cubic \(A\) has two irreducible branches at
  \(p=[0:0:1]\). If \(s=\beta/\alpha\) lies in the interval
  \((6,7]\) then the Farey process guided by one of these
  branches (see Definition \ref{dfn:guided}) yields a weighted
  blow-up \(g\colon Y\to\CP^2\) adapted to \(A\). Let
  \(f\colon X\to Y\) be the minimal resolution. Note that
  \(\dbtilde{A}=f^{-1}_*g^{-1}_*A\) intersects the left-most and
  right-most curves \(E_1\) and \(E_m\) in the chain each once
  transversely:
  \[\dbtilde{A}\cdot E_1 = \dbtilde{A}\cdot E_m=1,\qquad
    \dbtilde{A}\cdot E_j=0\text{ if }j\neq 1,m.\] Let
  \(D=\frac{1}{3}A\) so that \(D^2=1\). Let
  \(\til{D}=f_*^{-1}D\). We can compute \(\mu(Z,D;g)\) using
  Equation \eqref{eq:C_dot_D}; \(\dbtilde{A}\) hits only \(E_1\) and \(E_m\) and
  \(\beta_1=\alpha_m=1\), so we get:
  \begin{gather*}
    C\cdot \til{D} =
    \frac{1}{3}\left(\frac{1}{\beta}+\frac{1}{\alpha}\right) =
    \frac{\alpha+\beta}{3\alpha\beta},\\
    \mu(\CP^2,D;g) = \min\left(\alpha\beta C\cdot \til{D},
      \frac{D^2}{C\cdot \til{D}}\right) =
    \min\left(\frac{\alpha+\beta}{3},
      \frac{3\alpha\beta}{\alpha+\beta}\right).
  \end{gather*}
  The two quantities inside the \(\min\) are equal if and only if
  \[(\alpha+\beta)^2=9\alpha\beta,\]
  that is, setting \(s=\beta/\alpha\), if and only if
  \(s^2-7s+1=0\). This equation has roots
  \(s_\pm=\frac{7\pm\sqrt{45}}{2}\). If the slope \(s=\beta/\alpha\)
  is bigger than \(s_+\) then we have
  \(\frac{3\alpha\beta}{\alpha+\beta}<\frac{\alpha+\beta}{3}\), so, by
  Lemma \ref{lma:compute},
  \[\varepsilon(\CP^2,D;g) = \mu(\CP^2,D;g) =
    \frac{3\alpha\beta}{\alpha+\beta}=\frac{3s\alpha}{1+s}.\]
  Theorem \ref{thm:ellipsoids_galore} gives us ellipsoid
  embeddings \(\Ell(\sigma_1,\sigma_2)\to \CP^2\) for any
  \(\sigma_1<\frac{3s}{1+s}\) and \(\sigma_2 <\frac{3}{1+s}\),
  as required.

  (b) If \(7<\beta/\alpha\leq 8\) then we perform the Farey
  process guided by a branch of \(A\) until we get to the node
  \((\alpha,\beta)=(1,7)\). Then we continue by choosing
  \(q_7\in F_7\) to be any point {\em except} \(F_7\cap A\), and
  blowing up \(q_7\). This ensures that we stay adapted to
  \(A\), and, provided we don't go beyond \(\beta/\alpha=8\),
  there are no further arbitrary choices to be made. We continue
  until we reach \(X\) and again we take
  \(D=\frac{1}{3}A\). Since the curve \(E_{i_1}=F_m\) will now
  be to the right of \(E_1,\ldots,E_7\), and since the proper
  transform of the nodal cubic intersects only \(E_1\) and
  \(E_7\) (each once transversely), we have
  \[C\cdot \til{D} = \frac{1}{3}\left(\frac{1}{\beta} +
      \frac{7}{\beta}\right)=\frac{8}{3\beta}\] since
  \(\beta_1=1\) and \(\beta_7=7\). We get
  \[\mu(\CP^2,D;g) = \min\left(\frac{8\alpha}{3},
      \frac{3\beta}{8}\right).\] If \(s=\beta/\alpha\leq 64/9\)
  then \(\frac{3\beta}{8}\leq \frac{8\alpha}{3}\), so, by Lemma
  \ref{lma:compute},
  \(\varepsilon(\CP^2,D;g)=\mu(\CP^2,D;g)=3\beta/8\) and Theorem
  \ref{thm:ellipsoids_galore} gives an ellipsoid embedding
  \(\Ell(\sigma_1,\sigma_2)\to \CP^2\) for any
  \(\sigma_1 < \frac{3s}{8}\) and \(\sigma_2 < \frac{3}{8}\).
\end{proof}

\begin{remark}
  The reason the packing ratios are optimal in both these cases is
  that we are in the obstructive regime (see Remark
  \ref{rmk:obstr}). In other words, the curve \(\dbtilde{D}\) is an
  embedded rational curve of square \(-1\) in the minimal resolution,
  and so has non-zero Gromov--Witten invariant, and as we approach the
  weighted Seshadri constant its area goes to zero, so it provides an
  obstruction to making the ellipsoid bigger. If we lie in the
  interval \((6,8]\) but outside either of the subintervals listed in
    Proposition \ref{prp:first_step} then we are in the ineffective
    regime. Whilst we are still able to construct ellipsoids from
    this, they usually fail to be optimal: there is no obstructive
    curve whose area goes to zero.
\end{remark}

\subsection{Steps beyond the Fibonacci staircase}
\label{sct:examples_beyond}
The nodal cubic example enabled us to construct optimal embeddings for
all ellipsoids in the range of slopes between
\(\frac{7+\sqrt{45}}{2}\) and \(\frac{64}{9}\). In fact, we can also
construct optimal ellipsoids for all of the piecewise linear steps in
the McDuff--Schlenk staircase beyond the Fibonacci range (we will
refer to these as the {\em exceptional steps}). We will work out in
detail the most complicated step, which has its midpoint at
\(7\frac{2}{15}\), and give enough information for the reader to be
able to reconstruct the remaining steps following the same reasoning.

\begin{proposition}
  Equip \(\CP^2\) with the Fubini--Study form which gives a line area
  \(\pi\).
  \begin{itemize}
  \item[(a)] For all
    \(s\in
    \left(\frac{435+32\sqrt{179}}{121},\frac{107}{15}\right]\)
    and for all \(\sigma_1<\frac{64s}{14+22s}\) and
    \(\sigma_2 < \frac{64}{14+22s}\), the ellipsoid
    \(\Ell(\sigma_1,\sigma_2)\) embeds symplectically in
    \(\CP^2\).
  \item[(b)] For all
    \(s\in \left[\frac{107}{15},
      \frac{1201-64\sqrt{177}}{49}\right]\) and for all
    \(\sigma_1<\frac{64s}{121+7s}\) and
    \(\sigma_2<\frac{64}{121+7s}\), the ellipsoid
    \(\Ell(\sigma_1,\sigma_2)\) embeds symplectically in \(\CP^2\).
  \end{itemize}
\end{proposition}

\begin{proof}
  There is a rational curve \(A\subseteq\CP^2\) of degree \(64\)
  with a singularity at \(p_0=[0:0:1]\) adapted to a suitably
  general Farey process with \((\alpha,\beta)=(15,107)\) so that
  the proper transform of \(A\) is the embedded \(-1\)-curve
  indicated by a bold wiggly line in Figure \ref{fig:curves}
  (Step 3). This curve can be constructed by applying two
  birational transformations (Orevkov twists) to a suitable
  cubic curve, see {\cite[Remark 5.8]{DumnickiEtAl}}. The path
  in the Farey tree is: travel right until you reach \((1,8)\),
  then left until you reach \((8,57)\), then right one more step
  to \((15,107)\) (see Figure \ref{fig:farey_stairs}). If the
  desired slope \(\beta/\alpha\) lies in the interval
  \((7,107/15)\) (which contains the interval from (a)) then it
  involves only blow-ups between \(F_7\) and \(F_{16}\); if the
  slope lies in the interval \((107/15,50/7)\) (which includes
  the interval from (b)) then it involves only blow-ups between
  \(F_{16}\) and \(F_{14}\).  In either case, the proper
  transform of \(A\) will still satisfy
  \[\dbtilde{A}\cdot F_7=2,\qquad \dbtilde{A}\cdot F_{16}=1,\qquad
    \dbtilde{A}\cdot F_{14}=1.\] Let \(D=\frac{1}{64}A\), let
  \(g\colon Y\to Z\) be the weighted blow-up (obtained by
  contracting all the curves \(F_i\) for \(i<m\)) and let
  \(\til{D}=g^{-1}_*D\).

  (a) Since the final blow-up curve \(F_m\) lies between \(F_7\) and
  \(F_{16}\), Equation \eqref{eq:C_dot_D} yields:
  \[C\cdot \til{D}=\frac{1}{64}\left(\frac{2\beta_7}{\beta} +
  \frac{\alpha_{16}+\alpha_{14}}{\alpha}\right)\]
  where \(\beta_7=7\), \(\alpha_{16}=15\) and \(\alpha_{14}=7\), so
  \[C\cdot \til{D} = \frac{14\alpha+22\beta}{64\alpha\beta}.\]
  This
  gives
  \[\mu(\CP^2,D;g_{(\alpha,\beta)})=\min\left(\alpha\beta C\cdot
      \til{D},\frac{D^2}{C\cdot \til{D}}\right) =
    \alpha\min\left(\frac{14+22s}{64},\frac{64s}{14+22s}\right)\]
  where \(s=\beta/\alpha\). The first term in the minimum is
  smaller when \((14+22s)^2<4096s\), that is if \(s_-<s<s_+\)
  where \(s_\pm=\frac{435\pm 32\sqrt{179}}{121}\),
  \(s_-=0.05677\cdots\), \(s_+ = 7.13331\cdots\). The interval
  we are considering is \((7,107/15)\) and
  \(107/15=7.13333\cdots\). Therefore, if \(s\in (s_+,107/15)\)
  we are in the potentially obstructive regime and, by Lemma
  \ref{lma:compute}, the weighted Seshadri constant is
  \(\varepsilon(\CP^2,D;g) = \mu(\CP^2,D;g) =
  \frac{64\beta}{14+22s}\). This gives an ellipsoid embedding
  \(\Ell(\sigma_1,\sigma_2)\to\CP^2\) whenever \(\sigma_1<\frac{64s}{14+22s}\) and
  \(\sigma_2<\frac{64}{14+22s}\), as required.

  (b) Since the final blow-up curve \(F_m\) lies between \(F_{16}\)
  and \(F_{14}\), Equation \eqref{eq:C_dot_D} yields:
  \[C\cdot \til{D} =
    \frac{1}{64}\left(\frac{2\beta_7+\beta_{16}}{\beta} +
      \frac{\alpha_{14}}{\alpha}\right)\] where \(\beta_7=7\),
  \(\beta_{16}=107\) and \(\alpha_{14}=7\), so
  \[C\cdot\til{D} = \frac{121\alpha+7\beta}{64\alpha\beta}.\]
  This gives
  \[\mu(\CP^2,D;g) = \min\left(\alpha\beta
      C\cdot \til{D},\frac{D^2}{C\cdot \til{D}}\right) =
    \alpha\min\left(\frac{121+7s}{64},\frac{64s}{121+7s}\right)\]
  where \(s=\beta/\alpha\). The first term in the minimum is
  smaller when \((121+7s)^2<4096s\), that is if
  \(s_-<s<s_+\) where
  \(s_{\pm}=\frac{1201\pm 64\sqrt{177}}{49}\),
  \(s_-=7.13337\cdots\), \(s_+=41.88703\cdots\). The interval we
  are considering is between \(107/15\approx 7.13333\cdots\) and
  \(50/7\approx 7.14285\cdots\). Therefore, if
  \(s\in (107/15,s_-)\) we are in the potentially obstructive
  regime and, by Lemma \ref{lma:compute}, the weighted Seshadri
  constant is
  \(\varepsilon(\CP^2,D;g) = \mu(\CP^2,D;g) =
  \frac{64\beta}{121+7s}\). This gives an ellipsoid embedding
  \(\Ell(\sigma_1,\sigma_2)\to \CP^2\) whenever
  \(\sigma_1<\frac{64s}{121+7s}\) and \(\sigma_2<\frac{64}{121+7s}\),
  as required.
\end{proof}

\begin{remark}
  In conclusion, for slopes between
  \(\frac{435+32\sqrt{179}}{121}\) and
  \(\frac{1201-64\sqrt{177}}{49}\), for the given divisor \(D\),
  we are in the potentially obstructive regime\footnote{In fact,
    we are in the obstructive regime, since \(\dbtilde{D}\) is
    an embedded rational curve of square \(-1\).} and we obtain
  embeddings of ellipsoids realising the optimal packing ratios
  \(\frac{4096s}{(14+22s)^2}\) (below \(s=107/15)\) and
  \(\frac{4096s}{(121+7s)^2}\) (above \(s=107/15\)). For
  slopes just outside this interval, we are in the ineffective
  regime and obtain suboptimal ellipsoids. See Figure
  \ref{fig:third_step_packing_ratios}. In fact, just outside
  this interval of slopes, one can find full ellipsoid
  embeddings.
\end{remark}

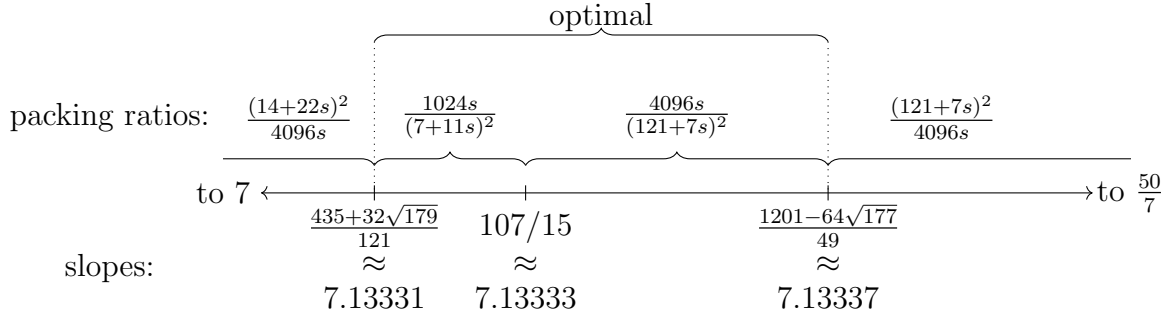
\begin{figure}[htb]
  \centering
  \begin{tikzpicture}
    \draw[<->] (-0.5,0) -- (10.5,0);
    \draw (1,0.1) -- (1,-0.1);
    \draw (3,0.1) -- (3,-0.1);
    \draw (7,0.1) -- (7,-0.1);
    \node at (1,0) [below] {\(\frac{435+32\sqrt{179}}{121}\)};
    \node at (1,-0.7) [below] {\(\approx\)};
    \node at (1,-1.1) [below] {\(7.13331\)};
    \node at (3,-0.1) [below] {\(107/15\)};
    \node at (3,-0.7) [below] {\(\approx\)};
    \node at (3,-1.1) [below] {\(7.13333\)};
    \node at (7,0) [below] {\(\frac{1201-64\sqrt{177}}{49}\)};
    \node at (7,-0.7) [below] {\(\approx\)};
    \node at (7,-1.1) [below] {\(7.13337\)};
    \node at (11,0) {to \(\frac{50}{7}\)};
    \node at (-1,0) {to \(7\)};
    \node at (-2.5,1) {packing ratios:};
    \node at (0,1) {\(\frac{(14+22s)^2}{4096s}\)};
    \node at (2,1) {\(\frac{1024s}{(7+11s)^2}\)};
    \node at (5,1) {\(\frac{4096s}{(121+7s)^2}\)};
    \node at (8.5,1) {\(\frac{(121+7s)^2}{4096s}\)};
    \node at (-2.5,-1) {slopes:};
    \draw [decorate,decoration={brace,amplitude=5pt,mirror,raise=2ex}] (3,0) -- (1,0);
    \draw [decorate,decoration={brace,amplitude=5pt,mirror,raise=2ex}] (7,0) -- (3,0);   
    \begin{scope}
      \clip(-1, 0) rectangle (1, 1);
      \draw [decorate,decoration={brace,amplitude=5pt,mirror,raise=2ex}] (1,0) -- (-9,0);
    \end{scope}
    \begin{scope}
      \clip(7, 0) rectangle (11, 1);
      \draw [decorate,decoration={brace,amplitude=5pt,mirror,raise=2ex}] (20,0) -- (7,0);
    \end{scope}
      \draw [decorate,decoration={brace,amplitude=5pt,mirror,raise=0ex}] (7,2) -- (1,2) node [midway,above] {optimal};
      \draw[dotted] (1,2) -- (1,0);
      \draw[dotted] (7,2) -- (7,0);
  \end{tikzpicture}
  \caption{The packing ratios for the third exceptional step in the
    McDuff--Schlenk staircase.}\label{fig:third_step_packing_ratios}
\end{figure}

\begin{remark}
  A similar analysis can be used with all the other curves in
  Figure \ref{fig:curves} to find the remaining exceptional
  steps of the McDuff--Schlenk staircase. As the slope
  increases, the centre of the Farey process moves rightward in
  the chain and this changes the contributions of the
  corresponding intersections \(\dbtilde{D}\cdot E_i\) to the
  intersection number \(C\cdot \til{D}\) when we pass through
  one of the finitely many slopes corresponding to vertices
  \((\alpha,\beta)\) of the Farey tree for curves \(E_i\) which
  intersect \(\dbtilde{D}\). We summarise the information about
  the staircase in Table \ref{tab:step_data}: each step starts
  at a particular slope \(s=\beta/\alpha\) and ends at another
  slope and yields ellipsoids \(\Ell(\sigma_1,\sigma_2)\) whenever
  \(\sigma_1<\frac{ds}{k+\ell s}\) and \(\sigma_2<\frac{d}{k+\ell s}\) where
  \(k\) and \(\ell\) are defined by
  \[C\cdot \til{D} = \frac{1}{d}\left(\frac{k}{\beta}+\frac{\ell}{\alpha}\right).\] The
  values of \(k\) and \(\ell\) change partway through the step
  (the ``break point''). In all cases, the packing ratio is
  \(\frac{d^2s}{(k+\ell s)^2}\).
\end{remark}

\begin{table}[htb]
  \centering
  \def\arraystretch{1.5}
  \begin{tabular}{c||c||c||c|c||c||c|c||c}
    Step & \(d\) &  Start  & \(k\) & \(\ell\) & Break 
    & \(k\) & \(\ell\) & End \\
    \hline\hline
    \(1\) & \(3\) & \(\frac{7+3\sqrt{5}}{2}\) &  \(1\) & \(1\) & \(7\) &
    \(8\) & \(0\) & \(\frac{8^2}{3^2}\)\\
    \(2\) & \(48\) & \(\frac{1033+48\sqrt{457}}{289}\) & \(7\) & \(17\)
    & \(\frac{57}{8}\) & \(121\) & \(1\) & \(1021-48\sqrt{445}\)\\
    \(3\) & \(64\) & \(\frac{435+32\sqrt{179}}{121}\) &  \(14\) & \(22\) &\(\frac{107}{15}\)
    & \(121\) & \(7\) & \(\frac{1201-64\sqrt{177}}{49}\)\\
    \(4\) & \(24\) & \(\frac{29+6\sqrt{22}}{8}\) & \(7\) & \(8\) &
    \(\frac{50}{7}\) & \(57\) & \(1\) & \(231-24\sqrt{87}\)\\
    \(5\) & \(40\) & \(\frac{618+40\sqrt{218}}{169}\) & \(14\) & \(13\) &
    \(\frac{93}{13}\) & \(107\) & \(0\) & \(\frac{107^2}{40^2}\)\\
    \(6\) & \(16\) & \(\frac{93+16\sqrt{29}}{25}\) & \(7\) & \(5\) &
    \(\frac{36}{5}\) & \(43\) & \(0\) & \(\frac{43^2}{16^2}\)\\
    \(7\) & \(35\) & \(\frac{35^2}{13^2}\) & \(0\) & \(13\) & \(\frac{29}{4}\) & \(87\) &
    \(1\) & \(\frac{1051-35\sqrt{877}}{2}\)\\
    \(8\) & \(8\) & \(\frac{9+4\sqrt{2}}{2}\) & \(7\) & \(2\) &
    \(\frac{15}{2}\) & \(22\) & \(0\) & \(\frac{22^2}{8^2}\)\\
    \(9\) & \(6\) & \(\frac{8+3\sqrt{7}}{2}\) & \(1\) & \(2\) &
    \(8\) & \(17\) & \(0\) & \(\frac{17^2}{6^2}\)
  \end{tabular}
  \caption{The data for the nine exceptional steps in the
    McDuff--Schlenk staircase. Start, Break and End refer to the
    slopes where the piecewise linear steps start, break and
    end. \(d\) refers to the degree of the rational curve \(A\)
    needed for our construction of the ellipsoid. \(k\) and
    \(\ell\) refer to the formula for
    \(C\cdot
    \til{D}=\frac{1}{d}\left(\frac{k}{\beta}+\frac{\ell}{\alpha}\right)\):
    they change in value on either side of the breakpoint of the
    step.}\label{tab:step_data}
\end{table}

\begin{figure}[htb]
  \centering
  \begin{tikzpicture}
    \node at (-1,0) {Step 1};
    \foreach \x in {0,1,2} {
      \tikzmath{
        integer \y, \z;
        \y = 2*\x+1;
        \z = 2*\x+2;
      }
      \draw ({2*\x},0) -- ({2*\x+1.2},0.5) node [pos=0.4,above] {\(F_{\y}\)} node [pos=0.6,below] {\(-2\)};
      \draw ({2*\x+1},0.5) -- ({2*\x+2.2},0) node [pos=0.6,above] {\(F_{\z}\)} node [pos=0.4,below] {\(-2\)};
    };
    \foreach \x in {0,1,2,3,4,5,6} {
      \tikzmath{
        integer \y, \z;
        \y = \x+1;
      }
      \node at ({\x+0.6},-0.5) {\(\frac{\y}{1}\)};
    }
    \draw (6,0) -- (7.2,0.5) node [pos=0.4,above] {\(F_7\)} node [pos=0.6,below] {\(-1\)};
    \draw[very thick] (0.2,0) -- (0.2,0.5) to[out=90,in=180] (3.5,1.2) to[out=0,in=90] (7,0.5) -- (7,0);
    \begin{scope}[shift={(-3,-2.5)}]
      \node at (2,0) {Step 2};
      \draw (3,0) -- (4.2,0.5) node [pos=0.4,above] {\(F_1\)} node [pos=0.6,below] {\(-2\)};
      \draw (5,0.5) -- (6.2,0) node [pos=0.6,above] {\(F_6\)} node [pos=0.4,below] {\(-2\)};;
      \node at (4.6,0.25) {\(\cdots\)};
      \draw (6,0) -- (7.2,0.5) node [pos=0.4,above] {\(F_7\)} node [pos=0.6,below] {\(-9\)};
      \draw (7,0.5) -- (9.2,0) node [pos=0.6,above] {\(F_{15}\)} node [pos=0.6,below] {\(-1\)};
      \foreach \x in {4,5,6} {
        \tikzmath{
          integer \y, \z;
          \y = 22-2*\x;
          \z = 22-2*\x-1;
        }
        \draw ({2*\x+1},0) -- ({2*\x+2.2},0.5) node [pos=0.4,above] {\(F_{\y}\)} node [pos=0.6,below] {\(-2\)};
        \draw ({2*\x+2},0.5) -- ({2*\x+3.2},0) node [pos=0.6,above] {\(F_{\z}\)} node [pos=0.4,below] {\(-2\)};
      };
      \node at (3.6,-0.5) {\(\frac{1}{1}\)};
      \node at (5.6,-0.5) {\(\frac{6}{1}\)};
      \node at (6.6,-0.5) {\(\frac{7}{1}\)};
      \node at (7.6,-0.5) {\(\frac{57}{8}\)};
      \node at (9.6,-0.5) {\(\frac{50}{7}\)};
      \node at (10.6,-0.5) {\(\frac{43}{6}\)};
      \node at (11.6,-0.5) {\(\frac{36}{5}\)};
      \node at (12.6,-0.5) {\(\frac{29}{4}\)};
      \node at (13.6,-0.5) {\(\frac{22}{3}\)};
      \node at (14.6,-0.5) {\(\frac{15}{2}\)};
      \node at (15.6,-0.5) {\(\frac{8}{1}\)};
      \draw (15,0) -- (16.2,0.5) node [pos=0.4,above] {\(F_8\)} node [pos=0.6,below] {\(-2\)};    
      \draw[very thick] (7.2,0) to[out=135,in=180] (7,1.2) to[out=0,in=180] (8.4,-0.5) to[out=0,in=200] (9,1.2) to[out=20,in=180] (12.5,1.2) to[out=0,in=90] (16.1,0);
    \end{scope}
    \begin{scope}[shift={(-3,-5)}]
      \node at (3.6,-0.9) {\(\frac{1}{1}\)};
      \node at (5.6,-0.9) {\(\frac{6}{1}\)};
      \node at (6.6,-0.9) {\(\frac{7}{1}\)};
      \node at (7.6,-0.9) {\(\frac{57}{8}\)};
      \node at (9.6,-0.9) {\(\frac{50}{7}\)};
      \node at (10.6,-0.9) {\(\frac{43}{6}\)};
      \node at (11.6,-0.9) {\(\frac{36}{5}\)};
      \node at (12.6,-0.9) {\(\frac{29}{4}\)};
      \node at (13.6,-0.9) {\(\frac{22}{3}\)};
      \node at (14.6,-0.9) {\(\frac{15}{2}\)};
      \node at (15.6,-0.9) {\(\frac{8}{1}\)};
      \node at (8.6,-0.9) {\(\frac{107}{15}\)};
      \node at (2,0) {Step 3};
      \draw (3,0) -- (4.2,0.5) node [pos=0.4,above] {\(F_1\)} node [pos=0.6,below] {\(-2\)};
      \draw (5,0.5) -- (6.2,0) node [pos=0.6,above] {\(F_6\)} node [pos=0.4,below] {\(-2\)};;
      \node at (4.6,0.25) {\(\cdots\)};
      \draw (6,0) -- (7.2,0.5) node [pos=0.4,above] {\(F_7\)} node [pos=0.6,below] {\(-9\)};
      \draw (7,0.5) -- (8.2,0) node [pos=0.6,above] {\(F_{15}\)} node [pos=0.4,below] {\(-2\)};
      \draw (7.8,0.1) -- (9.4,0.1) node [pos=0.55,above] {\(F_{16}\)} node [pos=0.55,below] {\(-1\)};
      \foreach \x in {4,5,6} {
        \tikzmath{
          integer \y, \z;
          \y = 22-2*\x-2;
          \z = 22-2*\x-1;
        }
        \draw ({2*\x+3},0) -- ({2*\x+4.2},0.5) node [pos=0.4,above] {\(F_{\y}\)} node [pos=0.6,below] {\(-2\)};
        \draw ({2*\x+2},0.5) -- ({2*\x+3.2},0) node [pos=0.6,above] {\(F_{\z}\)} node [pos=0.4,below] {\(-2\)};
      };
      \draw (9,0) -- (10.2,0.5) node [pos=0.4,above] {\(F_{14}\)} node [pos=0.6,below] {\(-3\)};    
      \draw[very thick] (6,0.5) to[out=-45,in=-135] (7,-0.4) to[out=45,in=-45] (6.7,0.7) to[out=135,in=180] (8,1.2) to[out=0,in=90] (8.3,0) to[out=-90,in=180] (8.5,-0.5) to[out=0,in=180] (10,-0.5) to[out=0,in=-45] (9.9,0.4) to[out=135,in=180] (11,1);
    \end{scope}
    \begin{scope}[shift={(-3,-7.7)}]
      \node at (3.6,-0.5) {\(\frac{1}{1}\)};
      \node at (5.6,-0.5) {\(\frac{6}{1}\)};
      \node at (6.6,-0.5) {\(\frac{7}{1}\)};
      \node at (7.6,-0.5) {\(\frac{50}{7}\)};
      \node at (8.6,-0.5) {\(\frac{43}{6}\)};
      \node at (9.6,-0.5) {\(\frac{36}{5}\)};
      \node at (10.6,-0.5) {\(\frac{29}{4}\)};
      \node at (11.6,-0.5) {\(\frac{22}{3}\)};
      \node at (12.6,-0.5) {\(\frac{15}{2}\)};
      \node at (13.6,-0.5) {\(\frac{8}{1}\)};
      \node at (2,0) {Step 4};
      \draw (3,0) -- (4.2,0.5) node [pos=0.4,above] {\(F_1\)} node [pos=0.6,below] {\(-2\)};
      \draw (5,0.5) -- (6.2,0) node [pos=0.6,above] {\(F_6\)} node [pos=0.4,below] {\(-2\)};;
      \node at (4.6,0.25) {\(\cdots\)};
      \draw (6,0) -- (7.2,0.5) node [pos=0.4,above] {\(F_7\)} node [pos=0.6,below] {\(-8\)};
      \foreach \x in {3,4,5} {
        \tikzmath{
          integer \y, \z;
          \y = 22-2*\x-2;
          \z = 22-2*\x-3;
        }
        \draw ({2*\x+1},0.5) -- ({2*\x+2.2},0) node [pos=0.6,above] {\(F_{\y}\)} node [pos=0.4,below] {\(-2\)};
        \draw ({2*\x+2},0) -- ({2*\x+3.2},0.5) node [pos=0.4,above] {\(F_{\z}\)} node [pos=0.6,below] {\(-2\)};
      };
      \draw (13,0.5) -- (14.2,0) node [pos=0.6,above] {\(F_{8}\)} node [pos=0.4,below] {\(-2\)};    
      \draw[very thick] (6.8,0.7) to[out=-45,in=180] (7.1,0.1) to[out=0,in=-135] (7.4,0.7) to[out=45,in=180] (11.5,1.2) to[out=0,in=180] (13,1.2) to[out=0,in=45] (14,0);
    \end{scope}
    \begin{scope}[shift={(-3,-10)}]
      \node at (3.6,-0.9) {\(\frac{1}{1}\)};
      \node at (5.6,-0.9) {\(\frac{6}{1}\)};
      \node at (6.6,-0.9) {\(\frac{7}{1}\)};
      \node at (7.6,-0.9) {\(\frac{50}{7}\)};
      \node at (9.6,-0.9) {\(\frac{43}{6}\)};
      \node at (10.6,-0.9) {\(\frac{36}{5}\)};
      \node at (11.6,-0.9) {\(\frac{29}{4}\)};
      \node at (12.6,-0.9) {\(\frac{22}{3}\)};
      \node at (13.6,-0.9) {\(\frac{15}{2}\)};
      \node at (14.6,-0.9) {\(\frac{8}{1}\)};
      \node at (8.6,-0.9) {\(\frac{93}{13}\)};
      \node at (2,0) {Step 5};
      \draw (3,0) -- (4.2,0.5) node [pos=0.4,above] {\(F_1\)} node [pos=0.6,below] {\(-2\)};
      \draw (5,0.5) -- (6.2,0) node [pos=0.6,above] {\(F_6\)} node [pos=0.4,below] {\(-2\)};;
      \node at (4.6,0.25) {\(\cdots\)};
      \draw (6,0) -- (7.2,0.5) node [pos=0.4,above] {\(F_7\)} node [pos=0.6,below] {\(-8\)};
      \foreach \x in {4,5} {
        \tikzmath{
          integer \y, \z;
          \y = 22-2*\x-4;
          \z = 22-2*\x-3;
        }
        \draw ({2*\x+4},0.5) -- ({2*\x+5.2},0) node [pos=0.6,above] {\(F_{\y}\)} node [pos=0.4,below] {\(-2\)};
        \draw ({2*\x+3},0) -- ({2*\x+4.2},0.5) node [pos=0.4,above] {\(F_{\z}\)} node [pos=0.6,below] {\(-2\)};
      };
      \draw (10,0.5) -- (11.2,0) node [pos=0.6,above] {\(F_{12}\)} node [pos=0.4,below] {\(-2\)};
      \draw (9,0) -- (10.2,0.5) node [pos=0.4,above] {\(F_{13}\)} node [pos=0.6,below] {\(-3\)};
      \draw (7,0.5) -- (8.2,0) node [pos=0.6,above] {\(F_{14}\)} node [pos=0.4,below] {\(-2\)};
      \draw (7.8,0.1) -- (9.4,0.1) node [pos=0.55,above] {\(F_{15}\)} node [pos=0.55,below] {\(-1\)};
      \draw[very thick] (6,0.5) to[out=-45,in=-135] (7,-0.4) to[out=45,in=-45] (6.7,0.7) to[out=135,in=180] (8,1.2) to[out=0,in=90] (8.3,0) to[out=-90,in=180] (8.5,-0.5) to[out=0,in=180] (10,-0.5);
    \end{scope}
    \begin{scope}[shift={(-3,-12.6)}]
      \node at (3.6,-0.9) {\(\frac{1}{1}\)};
      \node at (5.6,-0.9) {\(\frac{6}{1}\)};
      \node at (6.6,-0.9) {\(\frac{7}{1}\)};
      \node at (7.6,-0.9) {\(\frac{36}{5}\)};
      \node at (8.6,-0.9) {\(\frac{29}{4}\)};
      \node at (9.6,-0.9) {\(\frac{22}{3}\)};
      \node at (10.6,-0.9) {\(\frac{15}{2}\)};
      \node at (11.6,-0.9) {\(\frac{8}{1}\)};
      \node at (2,0) {Step 6};
      \draw (3,0) -- (4.2,0.5) node [pos=0.4,above] {\(F_1\)} node [pos=0.6,below] {\(-2\)};
      \draw (5,0.5) -- (6.2,0) node [pos=0.6,above] {\(F_6\)} node [pos=0.4,below] {\(-2\)};;
      \node at (4.6,0.25) {\(\cdots\)};
      \draw (6,0) -- (7.2,0.5) node [pos=0.4,above] {\(F_7\)} node [pos=0.6,below] {\(-6\)};
      \foreach \x in {3,4} {
        \tikzmath{
          integer \y, \z;
          \y = 22-2*\x-6;
          \z = 22-2*\x-5;
        }
        \draw ({2*\x+3},0.5) -- ({2*\x+4.2},0) node [pos=0.6,above] {\(F_{\y}\)} node [pos=0.4,below] {\(-2\)};
        \draw ({2*\x+2},0) -- ({2*\x+3.2},0.5) node [pos=0.4,above] {\(F_{\z}\)} node [pos=0.6,below] {\(-2\)};
      };
      \draw (7,0.5) -- (8.2,0) node [pos=0.6,above] {\(F_{12}\)} node [pos=0.4,below] {\(-1\)};
      \draw[very thick] (7,-0.4) to[out=45,in=-45] (6.7,0.7) to[out=135,in=180] (7.8,1.2) to[out=0,in=45] (7.8,-0.2) to[out=-130,in=180] (8.5,-0.5) to[out=0,in=180] (10,-0.5);
    \end{scope}
    \begin{scope}[shift={(-3,-15.2)}]
      \node at (3.6,-0.7) {\(\frac{1}{1}\)};
      \node at (5.6,-0.7) {\(\frac{6}{1}\)};
      \node at (6.6,-0.7) {\(\frac{7}{1}\)};
      \node at (9.3,-0.7) {\(\frac{29}{4}\)};
      \node at (10.6,-0.7) {\(\frac{22}{3}\)};
      \node at (11.6,-0.7) {\(\frac{15}{2}\)};
      \node at (12.6,-0.7) {\(\frac{8}{1}\)};
      \node at (2,0) {Step 7};
      \draw (3,0) -- (4.2,0.5) node [pos=0.4,above] {\(F_1\)} node [pos=0.6,below] {\(-2\)};
      \draw (5,0.5) -- (6.2,0) node [pos=0.6,above] {\(F_6\)} node [pos=0.4,below] {\(-2\)};;
      \node at (4.6,0.25) {\(\cdots\)};
      \draw (6,0) -- (7.2,0.5) node [pos=0.4,above] {\(F_7\)} node [pos=0.6,below] {\(-5\)};
      \draw (7,0.5) -- (10.2,0) node [pos=0.7,above] {\(F_{11}\)} node [pos=0.7,below] {\(-1\)};
      \draw (10,0) -- (11.2,0.5) node [pos=0.4,above] {\(F_{10}\)} node [pos=0.6,below] {\(-2\)};
      \draw (11,0.5) -- (12.2,0) node [pos=0.6,above] {\(F_9\)} node [pos=0.4,below] {\(-2\)};
      \draw (12,0) -- (13.2,0.5) node [pos=0.4,above] {\(F_8\)} node [pos=0.6,below] {\(-2\)};
      \draw[very thick] (7.1,-0.2) to[out=45,in=135] (7.8,0.8) to[out=-45,in=90] (7.5,-0.1) to[out=-90,in=-90] (8.3,0.5) to[out=90,in=180] (10,1.2) to[out=0,in=90] (12.9,0.5) to[out=-90,in=-45] (13,0.2);
    \end{scope}
    \begin{scope}[shift={(-3,-17.5)}]
      \node at (3.6,-0.9) {\(\frac{1}{1}\)};
      \node at (5.6,-0.9) {\(\frac{6}{1}\)};
      \node at (6.6,-0.9) {\(\frac{7}{1}\)};
      \node at (7.6,-0.9) {\(\frac{15}{2}\)};
      \node at (8.6,-0.9) {\(\frac{8}{1}\)};
      \node at (2,0) {Step 8};
      \draw (3,0) -- (4.2,0.5) node [pos=0.4,above] {\(F_1\)} node [pos=0.6,below] {\(-2\)};
      \draw (5,0.5) -- (6.2,0) node [pos=0.6,above] {\(F_6\)} node [pos=0.4,below] {\(-2\)};;
      \node at (4.6,0.25) {\(\cdots\)};
      \draw (6,0) -- (7.2,0.5) node [pos=0.4,above] {\(F_7\)} node [pos=0.6,below] {\(-3\)};
      \draw (7,0.5) -- (8.2,0) node [pos=0.6,above] {\(F_9\)} node [pos=0.4,below] {\(-1\)};
      \draw (8,0) -- (9.2,0.5) node [pos=0.4,above] {\(F_8\)} node [pos=0.6,below] {\(-2\)};
      \draw[very thick] (7,-0.4) to[out=45,in=-45] (6.7,0.7) to[out=135,in=180] (7.8,1.2) to[out=0,in=45] (7.8,-0.2) to[out=-130,in=180] (8.5,-0.5) to[out=0,in=180] (10,-0.5);
    \end{scope}
    \begin{scope}[shift={(-3,-20)}]
      \node at (3.6,-0.7) {\(\frac{1}{1}\)};
      \node at (5.6,-0.7) {\(\frac{6}{1}\)};
      \node at (6.6,-0.7) {\(\frac{7}{1}\)};
      \node at (8.6,-0.7) {\(\frac{8}{1}\)};
      \node at (2,0) {Step 9};
      \draw (3,0) -- (4.2,0.5) node [pos=0.4,above] {\(F_1\)} node [pos=0.6,below] {\(-2\)};
      \draw (5,0.5) -- (6.2,0) node [pos=0.6,above] {\(F_6\)} node [pos=0.4,below] {\(-2\)};;
      \node at (4.6,0.25) {\(\cdots\)};
      \draw (6,0) -- (7.2,0.5) node [pos=0.4,above] {\(F_7\)} node [pos=0.6,below] {\(-2\)};
      \draw (7,0.5) -- (9.2,0) node [pos=0.6,above] {\(F_8\)} node [pos=0.4,below] {\(-1\)};
      \draw[very thick] (3.3,-0.2) to[out=135,in=180] (3.3,1.2) to[out=0,in=180] (7.8,1.2) to[out=0,in=45] (7.2,-0.2) to[out=-130,in=180] (7.5,-0.5) to[out=0,in=-90] (9,0.5);
    \end{scope}
  \end{tikzpicture}
  \caption{The curves required for constructing optimal
    ellipsoids in the McDuff--Schlenk staircase beyond the
    Fibonacci stairs. The fractions indicate the Farey labels
    \(\frac{\beta_i}{\alpha_i}\) defined in Equation
    \eqref{eq:alphas_betas}.}\label{fig:curves}
\end{figure}
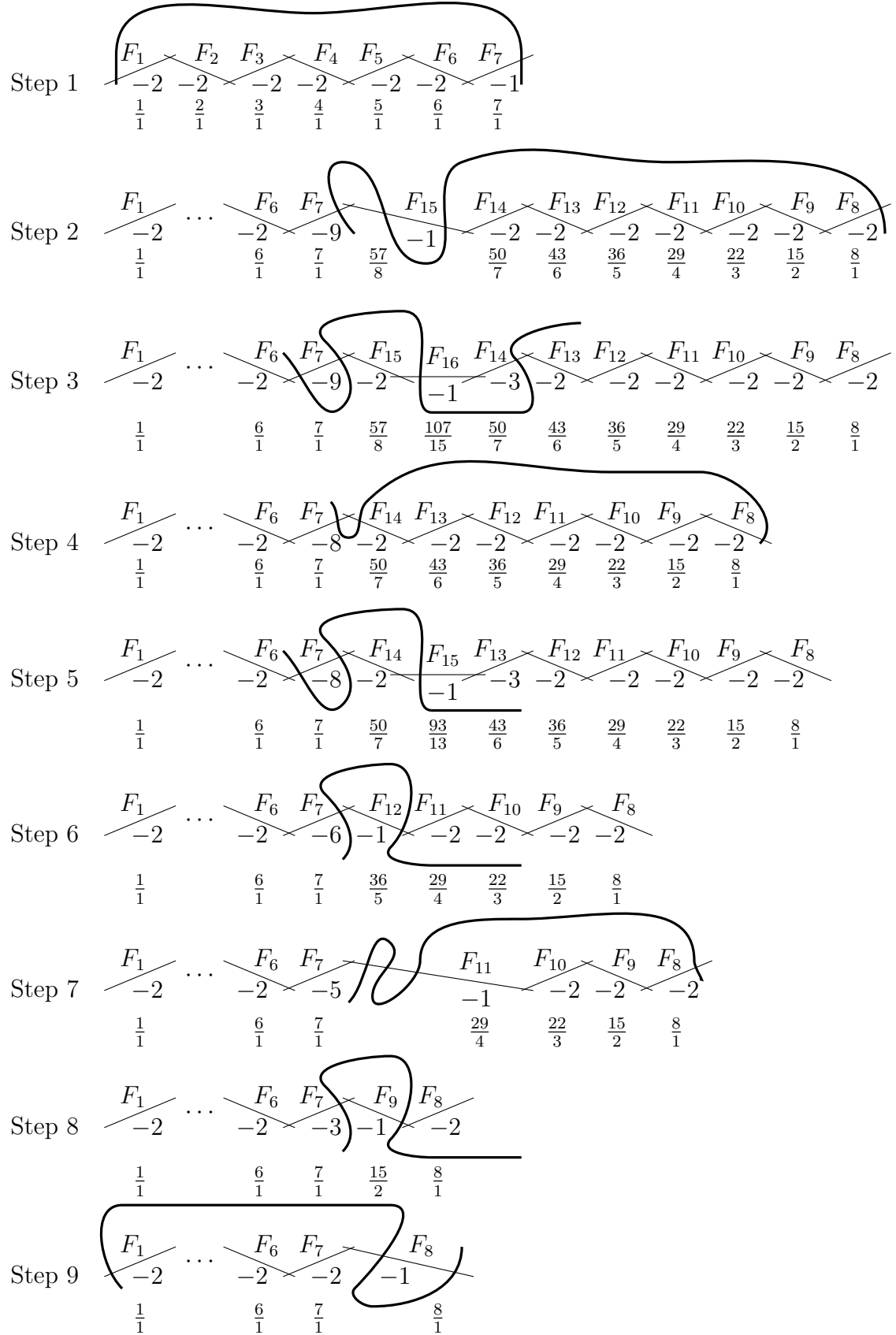

\begin{figure}[htb]
  \centering
  \begin{tikzpicture}
    \node (0) at (-2,-0.2) {};
    \node (a) at (0,0) {\((1,7)\)};
    \node (b) at (10,1) {\((1,8)\)};
    \node (c) at (9,2) {\((2,15)\)};
    \node (d) at (8,3) {\((3,22)\)};
    \node (e) at (7,4) {\((4,29)\)};
    \node (f) at (6,5) {\((5,36)\)};
    \node (g) at (5,6) {\((6,43)\)};
    \node (h) at (3,7) {\((7,50)\)};
    \node (i) at (4,8) {\((13,93)\)};
    \node (j) at (1,8) {\((8,57)\)};
    \node (k) at (2,9) {\((15,107)\)};
    \draw[dashed,->] (0) -- (a);
    \draw[->] (a) -- (b);
    \draw[->] (b) -- (c);
    \draw[->] (c) -- (d);
    \draw[->] (d) -- (e);
    \draw[->] (e) -- (f);
    \draw[->] (f) -- (g);
    \draw[->] (g) -- (h);
    \draw[->] (h) -- (j);
    \draw[->] (h) -- (i);
    \draw[->] (j) -- (k);
    \draw (-0.5,-1) -- (10.5,-1);
    \draw[dotted] (a) -- (0,-1);
    \draw[dotted] (b) -- (10,-1);
    \draw[dotted] (c) -- (5,-1);
    \draw[dotted] (e) -- (2.5,-1);
    \draw[dotted] (f) -- (2,-1);
    \draw[dotted] (h) -- (10/7,-1);
    \draw[dotted] (i) -- (20/13,-1);
    \draw[dotted] (j) -- (10/8,-1);
    \draw[dotted] (k) -- (20/15,-1);
    \node at (0,-1) [below] {\(7\)};
    \node at (10,-1) [below] {\(8\)};
    \node at (5,-1) [below] {\(7\tfrac{1}{2}\)};
    \node at (2.5,-1) [below] {\(7\tfrac{1}{4}\)};
    \node at (2,-1) [below] {\(7\tfrac{1}{5}\)};
    \node at (1.3,-1) [below] {\(\cdots\)};
  \end{tikzpicture}
  \caption{The part of the Farey tree relevant for the Farey
    processes involved in finding the midpoints of the
    exceptional steps of the McDuff--Schlenk staircase. Below we
    see a number line: each pair \((\alpha,\beta)\) involved in
    a step is connected to the corresponding rational number
    \(\beta/\alpha\).}\label{fig:farey_stairs}
\end{figure}
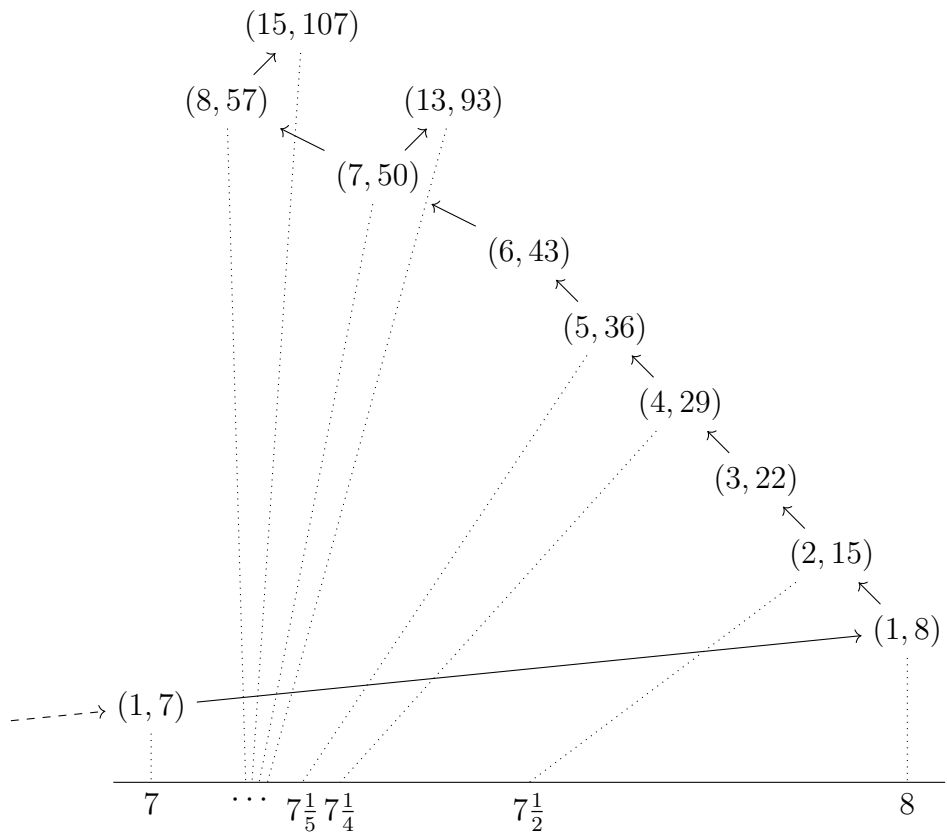

\clearpage

\subsection{Unicuspidal curves}
\label{sct:examples_unicuspidal}
McDuff and Siegel \cite{McDuffSiegelInfiniteStairs} use {\em
  unicuspidal} and {\em sesquicuspidal} curves to produce ellipsoids
for the Fibonacci portion of the McDuff--Schlenk staircase. Recall
that a {\em \((p,q)\)-cusp} is a singularity locally analytically
modelled on the curve germ \(\{(x,y)\in\CC^2\,:\,y^p=x^q\}\), where
\(\gcd(p,q)=1\); a {\em \((p,q)\)-unicuspidal curve} is a curve with a
unique singular point which is a \((p,q)\)-cusp, and a {\em
  \((p,q)\)-sesquicuspidal curve} is allowed to have some additional
nodal singularities. Theorem A(a) of
{\cite{McDuffSiegelInfiniteStairs}} explains how to use such a curve
to produce symplectic ellipsoids. We illustrate how our methods work
in this context by proving the following result:

\begin{proposition}\label{prp:unicuspidal}
  Let \(A\subseteq\CP^2\) be an irreducible curve of degree \(d\)
  with a \((p,q)\)-cusp (and possibly other singular
  points). Then:
  \begin{itemize}
  \item for any slope \(s\) in the interval
    \(d^2/q^2 < s < p/q\) there is a symplectic embedding of
    \(\Ell\left(\frac{d}{q},\frac{d}{qs}\right)\) into \(\CP^2\),
  \item and for any slope \(s\) in the interval
    \(p/q\leq s<p^2/d^2\) there is a symplectic embedding of
    \(\Ell\left(\frac{ds}{p},\frac{d}{p}\right)\) into
    \(\CP^2\).
  \end{itemize}
  In particular, we get full fillings for slopes
  \(s=\frac{d^2}{q^2}\) and \(s=\frac{p^2}{d^2}\).
\end{proposition}
\begin{proof}
  The \((p,q)\)-cusp admits a toric weighted blow-up with weights
  \((q,p)\) so that the proper transform of the cusp is a curve
  intersecting the final blow-up curve \(F_m\) once
  transversely. Since \(A\) has a singular point modelled locally
  analytically on this cusp, we can perform the corresponding weighted
  blow-up on \(\CP^2\) (note that this need no longer be toric since
  the local analytic chart around the cusp will not be toric).

  As usual, we will take \(D=A/d\). Given a slope
  \(s=\beta/\alpha\), we can perform a Farey process with
  weights \((\alpha,\beta)\) which follows the Farey process
  guided by the cusp as far as it can in the Farey tree. If
  \(s<p/q\) (respectively \(s\geq p/q\)) then \(\dbtilde{A}\)
  will intersect the resolution chain \(E_1,\ldots,E_m\)
  somewhere to the right (respectively left) of \(E_{i_1}\) so
  we will have \(C\cdot\til{D}=\frac{q}{d\alpha}\) (respectively
  \(C\cdot\til{D}=\frac{p}{d\beta}\)). This gives
  \[\mu(\CP^2,D;g)=\min\left(\alpha\beta C\cdot\til{D},\frac{1}{C\cdot
    \til{D}}\right)=\begin{cases}\min\left(\frac{\beta
    q}{d},\frac{d\alpha}{q}\right)&\text{ if
  }\frac{\beta}{\alpha}<\frac{p}{q},\\
  \min\left(\frac{p\alpha}{d},\frac{d\beta}{p}\right)&\text{
    if }\frac{p}{q}\leq \frac{\beta}{\alpha}.\end{cases}\]
  Therefore if \(\frac{d^2}{q^2}<s<\frac{p}{q}\) or \(\frac{p}{q}\leq
  s<\frac{p^2}{d^2}\) then we are in the potentially obstructive
  regime and, by Lemma \ref{lma:compute},
  \[\varepsilon(\CP^2,D;g) = \mu(\CP^2,D;g)
    = \begin{cases}\frac{d\alpha}{q}&\text{
        if }\frac{d^2}{q^2}<s<\frac{p}{q},\\
      \frac{d\beta}{p}&\text{ if }\frac{p}{q}\leq
      s<\frac{p^2}{d^2}.\end{cases}\] Now Theorem
  \ref{thm:ellipsoids_galore} gives the ellipsoids we seek. The
  packing ratios are \(\frac{d^2}{q^2s}\) if \(s<p/q\) and
  \(\frac{d^2s}{p^2}\) if \(s>p/q\), so we get full fillings
  when \(s=d^2/q^2\) or \(s=p^2/d^2\).
\end{proof}

\begin{remark}
  Note that outside the intervals mentioned in the statement of
  Proposition \ref{prp:unicuspidal}, we are in the ineffective
  regime and if we tried to construct ellipsoids with such
  slopes using the same curve \(A\), they would end up being
  suboptimal.
\end{remark}

\begin{example}
  Let \(F_1,F_2,F_3,F_4,F_5,\ldots=1,1,2,3,5,\ldots\) denote the
  Fibonacci sequence and recall that, for each \(k\geq 1\),
  there is an \((F_{2k+1},F_{2k+5})\)-unicuspidal curve of
  degree \(F_{2k+3}\); see {\cite[Theorem C(a)]{Orevkov}}. When
  applied to these curves, Proposition \ref{prp:unicuspidal}
  yields the Fibonacci stairs. Indeed, in this case the proper
  transform of \(A\) is an embedded rational \(-1\)-curve, so we
  are in the obstructive regime and the ellipsoids are optimal.
\end{example}

\begin{example}\label{exm:genus_1_unicuspidal}
  There is a \((64,9)\)-unicuspidal curve of degree \(24\) and
  genus \(1\). This can be obtained from a smooth cubic (degree
  \(3\)) by applying a suitable Orevkov twist (which is a
  birational automorphism of \(\CP^2\) of degree \(8\)) -- see
  for example {\cite[Example 8.18]{BodnarCeloriaGolla}} or
  {\cite[Proof of Corollary 5.28]{CilibertoEtAl}}. In this case,
  \(d^2/q^2=p/q=p^2/d^2=64/9\), so we obtain ellipsoids
  \(\Ell(a,b)\subseteq\CP^2\) for any \(a<8/3\) and any
  \(b<3/8\). This gives another construction of the full
  fillings with slope \(s=64/9\).
\end{example}

\begin{remark}
  If one can find a \((p^2,q^2)\)-unicuspidal curve of degree \(pq\)
  in \(\CP^2\) then one would obtain a full filling with slope
  \(p^2/q^2\); by the adjunction formula, such a curve would
  necessarily have positive genus. One might try to construct such a
  curve by applying a birational transformation of degree \(p\) to a
  smooth curve of degree \(q\) (Example \ref{exm:genus_1_unicuspidal}
  is the case \(p=8\), \(q=3\)). As we saw, the exceptional steps of
  the McDuff--Schlenk staircase give us full fillings of \(\CP^2\) at
  the rational slopes \(s=\frac{107^2}{40^2}\), \(\frac{43^2}{16^2}\),
  \(\frac{35^2}{13^2}\), \(\frac{22^2}{8^2}\) and
  \(\frac{17^2}{6^2}\). It would be nice to know if there exist
  corresponding unicuspidal curves.\footnote{Since this paper was
  written, Jenia Tevelev \cite{Tevelev} has ruled out the existence of
  such a unicuspidal curve corresponding to the slope \(17^2/6^2\).}
\end{remark}

\begin{remark}
  One could also allow more general cusps, with multiple Puiseux
  pairs. The overall effect of this would be to change the value
  of \(C\cdot\til{D}\) by a multiplicative constant: the only
  difference is that the unique branch of \(\dbtilde{A}\) would now
  hit the curve labelled \((p,q)\) in the Farey tree
  non-transversely, with some intersection number \(k\). The
  effect in the calculations is to replace \(p\) and \(q\) by
  \(kp\) and \(kq\). Compare with {\cite[Theorem
    A(b)]{McDuffSiegelInfiniteStairs}}.
\end{remark}

\subsection{Full fillings outside the steps?}
\label{sct:examples_outside}
As remarked in Section \ref{sct:disadvantages}, this method usually
does not seem to produce optimal ellipsoids whose slopes lie in
intervals where full fillings exist, in other words for any slope
\(s\) between \(\frac{7+\sqrt{45}}{2}\) and \(\infty\) which does not
lie on one of the nine McDuff--Schlenk steps. The exception is when
the slope is a square number \(s=d^2\). Such full fillings were
constructed by Opshtein {\cite[Lemma 2.1]{OpshteinMaximal}}: the
neighbourhood of a smooth curve of degree \(d\) is a ball subbundle of
its symplectic normal bundle, and the total space of the restriction
of this bundle to an open ball in the smooth curve is an ellipsoid of
slope \(d^2\). McDuff gave an alternative construction of the same
ellipsoid {\cite[Proposition 2.1]{McDuffEllipsoids}} by inflating
along the smooth curve of degree \(d\). To illustrate the connection
with our method, we explain his result in our language before going on
to show why things are harder for other slopes.

\begin{proposition}[Opshtein \cite{OpshteinMaximal}]
  There is a symplectic embedding of
  \(\Ell(\varepsilon,\varepsilon/d^2)\) into \(\CP^2\) for any
  integer \(d\) and any \(\varepsilon<d\).
\end{proposition}
\begin{proof}
  Let \(A\) be a smooth curve of degree \(d\) passing through
  \(p=[0:0:1]\). Consider the Farey process guided by \(A\)
  corresponding to the path in the Farey tree which simply moves
  right \(d^2\) times; this produces a weighted blow-up
  \(g\colon Y\to \CP^2\) with weights \((1,d^2)\) where the
  proper transform of \(A\) intersects \(E_{d^2}\) once
  transversely. If we use \(D=A/d\) then \(C\cdot\til{D}=1/d\)
  and, by Lemma \ref{lma:compute}, the weighted Seshadri
  constant is \(\varepsilon(\CP^2,D;g)=\mu(\CP^,D;g)=d\), so
  Theorem \ref{thm:ellipsoids_galore} gives embeddings of
  \(\Ell(\varepsilon,\varepsilon/d^2)\) for all
  \(\varepsilon<d\) as required.
\end{proof}

\begin{example}\label{exm:slope_10}
  Now let us focus on the slope \(s=10\). Let \(g\colon Y\to Z\)
  be a weighted blow-up with weights \((\alpha,\beta)=(1,10)\)
  and let \(f\colon X\to Y\) be the minimal resolution: the
  exceptional locus of \(g\circ f\) is a chain of curves
  \(E_1,\ldots,E_{10}\) with self-intersections
  \(E_1^2=\cdots=E_9^2=-2\) and \(E_{10}=-1\). Suppose we find a
  curve \(A\subseteq\CP^2\) whose proper transform \(\dbtilde{A}\)
  satisfies \(\dbtilde{A}\cdot E_i=\varphi_i\). Then
  \(C\cdot \til{A}=\frac{k}{10}\) where
  \(k=\sum_{i=1}^{10}i\varphi_i\), so if we take \(D=A/d\) then
  we get \(C\cdot\til{D}=k/(10d)\) and
  \[\varepsilon(\CP^2,D;g)=\min(10 C\cdot
    \til{D},1/C\cdot\til{D})= \min(k/d,10d/k).\] If
  \(k/d<\sqrt{10}\) then the weighted Seshadri constant equals
  \(k/d\) and we get an ellipsoid \(\Ell(k/d,k/(10d))\) with
  packing ratio \(k^2/(10d^2)\). If \(k/d\geq \sqrt{10}\) then
  the weighted Seshadri constant equals \(10d/k\) and we get an
  ellipsoid \(\Ell(10d/k,d/k)\) with packing ratio
  \(10d^2/k^2\). Clearly we can never achieve a packing ratio of
  \(1\) because \(\sqrt{10}\) is irrational. However, we can try
  and get close if we are able to find a curve where \(k/d\) is
  a good
  rational approximation to \(\sqrt{10}\). The continued
  fraction expansion of \(\sqrt{10}\) and the first
  few convergents are:
  \[\sqrt{10}=3+\frac{1}{6+\frac{1}{6+\cdots}}:\qquad\frac{19}{6},\qquad \frac{117}{37},\qquad
    \frac{721}{228},\qquad \cdots \] Let \(A\) be a rational
  sextic curve with an \(A_{18}\) singularity and a node: such
  curves exist, see for example {\cite[p.223, Table 2 (cont.):
    first column, penultimate row]{YangSextic}} or {\cite[\S
    1.2]{OrevkovSextic}} for an explicit rational
  parametrisation. Blow-up the \(A_{18}\) singularity: the
  proper transform of \(A\) has an \(A_{16}\) singularity on the
  exceptional curve. Continue blowing up the singular point of
  the proper transform until we have blown up nine times. The
  two branches of the germ of \(A\) separate at this point and
  intersect \(E_9\) transversely. Blow up one of these two
  intersection points. The result is precisely a chain
  \(E_1,\ldots,E_{10}\) and
  \[\varphi_i=\dbtilde{A}\cdot E_i=\begin{cases}1&\text{ if
    }i=9,10\\ 0 &\text{ otherwise.}\end{cases}\]
  This gives \(k=19\) and \(d=6\), so we find ellipsoids
  \(\Ell(a,b)\) for any \(a<10d/k=60/19\) and
  \(b<d/k=6/19\). Although we are in the potentially obstructive
  regime, the curve \(\dbtilde{A}\) is a {\em nodal} rational
  curve of square \(-1\) which has negative virtual dimension
  and hence has vanishing Gromov--Witten invariant and does not
  provide an obstruction to enlarging the ellipsoid.
\end{example}

\begin{remark}
  It seems that to find optimal ellipsoids using this technique,
  one would need to find a whole sequence of curves \(A\) of
  higher and higher degrees so that \(k/d\to\sqrt{10}\). This
  has the same flavour as finding a sequence of rays in the
  ample cone which converge to the Nagata ray.
\end{remark}

\subsection{Ellipsoids in ellipsoids}
\label{sct:examples_ellipsoids}
The final example we give is motivated by a question asked to us
by Kyler Siegel, namely which ellipsoids pack into other
ellipsoids? Let \(2\leq k<\ell\) be relatively prime positive
integers. By making a symplectic cut along its boundary, one can
compactify the ellipsoid \(\Ell(k,\ell)\) inside the weighted
projective plane \(\PP(k,\ell,1)\); let \([x:y:z]\) be weighted
homogeneous coordinates on \(\PP(k,\ell,1)\) and \(T=\{z=0\}\) be
the compactifying divisor, so that
\(\Ell(k,\ell)=\PP(k,\ell,1)\setminus T\). We will produce ellipsoid
embeddings in ellipsoids by first embedding them into weighted
projective planes and then observing that the resulting
embedding can be disjoined from \(T\) by an isotopy.

\begin{theorem}
  Let \(\beta,\alpha\) be coprime positive integers with
  \(s=\beta/\alpha>\frac{1}{2}(k\ell-2+\sqrt{k^2\ell^2-4k\ell})\). For
  any \(\epsilon<\frac{k\ell\alpha\beta}{\alpha+\beta}\) there is a
  symplectic embedding
  \(\Ell(\epsilon/\alpha,\epsilon/\beta)\to \Ell(k,\ell)\). 
\end{theorem}
\begin{proof}
  Take weighted homogeneous coordinates \([x:y:z]\) on
  \(Z\coloneqq \PP(k,\ell,1)\) and consider the curve
  \(D=\{xyz^{k\ell-k-\ell}=x^\ell+y^k\}\) of weighted degree \(k\ell\). This
  is an ample Cartier divisor with \(D^2 = k\ell\) (note that \(D\)
  is numerically equivalent to \(k\ell T\) and \(T^2 = 1/(k\ell)\)). In
  the smooth affine chart \(z=1\), \(D\) is locally analytically
  equivalent to \(xy=0\) (or \(xy=y^2\) if \(k=2\)), so it has a
  nodal singularity. Taking a Farey process with weights
  \((\alpha,\beta)\) guided by a branch of \(D\) we obtain a
  weighted blow-up \(g\colon Y\to Z=\PP(k,\ell,1)\) such that the
  proper transform \(\til{D}\) intersects the left- and
  right-most curves \(E_1\) and \(E_m\) each once transversely,
  so
  \(C\cdot\til{D} = \frac{1}{\beta} + \frac{1}{\alpha} =
  \frac{\alpha + \beta}{\alpha\beta}\). We have
  \[\mu(\PP(k,\ell,1),D;g) =
    \min\left(\alpha + \beta, \frac{k\ell\alpha\beta}{\alpha +
        \beta}\right).\] We are in the potentially obstructive
  regime when
  \(\frac{k\ell\alpha\beta}{\alpha + \beta}<\alpha+\beta\), which
  is equivalent to
  \[s\coloneqq
    \beta/\alpha>\frac{1}{2}(k\ell-2+\sqrt{k^2\ell^2-4k\ell}).\] In this
  case, the weighted Seshadri constant is
  \[\epsilon(\PP(k,\ell,1),D;g) = \mu(\PP(k,\ell,1),D;g)=
    \frac{k\ell\alpha\beta}{\alpha + \beta}=\frac{k\ell s\alpha}{1 +
      s}\] so Theorem \ref{thm:ellipsoids_galore} should give us
  the ellipsoids we seek. However, there is a minor issue here:
  in Theorem \ref{thm:ellipsoids_galore}, we assumed \(Z\) to be
  smooth, but now \(Z\) is only an orbifold, and \(D\) is only
  ample (not orbi-ample). We can fix this by instead using the
  orbi-ample divisor \(ND+T\) for sufficiently large \(N\); this
  works because \(T\) is locally ample. The only other place
  where smoothness of \(Z\) is needed is in the proof of
  Corollary \ref{cor:ellipsoids_2}, where we used Moser's
  argument to show that \((Z,\zeta)\) (the symplectic manifold
  we started with) is symplectomorphic to \((Z,\omega)\) (the
  symplectic manifold obtained by weighted blow-down). Provided
  we work with orbifold symplectic forms, nondegenerate along
  the orbifold locus, Moser's argument works just as well for
  orbifolds: one simply constructs the orbifold diffeomorphism
  by working \(\Gamma\)-invariantly in the local uniformising
  cover. Since the resulting Moser-diffeomorphisms preserve the
  orbifold locus, the ellipsoids we construct are disjoint from
  the orbifold locus.

  The ellipsoids we produce in \((Z,\omega)\) are manifestly
  disjoint from \(T\), and the arguments of {\cite[Remark 2.1.E
    and Corollary 4.1.B]{McDuffPolterovich}} apply in the
  orbifold setting, using the symplectic tubular neighbourhood
  theorem for suborbifolds {\cite[Proposition 17]{MunozRojo}},
  to show that we can modify the Moser isotopy to ensure that
  the ellipsoids in \((Z,\zeta)\) remain in the complement of
  \(T\). Therefore we produce symplectic embeddings
  \(\Ell(\epsilon/\alpha,\epsilon/\beta)\to\Ell(k,\ell)\).
\end{proof}

\begin{remark}
  As \(\epsilon\) approaches the Seshadri constant, the packing
  ratio approaches \(\frac{k\ell}{s+s^{-1}+2}\). As
  \(s\to \frac{1}{2}(k\ell-2+\sqrt{k^2\ell^2-4k\ell})\), this approaches
  \(1\), so we get fuller and fuller fillings. Note that the
  slopes of these fillings are substantially larger than
  \(\ell/k\), the slope of \(\Ell(k,\ell)\).
\end{remark}
\begin{remark}
  If we take \(k=1\) then the proof breaks down: the curve \(D\)
  is no longer nodal. Instead, if \(\ell>1\), one should look at
  \(A=\{xyz^{\ell-1}=x^\ell+y^2\}\) of weighted degree \(2\ell\) and take
  \(D=\frac{1}{2}A\). The conclusion is modified accordingly:
  for all slopes \(s>2\ell-1+4\sqrt{\ell^2-\ell}\) we get an embedding of
  \(\Ell(\epsilon/\beta,\epsilon/\alpha)\) into \(\Ell(1,\ell)\)
  whenever \(\epsilon<\frac{2\ell\alpha\beta}{\alpha+\beta}\). If
  \(k=1\) and \(\ell=1\) then we need to pass to cubics and we are
  back in the situation of Proposition \ref{prp:first_step}.
\end{remark}

\begin{remark}
  Clearly this is just one family of ellipsoid embeddings,
  analogous to the first post-Fibonacci step of the
  McDuff--Schlenk staircase explained in Proposition
  \ref{prp:first_step}. The full story of ellipsoids in
  ellipsoids remains to be worked out.
\end{remark}

\bibliographystyle{plain} \bibliography{weighted_seshadri}
\Addresses
\end{document}